\documentclass[11pt,a4paper]{amsart}
\usepackage[utf8]{inputenc}
\usepackage{amsmath}
\usepackage{amsfonts}
\usepackage{amscd}
\usepackage{enumerate}
\usepackage[margin=1in]{geometry}
\usepackage{amsthm}
\usepackage{mathrsfs}
\usepackage{amssymb}
\usepackage[all]{xy}
\usepackage{xcolor}
\usepackage{graphics}
\usepackage{lscape}
\usepackage{array}
\usepackage{mathabx}
\usepackage{xy}

\numberwithin{equation}{section}
\newtheorem*{theorem*}{Theorem}
\newtheorem{theorem}{Theorem}[section]
\newtheorem{lemma}[theorem]{Lemma}
\newtheorem{proposition}[theorem]{Proposition}
\newtheorem{corollary}[theorem]{Corollary}
\theoremstyle{definition}
\newtheorem{definition}[theorem]{Definition}
\newtheorem{assumption}{Assumption}

\theoremstyle{remark}
\newtheorem{remark}[theorem]{Remark}
\newtheorem{notation}[theorem]{Notation}

\usepackage{hyperref}\hypersetup{colorlinks}

\usepackage{color} 

\definecolor{darkred}{rgb}{1,0,0} 
\definecolor{darkgreen}{rgb}{0,1,0}
\definecolor{darkblue}{rgb}{0,0,1}

\hypersetup{colorlinks,
linkcolor=darkblue,
filecolor=darkgreen,
urlcolor=darkred,
citecolor=darkgreen}

\DeclareMathOperator{\Jac}{Jac}

\newcommand{\Nm}{\mathrm{Nm}}

\newcommand{\BBB}{\mathrm{(BBB)}}
\newcommand{\BAA}{\mathrm{(BAA)}}
\newcommand{\ABA}{\mathrm{(ABA)}}
\newcommand{\AAB}{\mathrm{(AAB)}}
\newcommand{\BBBB}{\mathbf{(BBB)}}
\newcommand{\BBAA}{\mathbf{(BAA)}}

\newcommand{\Id}{\mathrm{Id}}

\newcommand{\sing}{\mathrm{sing}}
\newcommand{\Sing}{\mathrm{Sing}}

\newcommand{\Hilb}{\mathrm{Hilb}}
\newcommand{\supp}{\mathrm{supp}}
\newcommand{\End}{\mathrm{End}}

\newcommand{\gr}{\mathrm{gr}}
\newcommand{\id}{\mathbf{1}}

\newcommand{\DR}{\mathrm{dR}}

\newcommand{\GL}{\mathrm{GL}}
\newcommand{\SL}{\mathrm{SL}}
\newcommand{\PSL}{\mathrm{PSL}}

\newcommand{\suchthat}{\;|\;}
\newcommand{\tr}{\mathrm{tr}}

\newcommand{\U}{\mathrm{U}}

\newcommand{\image}{\mathrm{Im}}

\newcommand{\red}{\mathrm{red}}
\newcommand{\sm}{\mathrm{\, sm}}
\newcommand{\free}{\mathrm{free}}
\newcommand{\fs}{\mathrm{\, it}}
\renewcommand{\ni}{\mathrm{\, ni}}

\newcommand{\Hec}{\mathrm{NR} \,}
\newcommand{\PMod}{\mathrm{PMod}}

\newcommand{\Ord}{\mathrm{Ord}}

\newcommand{\Dd}{\mathcal{D}}
\newcommand{\Ee}{\mathcal{E}}
\newcommand{\Ff}{\mathcal{F}}

\newcommand{\Ii}{\mathcal{I}}

\newcommand{\Jj}{\mathcal{J}}

\newcommand{\Ll}{\mathcal{L}}

\newcommand{\Oo}{\mathcal{O}}
\newcommand{\Pp}{\mathcal{P}}

\newcommand{\Uu}{\mathcal{U}}
\newcommand{\Vv}{\mathcal{V}}
\newcommand{\Ww}{\mathcal{W}}

\newcommand{\M}{\mathbf{M}}

\newcommand{\B}{\mathrm{B}}

\newcommand{\Uni}{\mathrm{Uni}}

\renewcommand{\L}{\mathrm{L}}
\renewcommand{\P}{\mathrm{P}}

\renewcommand{\i}{\mathrm{i}}

\newcommand{\ol}[1]{\overline{#1}}

\newcommand{\wt}[1]{\widetilde{#1}}

\newcommand{\Fff}{\mathscr{F}}
\newcommand{\Lll}{\mathscr{L}}

\newcommand{\Xxx}{\mathscr{X}}

\newcommand{\CC}{\mathbb{C}}

\newcommand{\ZZ}{\mathbb{Z}}

\newcommand{\HH}{\mathbb{H}}

\newcommand{\PP}{\mathbb{P}}

\newcommand{\hH}{\mathfrak{h}}
\newcommand{\tT}{\mathfrak{t}}

\renewcommand{\to}{\longrightarrow}

\newcommand{\plonge}{\hookrightarrow}

\newcommand\Quotient[2]{
\mathchoice
{
\text{\raise1ex\hbox{\thinspace $#1$}\Big{/} \lower1ex\hbox{$#2$} \thinspace}%
}
{
#1\,/\,#2
}
{
#1\,/\,#2
}
{
#1\,/\,#2
}
}

\newcommand\GIT[2]{
\mathchoice
{
\text{\raise1ex\hbox{\thinspace $#1$}\Big{/}\!\!\!\!\Big{/} \lower1ex\hbox{$#2$} \thinspace}%
}
{
#1\,/\,#2
}
{
#1\,/\,#2
}
{
#1\,/\,#2
a       }
}

\newcommand{\map}[5]{\begin{array}{ccc}   #1  & \stackrel{#5}{\longrightarrow} &  #2  \\  #3 & \longmapsto & #4  \end{array}}

\newcommand{\morph}[6]{\begin{array}{cccc} #6: & #1  & \stackrel{#5}{\longrightarrow} &  #2  \\ & #3 & \longmapsto & #4  \end{array}}

\title[Unramified covers and branes on the Hitchin system]{\bf Unramified covers and branes on the Hitchin system}

\author[E. Franco]{Emilio Franco}
\address{E. Franco,
\newline\indent Centro de An\'alise Matem\'atica, Geometria e Sistemas Din\^{a}micos, 
\newline\indent Instituto Superior T\'ecnico, Universidade de Lisboa, 
\newline\indent Av. Rovisco Pais s/n, 1049-001 Lisboa, Portugal}
\email{emilio.franco@tecnico.ulisboa.pt}

\author[P.~B.\ Gothen]{Peter B.\ Gothen}
\address{P.~B. Gothen,
\newline\indent Centro de Matemática da Universidade do Porto, 
\newline\indent Faculdade de Ci\^encias da Universidade do Porto, 
\newline\indent Rua do Campo Alegre s/n, 4169-007 Porto, Portugal}
\email{pbgothen@fc.up.pt}

\author[A. Oliveira]{Andr\'e Oliveira}
\address{A. Oliveira, 
\newline\indent Centro de Matem\'atica da Universidade do Porto,
\newline\indent Faculdade de Ci\^encias da Universidade do Porto, 
\newline\indent
Rua do Campo Alegre s/n, 4169-007 Porto, Portugal \newline\indent \textsl{On leave from:}\newline\indent Departamento de Matem\'atica, Universidade de Tr\'as-os-Montes e Alto Douro, UTAD,\newline\indent
Quinta dos Prados, 5000-911 Vila Real, Portugal}
\email{andre.oliveira@fc.up.pt\newline\indent agoliv@utad.pt}

\author{Ana Pe\'on-Nieto}
\address{Ana Pe\'on-Nieto	\newline\indent
	School of Mathematics \newline\indent University of Birmingham\newline\indent Watson Building, Edgebaston\newline\indent Birmingham B15 2TT, UK\newline \indent \textit{and}
 \newline\indent Laboratoire de Math\'ematiques J.A. Dieudonn\'e \newline\indent UMR  7351 CNRS \newline\indent Universit\'e de Nice Sophia-Antipolis \newline\indent 06108 Nice Cedex 02, France}
\email{ana.peon-nieto@unice.fr, 
	a.peon-nieto@bham.ac.uk}

\date{\today}

\thanks{
First, second and third authors partially supported by CMUP
(UIDB/00144/2020) and the project PTDC/MAT-GEO/2823/2014 funded by FCT
(Portugal) with national funds. 
First author supported as well by FCT (Portugal) in the framework of the Investigador FCT program, fellowship reference CEECIND/04153/2017. 
Third author also partially supported by the Post-Doctoral fellowship SFRH/BPD/100996/2014, also funded by FCT (Portugal) with
national funds.
Fourth author is currently supported by the scheme H2020-MSCA-IF-2019,  Agreement n.	897722 (GoH). She was formerly funded through a Beatriu de Pin\'os grant n. 2018 BP 332 (H2020-MSCA-COFUND-2017 Agreement n. 801370), a postdoctoral grant associated to the project FP7 - PEOPLE - 2013 - CIG - GEOMODULI number: 618471 and the 
Swiss National Science Foundation project SwissMAP. 
The authors
acknowledge support from U.S. National Science Foundation grants
DMS 1107452, 1107263, 1107367 ``RNMS: GEometric structures And
Representation varieties'' (the GEAR Network).
}


\begin{document}

\begin{abstract}
We study the locus of the moduli space of $\mathrm{GL}(n,\mathbb{C})$-Higgs bundles on a curve given by those Higgs bundles obtained by pushforward under a connected unramified cover. We equip these loci
with a hyperholomorphic bundle so that they can be viewed as
BBB-branes, and we introduce corresponding BAA-branes which can be
described via Hecke modifications. We then show how these branes are
naturally dual via explicit Fourier--Mukai transform (recall that $\mathrm{GL}(n,\mathbb{C})$ is Langlands self dual). It is
noteworthy that these branes lie over the singular locus of the
Hitchin fibration. 

As a particular case, our construction describes the behavior under mirror symmetry of the fixed loci for the action of tensorization by a line bundle of order $n$. These loci play a key role in the work of Hausel and Thaddeus on topological mirror symmetry for Higgs moduli spaces.
\end{abstract}

\maketitle

\tableofcontents


\section{Introduction}

\subsection{In a nutshell}
Among the many fundamental contributions of Narasimhan and Ramanan to
the study of moduli of vector bundles on curves are the Hecke
correspondence \cite{narasimhan-ramanan:1969,narasimhan-ramanan:1975b} and the study of generalized Prym varieties as fixed
points \cite{NR}. In this paper we use these ideas to explore mirror symmetry
for the moduli space of Higgs bundles in the spirit of the seminal
work of A. Kapustin and E. Witten \cite{kapustin&witten}. More precisely, we exhibit pairs of dual
branes for the Langlands self dual group $\GL(n,\CC)$.
The interest of our construction relies on two aspects: firstly the branes we consider are sheaves (rather than just submanifolds) and the duality is realized via an explicit Fourier--Mukai
transform; secondly, we are making progress in the understanding of mirror symmetry in the
singular locus of the Hitchin system, since the branes lie entirely over this locus. As far as we know, this is the first example of dual branes lying over the singular locus, where mirror symmetry is explicitly realized by a Fourier--Mukai transform. Finally, it is also important to note that among the branes we
construct are the fixed loci under tensorization by an order $n$ line
bundle, central in the work of Hausel and Thaddeus \cite{HT} on
topological mirror symmetry, so our construction ought to be important in a deeper understanding of the topological mirror symmetry phenomenon. In the remainder of this section, we
explain our constructions and results in more detail.

\subsection{Context}
 N. Hitchin introduced in \cite{hitchin-self} Higgs
bundles over a smooth projective complex curve $X$ of genus $g\geq 2$ as solutions to certain equations obtained by dimensional reduction of the self-dual equations on a
$4$-manifold. 
These are pairs $(E,\varphi)$, where $E$ is a holomorphic vector
bundle over $X$ and $\varphi$ is a holomorphic one-form with values in
$\End(E)$.  The moduli space  $\M_X(n,d)$ of Higgs bundles of rank $n$
and degree $d$ is a holomorphic symplectic manifold carrying a
hyperk\"ahler metric. Moreover, it admits the structure of an
algebraically completely integrable system given by the Hitchin map
$h_{X,n}\colon\M_X(n,d)\to B_{X,n}$. Here the Hitchin base $B_{X,n}$
is an affine space whose dimension is half that of $\M_X(n,d)$, and
the components of $h_{X,n}$ are the coefficients of the characteristic
polynomial of $\varphi$. The fiber of $h_{X,n}$ over a generic point
of the Hitchin base is a torsor for an abelian variety, namely the
Jacobian of an associated spectral curve.

The concept of a $G$-Higgs bundle can be defined for any complex (and
even real) reductive Lie group $G$. In these terms, the above
definition becomes that of a $\GL(n,\CC)$-Higgs bundle.  The Hitchin
map can also be defined in this generality, and it has been shown that
it is an algebraically completely integrable system for any complex
reductive Lie group $G$ \cite{hitchin_duke, Faltings,Sco, DG}.

A new development arose with the discovery by T. Hausel and
M. Thaddeus \cite{HT} of a close relation between Higgs bundles,
mirror symmetry and the Langlands correspondence. They proved that the
moduli spaces of Higgs bundles for the group
$\SL(n,\CC)$ and its Langlands dual group $\PSL(n,\CC)$ form a pair of
SYZ-mirror partners \cite{strominger-yau-zaslow:1993}, in the sense
that the respective Hitchin maps have naturally isomorphic bases and
their fibers over corresponding points are, generically,
half-dimensional torsors for a pair of dual abelian varieties. This
was subsequently generalized by N. Hitchin \cite{hitchin_G2} for the
self-dual group $G_2$ and then by R. Donagi and T. Pantev \cite{DP}
for any pair $(G,{}^LG)$ of Langlands dual groups. The duality is
reflected by a Fourier--Mukai transform between the
moduli spaces interchanging fibers of the Hitchin map over corresponding
points in the base. These dualities were obtained over the locus of
the Hitchin base where the corresponding spectral curves are smooth.

As mentioned above, the moduli space $\M_X(n,d)$ is
hyperk\"ahler. This means that it carries three natural complex
structures $I_1$, $I_2$ and $I_3$ verifying the quaternionic relations
and a metric which is K\"ahler with respect to all three holomorphic
structures. In the present case, $I_1$ is the natural complex
structure on the moduli space of Higgs bundles $\M_X(n,d)$, while the
complex structures $I_2$ and $I_3=I_1I_2$ arise via the non-abelian
Hodge Theorem, which identifies $\M_X(n,d)$ with the moduli space of
projectively flat $\GL(n,\CC)$-connections (see \cite{hitchin-self,Si}).

A. Kapustin and E. Witten considered in \cite{kapustin&witten} certain
special subvarieties of $\M_X(n,d)$, equipped with special
sheaves. The pair composed by such a subvariety and the corresponding
sheaf is called a brane.  For each of the complex structures on
$\M_X(n,d)$ a brane is classified as follows: it is of type A if it is
a Lagrangian subvariety with respect to the corresponding K\"ahler
form and the sheaf over it is equipped with a flat connection, and it is of type B if it
is a holomorphic subvariety and the sheaf over it is also
holomorphic. Thus, for instance, a $\BBB$-brane is a subvariety
equipped with a sheaf, holomorphic with respect to all three
complex structures $I_1$, $I_2$ and $I_3$; in other words, it is a
hyperholomorphic subvariety equipped with a hyperholomorphic
sheaf (this is a sheaf with a connection whose curvature is of type $(1,1)$ with respect to all complex structures). 
A $\BAA$-brane is a subvariety which is holomorphic
with respect to $I_1$, and Lagrangian with respect to the K\"ahler forms $\omega_2$ and $\omega_3$ associated to $I_2$ and $I_3$ (hence complex Lagrangian for $\Omega_1 = \omega_2 + \i \, \omega_3$), and which in addition supports a flat vector bundle.
There are only two other possible types of branes on $\M_X(n,d)$, namely
$\ABA$- and $\AAB$-branes.  Again all this holds for any complex Lie
group and not just $\GL(n,\CC)$.

According to \cite{kapustin&witten}, mirror symmetry conjecturally
interchanges $\BBB$-branes and $\BAA$-branes, and mathematically this
duality should again be realised via a Fourier--Mukai transform (in
complex structure $I_1$). The support of the $\BAA$-brane should
depend not only on the support of the dual $\BBB$-brane but also on
the hyperholomophic sheaf over it (and vice-versa). A similar story
holds for pairs of $\ABA$-branes and also for pairs of $\AAB$-branes.

Since Kapustin and Witten's paper---and because of it---an intense study
of several kinds of branes on Higgs bundle moduli spaces has been
carried out. Some examples may be found in
\cite{hitchin-charclasses,baraglia-schaposnik-realstruct,BGP,HS,BCFG,hitchin-spinors,Ga,FJ,baraglia-schaposnik-CayleyLangl,Borel,B,HMDP}
(see also \cite{AFES} for a survey on this subject). Most of these
works mainly focus either on the smooth locus of the Hitchin system
(exceptions are \cite{baraglia-schaposnik-CayleyLangl,Borel,B}) or
only deal with the support of the branes and not with the sheaves on
it (exceptions are \cite{hitchin-charclasses,
hitchin-spinors,Ga,FJ,Borel}).

\subsection{Our construction}

Starting from a connected unramified cover $p:C \to X$ of degree $n$ and Galois group $\Gamma$, we introduce in this paper new types of $\BBB$-branes and $\BAA$-branes on $\M_X(n,d)$, the moduli space for the self-dual group
$\GL(n,\CC)$. As required in
the general picture, our $\BBB$-branes come equipped with
flat, hence hyperholomorphic, bundles. We 
explicitly prove (when $d=0$) that their (fiberwise) Fourier--Mukai transform generically yields a sheaf 
supported exactly over the support of our $\BAA$-brane. As expected, the support of the $\BAA$-brane
depends on the hyperholomorphic bundle over the $\BBB$-brane.

These branes are supported on a subspace $B^p\subset B_{X,n}$ of
the \emph{singular locus} of the Hitchin system. For a dense open subset $B^p_{\ni} \subset B^p$ of {\it nodal} and {\it integral} spectral curves, the normalization of these curves is $C$ itself. Since $p : C \to X$ is unramified, $B^p_{\ni}$ is, by definition, contained in the so-called \emph{endoscopic locus} of $h_{X,n}$ (cf. \cite{HP,Ngo}). So our construction (more precisely, its analogue for the Langlands dual groups $\SL(n,\CC)$ and $\PSL(n,\CC)$) may eventually be relevant in the context of geometric endoscopy, introduced by E. Frenkel and E. Witten in \cite{FW}.

In the following we outline our construction in more detail, starting
with the $\BBB$-branes. Fix the rank $n$ to coincide with the degree of $p$ and set $\M_X(n,d)^p$ to be the locus of Higgs bundles obtained as a pushforward under $p$ of Higgs bundles in $\M_C(1,d) \cong T^*\Jac^d({C})$. Let $B^p$ be the image of $\M_X(n,d)^p$
under the Hitchin map $h_{X,n}\colon\M_X(n,d)\to B_{X,n}$. As a direct consequence of non-abelian Hodge theory, one concludes that $\M_X(n,d)^p$ is a hyperholomorphic subvariety. The pushforward by $p$
yields an isomorphism between $\M_X(n,d)^p$ and the quotient of
$T^*\Jac^d({C})$ by the Galois group, acting by
pullback. From this, one defines a hyperholomorphic line bundle $\Lll$
over $\M_X(n,d)^p$, naturally associated to a flat line bundle
$\Ll$ on $X$. We call the pair $(\M_X(n,d)^p,\Lll)$ a \emph{rank $1$
Narasimhan--Ramanan $\BBB$-brane}. We represent it by $\BBBB_\Ll^p$ and write
$\BBBB_{\ni}^{p,\Ll}$ for its restriction to $B^p_{\ni}$. More generally,
we can construct a rank $n$ coherent and hyperholomorphic sheaf $\Fff$ on $\M_X(n,d)^p$ which
canonically associated to a flat line bundle $\Ff$ over ${C}$,
and we call the pair $(\M_X(n,d)^p,\Fff)$ a \emph{rank $n$
Narasimhan--Ramanan $\BBB$-brane} and represent it by
$\BBBB_{\Ff}^p$. 

Let $p: C \to X$ be a Galois $\ZZ_n$-cover, and let $\xi\in \ZZ_n$ be the standard generator. Parallel transport of the lifts from $X$ to $C$ provides a line bundle $L_\xi \in \Jac^0(X)$ of order $n$. In this case, it basically follows from \cite{NR} that the locus $\M_X(n,d)^p$ coincides with the subvariety $\M_X(n,d)^\xi \subset \M_X(n,d)$ of points $(E,\varphi)$ fixed by tensorization of by $L_\xi$, \textit{i.e.}\  $(E,\varphi)\cong(E\otimes L_\xi,\varphi)$. The study of $\M_X(n,d)^\xi$ was our original motivation. So this justifies the name chosen for the $\BBB$-branes appearing in this paper.



If our $\BBB$-branes are intimately related to the work of Narasimhan--Ramanan in \cite{NR}, the construction of our $\BAA$-branes is closely linked to their work on \emph{Hecke modifications} of vector bundles published in \cite{narasimhan-ramanan:1969,narasimhan-ramanan:1975b}. Hecke modifications in the context of Higgs
bundles have previously appeared in several papers; see, for example,
\cite{hitchin:2017,hwang-ramanan:2004,ramanan:2010,wilkin:2016,witten:2015}. Before describing the construction, we recall that under certain assumptions on the values of the rank and the degree, there exists a Hitchin section on the moduli space $\M_C(r,d+\delta)$ constructed out of a line bundle $\Jj\in\Jac^{d+\delta}({C})$. The pushforward under $p$ defines a Hitchin--type section of $\M_X(n, d+\delta)^p \to B^p$. We define the subvariety $\Hec^{p,\Jj}_{\ni} \subset \M_X(n,d)$ of those Higgs bundles over $B^p_{\ni}$ obtained as Hecke modifications of this Hitchin--type section at the divisor of singularities of the corresponding integral and nodal spectral curve (which has length $\delta$) classified by $B^p_{\ni}$. The notation we use for this subvarieties is chosen to recognize the pioneer work Narasimhan and Ramanan on Hecke modifications. We prove next that the subvarieties $\Hec^{p,\Jj}_{\ni}$ are complex Lagrangian with respect the holomorphic symplectic form $\Omega_{1}=\omega_{2}+\i \, \omega_{3}$ on $\M_X(n,d)$. This shows that this subvariety is the support of a $\BAA$-brane on $\M_X(n,d)$, when endowed with a flat bundle.

Our construction of $\Hec^{p,\Jj}_{\ni}$ (for $d=0$ and $p$ of degree $n$) was aimed at obtaining the support of a $\BAA$-brane dual to the rank $1$ $\BBB$-brane $\BBBB_\Ll^p$, for an appropriate choice of the line bundle $\Jj$. Towards this goal, we provide an extensive study of the spectral data of the Higgs bundles appearing in $\Hec^{p,\Jj}_{\ni}$ and in $\M_X(n,d)^p$, the support of $\BBBB_\Ll^p$. For a given $b\in B^p_{\ni}$, let $X_b$ be the corresponding spectral curve and $\nu_b:{C}\to X_b$ the normalization. Over the Hitchin fiber associated to $X_b$, the spectral data in $\Hec^{p,\Jj}_{\ni}$ are those contained in the closure of the preimage of $\Jj$ by the pull-back under $\nu_b$. On the other hand, the spectral data contained in $\M_X(n,d)^p$ are those given by pushforward under $\nu_b$. This paves the way for our main result, Theorem \ref{tm duality}, which is described below.

\begin{theorem*}
Let $p : C \to X$ be a connected unramified $n$-cover. Consider the moduli space $\M_X(n,0)$.
\begin{enumerate}[(i)]
\item Let $\Jj={p}^*(\Ll\otimes K_X^{(n-1)/2})$. The (fiberwisewise) dual of the rank $1$ Narasimhan-Ramanan $\BBB$-brane $\BBBB^p_\Ll$ (restricted to the locus of nodal and irreducible spectral curves) is the $\BAA$-brane supported on $\Hec^{p,\Jj}_{\ni}$, and whose flat bundle satisfies \eqref{eq:relation FM transf 1}.
\item Let $\Jj=\Ff\otimes {p}^*K_X^{(n-1)/2}$. The (fiberwise) dual of the rank $n$ Narasimhan-Ramanan $\BBB$-brane $\BBBB^p_{\Ff}$ (restricted to the locus of nodal and irreducible spectral curves) is the $\BAA$-brane supported on $\Hec^{p,\Jj}_{\ni}$, and whose flat bundle satisfies \eqref{eq:relation FM transf n} .
\end{enumerate}
\end{theorem*}


It is important to note that this duality is proved by an explicit fiberwise Fourier--Mukai transform, on the fibers over $B^p_{\ni}$, mapping the hyperholomorphic sheaf to a sheaf supported on $\Hec^{p,\Jj}_{\ni}$. This Fourier--Mukai transform is carried out using the autoduality of compactified Jacobians of integral curves with
planar singularities, from the general results of D. Arinkin
\cite{arinkin}. It uses a Hitchin section (which embeds $B^p_{\ni}$ as a subvariety of $\Hec^{p,\Jj}_{\ni}$) to identify $\overline{\Jac}^{\, \delta}(X_b)$ with the corresponding
$\overline{\Jac}^{\, 0}(X_b)$, and then apply Arinkin's Fourier--Mukai functor. In order to explicitly do it, 
we relate this functor with the classical Fourier--Mukai functor of $\Jac^{\, \delta}({C})$, via the pullback and the pushforward maps induced by the normalization morphism $\nu_b:{C}\to X_b$.

 

It is worth noticing in this case that $\BBBB^p_{\Ff}$ appears as the pushforward of $\BBB$-brane $(\Fff, \nabla_\Fff) \to \M_C(1,0)$ supported over the whole moduli space, where $\Fff$ is the pullback under $\M_C(1, 0) \to \Jac^0(X)$ of the flat line bundle over $\Jac^0(X)$ associated to $\Ff \to X$. Mirror symmetry conjectures that $(\Fff, \nabla_\Fff) \to \M_C(1,0)$ is dual to the $\BAA$-brane given by the Hitchin section associated to $\Ff$. As we said before, $\BBAA^{p,\Ff}_{\ni}$ can be interpreted in terms of Hecke modifications of the pushforward of this Hitchin section. This suggests a deep relation between duality of branes in $\M_X(n,0)$, duality in $\M_C(1,0)$ and the Hecke operators appearing in geometric Langlands conjecture (see \cite{DP}). 	For $d$ non-multiple of $n$ a similar result should hold, but the duality
should require a gerbe to work out properly.  We also note that the
results in this paper provide evidence for the dualities suggested in
\cite{Borel}.

\begin{remark}
We actually construct the support of the $\BBB$-brane (and describe its spectral data) in a wider generality, namely in the case where the unramified cover $p:C\to X$ is of degree $m$ not necessarily equal to the rank $n$. In such a case, one must consider polystable Higgs bundles over $C$ of rank $r$, such that $n=mr$. It is however unclear how to endow such $\BBB$-branes with hyperholomorphic bundles.  

Similarly, we construct $\Hec^{p,\Jj}_{\ni}$ in the more general setup of a degree $m$ cover $p:C\to X$. In the absence of a Hitchin section on $\M_C(r,d)$ we make use of very stable bundles on $C$, which define natural complex Lagrangian multisections of the Hitchin fibration. We explore this in Section \ref{sec:non-max-ord}.
\end{remark}

As mentioned above, when the Galois group is cyclic, the support of our $\BBB$-branes is $\M_X(n,d)^\xi$. It is interesting to notice that $\M_X(n,d)^\xi$ plays a central role in the proof by T. Hausel and M. Thaddeus \cite{HT} of
topological mirror symmetry for the moduli spaces of Higgs bundles for
the Langlands dual groups $\SL(n,\CC)$ and $\PSL(n,\CC)$ for $n=2,3$
(the general case has recently been proved by M. Groechenig, D. Wyss and
P. Ziegler \cite{groechenig-wiss-ziegler:2017}, and, more recently, by D. Maulik and J. Shen  \cite{Maulik-Shen}). One might thus
hope that further study of our dual branes in this setting would provide a
better geometric understanding of the calculation by Hausel and
Thaddeus. We hope to come back to this question in a future article.

\subsection{Organization of the paper}
Here is a brief description of the organization of the paper. In
Section~\ref{section:prelim} we recall some background material on the
Hitchin system. In Section \ref{sc Fixed point subvarieties} we study the locus $\M_X(n,d)^p$, including the
corresponding spectral data, for $p$ an unramified cover of degree $m$, with $m$ dividing $n$. Section~\ref{sc BBB branes} deals with the construction and
description of the Narasimhan--Ramanan $\BBB$-branes. In Section~\ref{section:Heckebranes} we construct the complex Lagrangian subvarieties $\Hec^{p,\Jj}_{\ni}$, which support $\BAA$-branes.
In Section~\ref{sec:FM}, after recalling some background
facts on the Fourier--Mukai transform for compactified Jacobians of integral curves and describing in Section \ref{subsection:FM} the role of the normalization of the curve in the transform, we prove our main duality result, namely Theorem \ref{tm duality}.
Finally, in Section~\ref{sec:non-max-ord}, we generalize parts of the previous study to the
case where $p: C \to X$ has degree strictly less than $n$ and no Hitchin section exists on $\M_C(r,d)$.

\subsubsection*{Acknowledgments}

The authors thank D. Arinkin, B. Collier, O. Garcia-Prada, T. Hausel,
N. Hitchin, C. Pauly and R. Wentworth for their interest and useful discussions, and also thank the referee for helpful remarks and corrections.

\section{Higgs bundles and the Hitchin system}\label{section:prelim}

The purpose of this section is to recall the basics on Higgs bundle moduli spaces which will be used in the remaining part of the paper.

\subsection{Higgs bundles and their moduli space}
\label{sc Mm hyperkahler}

Let $X$ be a smooth projective curve over $\CC$, of genus $g\geq 2$. A
\emph{Higgs bundle} over $X$ is a pair $(E, \varphi)$ given by a holomorphic
vector bundle $E \to X$, and $
\varphi \in H^0(X, \End(E) \otimes K_X), 
$ where $K_X$ is the canonical bundle. 
The section $\varphi$ is called the \emph{Higgs field}.  The
\emph{rank} and \emph{degree} of a Higgs bundle are those of the
underlying vector bundle $E$. Such a rank $n$ Higgs bundle is also said to be a \emph{$\GL(n,\CC)$-Higgs bundle}.
Occasionally, we shall refer to \emph{$\SL(n,\CC)$-Higgs bundles}, in which $E$ is required to have a fixed given determinant bundle and $\varphi$ to be traceless.

Let $\M_X(n,d)$ denote the moduli space os $S$-equivalence classes of semistable rank $n$ and degree $d$ Higgs bundles on $X$. Its points are represented by the unique polystable representative of the corresponding $S$-equivalence class.
It is a quasi-projective variety of complex dimension
\begin{equation}\label{eq:dim}
\dim \M_X(n,d) = 2n^2(g-1) + 2.
\end{equation}
The closely related \emph{de Rham moduli space} $\M^\DR_X(n,d)$ is
the moduli space of connections  with constant central
curvature on a fixed $C^\infty$ vector bundle
over $X$ of rank $n$ and degree $d$.
Non-abelian Hodge theory \cite{hitchin-self, simpson2, simpson3,
donaldson, corlette} establishes the existence of a homeomorphism
between
these spaces, $\M_{X}(n,d)\cong\M^\DR_X(n,d)$.
This homeomorphism restricts to a diffeomorphism on the smooth locus of
$\M_{X}(n,d)$, whose underlying manifold is a hyperk\"ahler manifold \cite{hitchin-self}
with complex structures
\[
I_1, \quad I_2 \quad \text{and} \quad I_3 =I_1 I_2.
\]
Here $I_1$ is the complex structure coming from $\M_{X}(n,d)$ and
$I_2$ is the one coming from $\M^\DR_{X}(n,d)$. Let also $\omega_j$ be the K\"ahler form associated to $I_j$ and
$\Omega_{X,j} = \omega_{j+1} + \i \omega_{j-1}$ the corresponding
holomorphic symplectic form.

\subsection{The Hitchin system}\label{sec:the Hitchin system}

We recall here the spectral construction given in \cite{hitchin_duke, BNR}. Let $(P_1,\ldots,P_n)$ be a basis of $\GL(n,\CC)$-invariant polynomials with $\deg(P_i)=i$; for instance, we could take $P_i(x)=(-1)^i\tr(\wedge^i x)$. The \emph{Hitchin map} or \emph{Hitchin fibration} is 
\begin{equation}\label{eq:Hitchinfibration}
\morph{\M_X(n,d)}{B_{X,n} := \bigoplus_{i=1}^n H^0(X,K_X^i)}{(E,\varphi)}
{\left (P_1(\varphi),\dots, P_n(\varphi) \right ).}{}{h_{X,n}}
\end{equation}
Note that $\dim(B_{X,n})=n^2(g-1)+1=\frac{1}{2}\dim(\M_X(n,d))$.

Consider the total space $|K_X|$ of the canonical bundle and the
surjective morphism $\pi: |K_X| \to X$. The pullback bundle
$\pi^*K_X \to |K_X|$ has a tautological section
$\lambda$. Given an element $b \in B_{X,n}$, with
$b = (b_1, \dots, b_n)$, the \emph{spectral curve}
$X_b \subset |K_X|$ is the vanishing locus of the section of
$\pi^*K_X^n$ given by
\begin{equation} \label{eq construction of spectral curve}
\lambda^n+\pi^*b_1\lambda^{n-1}+\dots+\pi^*b_{n-1}\lambda + \pi^*b_n\in H^0(|K_X|,\pi^*K_X^n).
\end{equation}
The restriction of $\pi\colon |K_X| \to X$ to $X_b$ yields a ramified degree $n$ cover denoted by
\begin{equation} \label{eq pi_b}
\pi_b\colon X_b \longrightarrow X.
\end{equation}
For generic $b$, the spectral curve $X_b$ is smooth.
For any $b$, the (arithmetic) genus of $X_b$ is \cite{hitchin_duke}
\[
g(X_b) = n^2(g-1)+1.
\]
Additionally, $\pi_{b,\ast}\Oo_{X_b}\cong\Oo_X\oplus K_X^{-1}\oplus\cdots\oplus K_X^{1-n}$, thus
\[
\deg(\pi_{b,\ast}\Oo_{X_b}) = - n(n - 1)(g-1).
\]

\begin{notation}
For the remainder of the paper, let us denote the degree of the ramification divisor of the spectral curve $X_b \to X$ in $B_{X,n}$, by 
\begin{equation}\label{eq:delta}
\delta:=n(n-1)(g-1).
\end{equation}
\end{notation}

Given a rank $1$ torsion-free sheaf $\Ff$ over $X_b$ of degree
$d+\delta$, we have that 
\begin{equation} \label{eq spectral correspondence for Ff}
E_{\Ff} := \pi_*\Ff
\end{equation}
is a vector bundle on $X$
of rank $n$ and degree $d$. Tensorization by the tautological section yields 
\[\mu_{\Ff} : \Ff \xrightarrow{\otimes\lambda} \Ff \otimes \pi_b^* K_X.\] 
Since $\pi$ is an affine morphism, $\mu_{\Ff}$ corresponds to a Higgs field 
\begin{equation} \label{eq spectral correspondence for varphi}
\varphi_{\Ff} := \pi_{b,*}\mu_{\Ff}:E_\Ff\to E_\Ff\otimes K_X
\end{equation}
on $E_{\Ff}$ with characteristic polynomial determined by $b \in B_{X,n}$
\cite{BNR,Schaub,simpson3}. The pair $(X_b,\Ff)$ is said to be the
\emph{spectral datum} of the Higgs bundle $(E_{\Ff},\varphi_{\Ff})$. This
establishes a one-to-one correspondence, sometimes called \emph{spectral correspondence}, between the \emph{Hitchin
fiber} $h_{X,n}^{-1}(b)$ and the moduli space of
rank $1$ torsion-free sheaves on $X_b$ of degree $d+\delta$
and linearization naturally induced from the base \cite[Corollary 6.9]{simpson3},
and so
\begin{equation} \label{eq description of Hitchin fiber}
h^{-1}_{X,n}(b) \cong \overline{\Jac}^{\, d+\delta}(X_b).
\end{equation}
The construction of $\overline{\Jac}^{\, d+\delta}(X_b)$ follows from \cite[Theorem 1.21]{simpson2} and it is a compactification of the Jacobian $\Jac^{\, d+\delta}(X_b)$ of degree $d+\delta$ line bundles on $X_b$, hence we refer to it as the {\it compactified Jacobian}. Denote by $\Xxx \to B_{X,n}$ the family of $n$-to-$1$ spectral curves inside $|K_X|$ and endow it with a linearization induced from a linearization on $X$. Thanks again to \cite[Theorem 1.21]{simpson2} one can consider the relative compactified Jacobian $\overline{\Jac}^{\, d+\delta}_{B_{X,n}}(\Xxx)$. The spectral correspondence promotes to the whole moduli space \cite[Section 6]{simpson3}, giving rise to the isomorphism  
\begin{equation} \label{eq spectral isomorphism}
\morph{\overline{\Jac}^{\, d+\delta}_{B_{X,n}}(\Xxx)}{\M_X(n,d)}{\Ff \to X_b}{(E_{\Ff}, \varphi_{\Ff}).}{}{S_{X,n}}
\end{equation}

When the degree is a multiple of the rank, $d = nd'$, the Hitchin fibration $h_{X,n}:\M_X(n,n d')\to B_{X,n}$ admits a so-called \emph{Hitchin section} associated to any line bundle $\Ll \in \Jac^{\, d'+\delta/n}(X)$. This section is constructed by assigning to each $b\in B$ the Higgs bundle whose spectral datum is given by the line bundle $\pi^*_b \Ll \to X_b$. In other words, we have a morphism
\begin{equation}\label{eq s_F}
\morph{B_{X,n}}{\M_{X}(n,nd')}{b}{(\pi_{b,\ast}\pi_b^{\ast}\Ll,\pi_{b,\ast}\mu_{\pi_b^*\Ll}).}{}{\sigma_{X, \Ll}}
\end{equation}
Hitchin \cite{hitchin_lie} considered such sections for $\Ll = K_X^{(n-1)/2}$. 
In this case we omit the reference to the line bundle in our notation and we simply denote the corresponding Hitchin section by $\sigma_X$.

%
%
%

\section{Unramified covers and Higgs bundles}
\label{sc Fixed point subvarieties}

\subsection{Unramified covers and hyperholomorphic subvarieties in the moduli space}

Let $p : C \to X$ be a connected unramified cover of degree $m$ and Galois group $\Gamma$. In this section, we study the subvarieties that arise in the moduli space of Higgs bundles out of this geometrical setting. Some of the following results have been already obtained in \cite{HP}.

Let $K_C$ be the canonical bundle of $C$, and let \[\eta: |K_C|\to C\] be the corresponding projection. As $p$ is unramified
\begin{equation} \label{eq p unramified}
K_C\cong {p}^*K_X
\end{equation}
and
\[
|K_C| \cong |K_X| \times_X C,
\]
hence we have a Cartesian diagram
\begin{equation}\label{eq cartesian diagram total spaces}
\xymatrix{|K_C|\ar[r]^{\eta}\ar[d]_{q}&{C}\ar[d]^{{p}}\\
|K_X|\ar[r]_{\pi}&X,}
\end{equation}
$q$ being the obvious projection. In particular, $q$ is an unramified $\Gamma$-cover and $\eta\colon |K_C|\to {C}$ is $\Gamma$-equivariant. Note that the automorphism $\gamma : {C} \to {C}$, associated to any element of the Galois group, gives rise to an automorphism $\gamma : |K_C| \to |K_C|$ that we still denote by $\gamma$ by abuse of notation.

By \eqref{eq p unramified}, the pullback under ${p}\colon {C}\to X$ of a Higgs bundle is again a Higgs bundle. Moreover, polystability is preserved (e.g. because it sends solutions to the Hitchin
equations on $X$ to solutions to the Hitchin equations on ${C}$,
cf.~\cite{hitchin-self}). So we have a morphism
\begin{equation} \label{eqpullback}
\morph{\M_X(n,d)}{\M_{{C}}(n,md)}{(E,\varphi)}{({p}^*E,{p}^*\varphi)}{}{\hat{p}}
\end{equation} 
between the moduli spaces. 

\begin{remark} \label{rm image of hat p}
The image of $\hat{p}$ lies in $\M_{C}(n,md)^\Gamma$, the fixed point locus under the induced Galois group action on $\M_{{C}}(n,md)$ by pullback.
\end{remark}

By the projection formula and \eqref{eq p unramified}, if $(F,\phi)$ is a Higgs bundle over ${C}$ of rank $r$, then $(p_{*}F,p_{*}\phi)$ is a rank $n = mr$ Higgs bundle over $X$. Since $p$ is unramified and $X$ and $C$ are proper, $p$ is finite, so by \cite[Lemma 2.1 (ii)]{NR} we have
\begin{equation} \label{eq M^gamma in M_rm}
{p}^*p_{*}(F,\phi) = \bigoplus_{\gamma \in \Gamma} \gamma^* (F, \phi).
\end{equation}

Consider the moduli space $\M_{{C}}(r,d)$ of rank $r$ and
degree $d$ Higgs bundles over ${C}$. 

\begin{proposition}
\label{prop:push-fixed}
Let $p:C \to X$ be a connected unramified $m$-cover with Galois group $\Gamma$ and let $n=mr$. The pushforward under $p$, 
\begin{equation} \label{eq definition of check p_gamma}
\morph{\M_{{C}}(r,d)}{\M_{X}(n,d)}{(F,\phi)}{(p_{*}F,p_{*}\phi),}{}{\check{p}}
\end{equation}
is a hyperholomorphic finite morphism. Moreover, two rank $r$ Higgs
bundles over ${C}$ have the same image if and only if they are in
the same orbit under the $\Gamma$-action by pullback, so
\[
\M_X(n,d)^p := \image (\check{p}) \cong \M_C(r,d) / \Gamma.
\]
\end{proposition}

\begin{proof}
It follows from \cite[Proposition 3.1]{NR} that $\check{p}$ has image contained in the semistable locus and so it is well defined.
Moreover, it is hyperholomorphic because it corresponds to pushforward of projectively flat bundles under the Non-abelian Hodge Theorem.  

Since $p$ is unramified, it is obvious that two rank $r$ Higgs bundles over ${C}$ in the same orbit under the $\Gamma$-action by pullback will give the same image under $\check{p}$. Thanks to \eqref{eq M^gamma in M_rm}, we see that they have the same image only if they lie in the same $\Gamma$-orbit, so $\M_X(n,d)^p \cong \M_C(r,d)/\Gamma$. Note that $\M_C(r,d) / \Gamma$ is naturally a geometric quotient since $\M_C(r,d)$ is quasi-projective, hence $\M_C(r,d) \to \M_C(r,d) / \Gamma$ is finite. Thus $\check{p}$ is a finite morphism as it commutes with the composition of the isomorphism $\M_X(n,d)^p \cong \M_C(r,d)/\Gamma$ with the finite quotient, which is a finite morphism.    
\end{proof}


\subsection{The Hitchin map and unramified covers}
\label{sec:pull-hitchin}

Fix a connected unramified cover $p:C \to X$ of degree $m$, with Galois group $\Gamma$. In this section we study the restriction of the Hitchin map to $\M_X(n,d)^p$, with $n$ a multiple of $m$. Let  $B^p:= h_{X,n}(\M_X(n,d)^p)\subset B_{X,n}$ to be the image under the Hitchin map of the image of $\check{p}$.

\begin{notation}\label{not spectral curve} Let $r=n/m$. We shall employ the same notation for the Hitchin system in $\M_C(r,d)$ as the one used in Section \ref{sec:the Hitchin system}. So let 
\[h_{C,r}:\M_C(r,d)\to B_{C,r}=\bigoplus_{i=1}^rH^0(C,K_C^i)\]
 be the Hitchin map. For any given $a=(a_1,\ldots, a_r) \in B_{{C},r}$, denote by ${C}_{a}$ the corresponding spectral curve in $|K_C|$, with projection map \begin{equation}\label{eq:spectral_C_a}
 \eta_a=\eta|_{C_a}:C_a\to C,
 \end{equation} 
 where $\eta$ is defined in \eqref{eq cartesian diagram total spaces}.
 The curve $C_a$ is defined by the equation 
\begin{equation}\label{eq:spectral C}
\hat\lambda^{r} + \eta^* a_1 \hat\lambda^{r-1} + \dots + \eta^* a_{r}=0,
\end{equation}
with $\hat\lambda$ the tautological section of $\eta^*K_C$.
Set $\gamma(a):=\gamma^{*} a$ for every element $\gamma \in \Gamma$, and write $\gamma(a) = (\gamma(a)_1, \dots, \gamma(a)_{r})$ where $\gamma(a)_i \in H^0({C}, K_C^i)$.  
\end{notation}

The next proposition establishes the behavior of the Hitchin map under pullback by the Galois group. 

\begin{proposition} \label{pr Hitchin map is Galois equivariant}
The Hitchin map $h_{{C},r}:\M_C(r,d)\to B_{C,r}$ is equivariant for the action of the Galois group $\Gamma$ of $p:{C} \to X$ by pullback. Furthermore, for any $a \in B_{C,r}$, one has the Cartesian diagram 
\begin{equation}\label{eq cart diagram Xolb}
\xymatrix{{C}_{\gamma(a)}\ar[r]^{\gamma} \ar[d]_{\eta_{\gamma(a)}}&{C}_{a}\ar[d]^{\eta_{a}}
\\
{C}\ar[r]_{\gamma} &{C}.
}
\end{equation}
In particular, for any $\gamma\in \Gamma$, the spectral curves ${C}_{\gamma(a)}$ and ${C}_{a}$ are isomorphic.
\end{proposition}

\begin{proof}
Let $(F,\phi)$ be a Higgs bundle in $\M_{{C}}(r,d)$ such that $h_{{C}, r}(F,\phi) = a \in B_{{C}, r}$. Let $P_i$ be an invariant polynomial of degree $i$, and observe that $P_i(\gamma^*\phi) = \gamma^*P_i(\phi)$. It then follows that $h_{{C}, r}(\gamma^* (F,\phi)) = \gamma^* h_{{C}, r}(F,\phi)$ and the first part follows.

Since the spectral curve ${C}_{ a} \subset |K_C|$ is given by the vanishing of \eqref{eq:spectral C}, then ${C}_{\gamma(a)}$ is given by the vanishing of the pullback of \eqref{eq:spectral C} under $\gamma$. Note that  $(\gamma(a))_i=\gamma^*a_i$ and that the embedding $C_a\plonge |K_C|$ is $\Gamma$-equivariant (so that $\eta^*\gamma^*=\gamma^*\eta^*$). Therefore, given $y \in {C}_{ \gamma(a)}$, by definition of pullback, one has that
\[
\left( \hat\lambda^{r} + \eta^* (\gamma^*a)_1 \hat\lambda^{r-1} + \dots + \eta^* (\gamma^*a)_{r} \right)(y) = 0
\]
is equivalent to 
\[
\left( \hat\lambda^{r} + (\eta^* a_1) \hat\lambda^{r-1} + \dots + \eta^* a_{r} \right)(\gamma(y)) = 0
\]
 because $\hat\lambda$ is $\Gamma$-invariant, since $\hat\lambda=q^*\lambda$, where $q:|K_C|\to |K_X|$ and $\lambda$ is the tautological section of $\pi^* K_X$.
Thus $\gamma(y) \in C_a$ and the commutativity of \eqref{eq cart diagram Xolb} holds. The rest of the proposition follows from this.\end{proof}

Consider the moduli space of rank $n$ and degree $d$ Higgs bundles on ${C}$ and its associated Hitchin map
\[
h_{{C},n} :  \M_{{C}}(n,d) \to B_{{C}, n} = \bigoplus^{n}_{i=1} H^0({C}, K_C^i). 
\]
By \eqref{eq p unramified}, it follows that $p$ induces
\begin{equation}
\label{eq p_gamma for B}
{p}^*\colon B_{X,n} \to B_{{C},n}.
\end{equation}

\begin{lemma}\label{pullbackpgamma:inj&imag}
The induced map ${p}^* \colon B_{X,n}\to B_{{C},n}$ is injective and the following diagram commutes, where $\hat{p}$ is defined in \eqref{eqpullback}:
\[\xymatrix{\M_{X}(n,d)\ar[r]^{\hat{p}}\ar[d]_{h_{X,n}}&\M_{{C}}(n,md)\ar[d]^{h_{{C},n}}\\
B_{X,n}\ar@{^{(}->}[r]^{{p}^*}&B_{{C},n}.}\]
\end{lemma}
\begin{proof} 
Since ${p}$ is a local
isomorphism, ${p}^*:H^0(X,K_X^i)\to H^0({C},K_C^i)$ is
injective for every $i$, so
${p}^* \colon B_{X,n}\to B_{{C},n}$ is injective as well. The commutativity of the diagram is immediate from functoriality of pullback.
\end{proof}

\begin{proposition}\label{pr cartesian diagram with X_wtb} 
Let $b\in B_{X,n}$ and $\tilde{b}={p}^*b\in B_{{C},n}$. 
Let $X_b\subset |K_X|$ and ${C}_{ \tilde{b}}\subset |K_C|$ be the corresponding spectral curves. Then
\[
{C}_{ \tilde{b}} \cong X_b \times_X {C}
\]
and there is a Cartesian diagram
\begin{equation}\label{eq cartesian diagram spectral}
\xymatrix{{C}_{\tilde{b}}\ar[r]^{{\eta}_{\tilde{b}}}\ar[d]_{q_{\tilde{b}}}&{C}\ar[d]^{{p}}\\
X_b\ar[r]_{\pi_b}&X,
}
\end{equation}
where $q_{\tilde{b}}$, ${\eta}_{\tilde{b}}$ and $\pi_b$ are the restrictions
of the maps from \eqref{eq cartesian diagram total spaces}.  
In
particular, 
\begin{enumerate}[(i)]
\item \label{it equiv-reducibility} $X_b$ is reduced if and only if ${C}_{\tilde{b}}$ is reduced,
\item \label{it q is unramified} $q_{\tilde{b}}\colon C_{\tilde{b}} \to X_b$ is a connected unramified $\Gamma$-cover, and 
\item $\eta_{\tilde{b}}$ is a $\Gamma$-equivariant ramified degree $n$
cover, whose ramification locus is the pullback of that of $\pi_b$.
\end{enumerate}
\end{proposition}

\begin{proof} 
View the curve $X_b$ as a divisor in $|K_X|$.  First we prove
that the pullback by $q:|K_C|\to |K_X|$ of this divisor is
${C}_{\tilde{b}}$.
Write $b=(P_1(\varphi),\ldots,P_n(\varphi))$ for some Higgs bundle $(E,\varphi)$ in the Hitchin fiber of $b$. Then $X_b\subset |K_X|$ is defined by
\[
\lambda^n+\pi^*P_1(\varphi)\lambda^{n-1}+\dots+\pi^*P_n(\varphi)=0
\]
where we recall that $\lambda\in H^0(|K_X|,\pi^*K_X)$ is the tautological section.
Thanks to Lemma \ref{pullbackpgamma:inj&imag}, ${C}_{\tilde{b}}$ is defined by
\begin{equation}\label{eq:X_tilde-b-equation}
\hat\lambda^n+\eta^*{p}^*P_1(\varphi)\hat\lambda^{n-1}+\dots+\eta^*{p}^*P_n(\varphi)=0,
\end{equation}
where $\hat\lambda\in H^0(|K_C|,\eta^*K_C)$ is the tautological section. 
Clearly $\hat\lambda=q^*\lambda$ and so, in view of 
\eqref{eq cartesian diagram total spaces}, the equation defining ${C}_{\tilde{b}}$ is 
\[
q^*\lambda^n+q^*\pi^*P_1(\varphi)q^*\lambda^{n-1}+\dots+q^*\pi^*P_n(\varphi)=0.
\]
This shows that ${C}_{\tilde{b}}=q^*X_b$ as desired. 

Viewing $X_b\times_X {C}$ inside $|K_X|\times_X
{C}\cong|K_C|$ one readily sees that it satisfies
(\ref{eq:X_tilde-b-equation}) and, therefore, by the universal
property of the fiber product, it is isomorphic to $q^*X_b$. The rest of the lemma follows from this observation.
\end{proof}

We now study the relation of $\hat{p}\circ\check{p}:\M_C(r,d)\to\M_C(n,md)$ with the corresponding Hitchin maps (recall that $n=mr$).

\begin{proposition}\label{prop:equationspectralcurve}
Let $(F,\phi)$ be a Higgs bundle of rank $r$ over ${C}$ and consider $\hat{p}\circ\check{p}(F,\phi)$. Let $a \in B_{{C},r}$ and $\tilde{b} \in B_{{C},n}$ be the image under the Hitchin map of $(F,\phi)$ and $\hat{p}\circ\check{p}(F,\phi)$ respectively. Then, the spectral curve ${C}_{ \tilde{b}}\subset|K_C|$ is given by the vanishing of the section
\[
\prod_{\gamma \in \Gamma} \left( \hat\lambda^{r} +\eta^* \gamma(a)_1 \hat\lambda^{r-1} + \dots + \eta^* \gamma(a)_{r} \right) \in H^0(|K_C|,\eta^*K_C^n).
\] 
In particular ${C}_{ \tilde{b}}$ is reduced if and only if ${C}_{ a}$ is reduced and $\gamma(a)\neq \gamma'(a)$ for all distinct $\gamma,\gamma' \in \Gamma$ ({\it i.e.}, if $a \in B_{{C}, r}$ is not fixed by any non-trivial element of the Galois group). In that case ${C}_{ \tilde{b}}$ is reducible and 
\begin{equation} \label{eq decomposition of X_wtb}
{C}_{ \tilde{b}} = \bigcup_{\gamma \in \Gamma} {C}_{ \gamma(a)}.
\end{equation}
\end{proposition}

\begin{proof}
From \eqref{eq M^gamma in M_rm}, one has that ${C}_{ \tilde{b}}$ is given by the vanishing of 
\[\hat\lambda^n+\eta^*P_1\Big (\bigoplus_{\gamma \in \Gamma}\gamma^{*} \phi \Big)\hat\lambda^{n-1}+\dots+\eta^*P_n\Big (\bigoplus_{\gamma \in \Gamma}\gamma^{*} \phi \Big )\]
{\it i.e.}, of
\[\prod_{\gamma \in \Gamma} \left( \hat\lambda^{r} + \eta^* P_1(\gamma^{*} \phi) \hat\lambda^{r-1} + \dots + \eta^* P_{r}(\gamma^{*} \phi) \right).
\]
The rest of the proposition follows immediately.
\end{proof}

We can now describe the subspace $B^p=h_{X,n}(\M_X(n,d)^p)\subset B_{X,n}$.

\begin{proposition}
\label{prop:commutative-rank-1-hitchin}
There exists a map, 
\[
\zeta : B_{{C},r} \to B^p,
\]
making the $\Gamma$-equivariant diagram 
\begin{equation}\label{eq hat pgamma}
\xymatrix{
\M_{{C}}(r,d)
\ar[rr]^{\check{p}}\ar[d]_{h_{{C},r}} & & \M_{X}(n,d)^p \ar[d]^{h_{X,n}}\\
B_{{C},r} \ar[rr]_{\zeta} & & B^p}
\end{equation}
commutative, with $\Gamma$ acting by pullback on $\M_C(r,d)$ and $B_{C,r}$, and trivially on $\M_X(n,d)^p$ and $B^p$. The map $\zeta$ induces an isomorphism
\begin{equation}\label{eq B^gamma cong B_Xgamma / ZZ_m}
B^p \cong B_{{C},r}/\Gamma.
\end{equation}
Hence,
\begin{equation} \label{eq dim B^gamma}
\dim(B^p)=rn(g-1)+1.
\end{equation}
\end{proposition}

\begin{proof} 
By Proposition \ref{pr Hitchin map is Galois equivariant}, $\hat{p} \circ \check{p}$ induces the morphism $B_{{C}, r} \to B_{{C}, n}$ defined by 
\[(P_1(\phi),\ldots, P_r(\phi))\mapsto\bigg(P_1\Big (\bigoplus_{\gamma \in \Gamma}\gamma^{*} \phi \Big),\ldots,P_n\Big (\bigoplus_{\gamma \in \Gamma}\gamma^{*} \phi \Big)\bigg).\] The image of this morphism is clearly contained in $p^*(B^p)$, and by Lemma \ref{pullbackpgamma:inj&imag}, ${p}^*$ is injective, so the previous map defines the map $\zeta : B_{{C},r} \to B^p$. Explicitly, 
\[\zeta(P_1(\phi),\ldots, P_r(\phi))=(P_1(p_\ast\phi),\ldots,P_1(p_\ast\phi)).\] The commutativity of \eqref{eq hat pgamma} is immediate from construction, and so is \eqref{eq B^gamma cong B_Xgamma / ZZ_m}.

Since $B^p$ is isomorphic to the finite quotient $B_{{C},r}/\Gamma$, its dimension is
\[
\dim(B^p) = \dim(B_{{C}, r}) = r^2(g(C) - 1) + 1.
\]
As ${C} \to X$ is an unramified $m$-cover, the genus of ${C}$ equals $g(C) = m(g-1) + 1$ and \eqref{eq dim B^gamma} holds. 
\end{proof}

\begin{definition}\label{def: reducedlocus}
Let $B^{\sm}_{C,r}$ be the dense open subset of $B_{C,r}$ given by those points $a\in B_{C,r}$ whose associated spectral curve $C_a$ is smooth. Let us also consider the $\Gamma$-invariant subset $B^{\fs}_{{C},r}$ to be the open subset of $B^{\sm}_{{C},r}$ given by those elements where $\Gamma$ acts freely. Next, define $B_{C,r}^{\ni} \subset B_{C,r}^{\fs}$ as the subset given by those curves $C_a$ whose intersection with $C_{\gamma(a)}$ has nodal singularities, for every $\gamma \in \Gamma$, and such that $C_a \cap C_{\gamma(a)}\cap C_{\gamma'(a)}=\emptyset$ for all $\gamma'\in\Gamma\setminus\{1,\gamma\}$. Set as well $B^p_{\fs} := B^\fs_{C,r} /\Gamma$ and  $B^p_{\ni} := B^\ni_{C,r} /\Gamma$. Finally, define 
\[
\M_X(n,d)^p_\ni := \M_X(n,d)^p \times_{B^p} B^p_\ni
\]
and 
\[
\M_C(r,d)_\ni := \M_C(r,d) \times_{B_{C,r}} B^\ni_{C,r}. 
\]
\end{definition}

\begin{remark}
Note that $B^p_{\fs}$ parametrizes reduced curves by Propositions \ref{prop:equationspectralcurve} and \ref{pr cartesian diagram with X_wtb} \eqref{it equiv-reducibility}. By Theorem \ref{thm:normaliz} (v) 
below, such spectral curves are {\it integral}, explaining the notation. The notation used for $B^p_{\ni}$ stands for {\it nodal and integral}.
\end{remark}

Recall that we defined $\delta$ in \eqref{eq:delta} as the degree of the ramification divisor of the spectral curves in $B_{X,n}$. Accordingly, we define $\rho$ to be the degree of the ramification divisor of the spectral curves ${C}_{a} \to {C}$ in $B_{{C}, r}$,
\begin{equation} \label{eq def delta_gamma}
\rho :=r(r-1)(g(C)-1) = n(r-1)(g-1).
\end{equation}
%

Next, we describe the geometry of the spectral curves $X_b$ when $b$ lies in $B^p_\fs$. For the statement of the following theorem, recall the notation of the Cartesian diagrams \eqref{eq cartesian diagram spectral} and \eqref{eq hat pgamma}.

\begin{theorem}\label{thm:normaliz}
Let $a \in B^\fs_{C,r}$ and consider $b = \zeta(a) \in B_\fs^p$ and $\tilde{b}={p}^*b\in B_{{C},n}$.
\begin{enumerate}[(i)]
\item \label{it irred comp X_tildeb} 
The spectral curve  ${C}_{\tilde{b}}$ is reduced, connected, with $m=|\Gamma|$ irreducible components $\{ {C}_{\gamma(a)} \}_{\gamma \in \Gamma}$, all isomorphic to each other. The singular divisor $\sing({C}_{\tilde{b}})$ of  ${C}_{\tilde{b}}$ is given by the intersections of distinct components, and has degree $n^2(m-1)(g-1)$. The Galois group $\Gamma$ of $q_{\tilde{b}}$ permutes the components of ${C}_{\tilde{b}}$.

\item \label{Xb singular} 
The spectral curve  $X_b$ is reduced but singular. Its singular divisor $\sing(X_b)$ satisfies $q_{\tilde{b}}^*\sing(X_b)=\sing(C_{\tilde{b}})$. In particular
\begin{equation} \label{eq deg sing X_b}   
\deg(\sing(X_b))=\delta - \rho = n(n-r)(g-1). 
\end{equation}
 If $x\in X_b$ is a singularity, and
$y \in {C}_{\gamma'(a)} \cap {C}_{\gamma''(a)}$ (with $\gamma'\neq \gamma''$) is a
singular point on ${C}_{\tilde{b}}$ mapping to $x$, then the $m$
singularities of ${C}_{\tilde{b}}$ mapping to $x$ are precisely the ones of the form
$\gamma(y)\in {C}_{\gamma\gamma'(a)} \cap {C}_{\gamma\gamma''(a)}$, for all
$\gamma \in \Gamma$.

\item \label{it description of X_b} 
Let  $C_{a}\subset {C}_{\tilde{b}}$ be an irreducible component. Let $q : |K_C| \to |K_X|$ be as in \eqref{eq cartesian diagram total spaces}. Then $q(C_{a})=q_{a}(C_{a})=X_b$ and 
\begin{equation} \label{eq def nu_a}
\nu_{a}:= q_{a} : {C}_{a} \longrightarrow X_b
\end{equation}
is a normalization fitting in the commutative diagram 
\begin{equation}\label{eq normalization X b}
\xymatrix{
{C}_{a}\ar[r]^{\eta_{a}}\ar[d]_{\nu_{a}}&{C}\ar[d]^-{{p}}
\\
X_b\ar[r]_{\pi_b}&X.
}
\end{equation}
In addition, 
\begin{equation} \label{eq relation between the nu_a i}
\nu_{\gamma(a)} = q_{\gamma(a)} =\nu_{a} \circ \gamma.
\end{equation}

\item \label{it Cartesian in X_gamma} 
The (disconnected) curve in $|K_C| \times B_{{C},r}$ given by
\[
\wt{C}_{\tilde{b}}=\bigsqcup_{\gamma\in \Gamma}{C}_{ \gamma(a)}\times\{\gamma(a)\}
\]
is the normalization of ${C}_{\tilde{b}}$, where the normalization morphism $\tilde\nu_{\tilde{b}}:\wt{C}_{\tilde{b}}\to {C}_{\tilde{b}}$ is given by projecting onto the first factor. Furthermore, the diagram
\begin{equation}\label{normalizXtildeb-2}
\xymatrix{\widetilde {C}_{\tilde{b}}\ar[r]^{\tilde\nu_{\tilde{b}}}\ar[d]_{\wt{q}_{a}}& {C}_{\tilde{b}}\ar[d]^{q_{\tilde{b}}}\\
{C}_{a}\ar[r]_{\nu_{a} }&X_b}
\end{equation}
is Cartesian, where $\wt{q}_{a}$ is the unramified $\Gamma$-cover
given by $\wt{q}_{a}(z,\gamma(a))=\gamma^{-1}(z)$.

\item \label{Xb irreducible} The spectral curve $X_b$ is integral.
\end{enumerate}
\end{theorem}

\begin{proof}
Since $a$ is taken in $B^\fs_{C,r}$, ${C}_{ \tilde{b}}$ satisfies the hypothesis of the second part of Proposition~\ref{prop:equationspectralcurve}. Then, ${C}_{ \tilde{b}}$ is reduced and reducible, decomposing as described in \eqref{eq decomposition of X_wtb}. Again by hypothesis, ${C}_{ a}$ is smooth and therefore irreducible. Furthermore, $C_a \cong C_{\gamma(a)}$ by Proposition \ref{pr Hitchin map is Galois equivariant}. Therefore, \eqref{eq decomposition of X_wtb} is, in fact, a decomposition of ${C}_{\tilde{b}}$ into its irreducible components and the singularities of ${C}_{\tilde{b}}$ are the intersections of the distinct components, \[
\sing(C_{\tilde{b}}) = \bigcup_{\gamma \neq \gamma'} C_{\gamma(a)} \cap C_{\gamma'(a)}.
\]
Note that all the $C_{\gamma(a)}$ lie in the linear system $rC$ inside $|K_C|$. This has two important consequences. The first one is that we can perform the count of the intersection divisor taking a generic element of $B^p_{\fs}$ without triple intersections of the components $C_{\gamma(a)}$. The second one is that two distinct components, ${C}_{\gamma(a)}$ and ${C}_{\gamma'(a)}$, lie in the linear system of $rC$ inside $|K_C|$, so their intersection is 
\begin{equation} \label{eq intersection of components X_gamma,a}
r^2 \cdot ({C})^2 = r^2(g(C) - 1) = 2nr(g-1).
\end{equation}
Hence, 
\[
\deg(\sing({C}_{\tilde{b}}))= \binom{m}{2}2nr(g-1) = n^2(m-1)(g-1).
\]
It follows from Proposition \ref{pr Hitchin map is Galois equivariant} that the Galois group permutes the $C_{\gamma(a)}$. This  completes the proof of \eqref{it irred comp X_tildeb}.

For \eqref{Xb singular}, note that by \eqref{it irred comp X_tildeb} and Proposition~\ref{pr cartesian diagram with X_wtb} \eqref{it equiv-reducibility}, $X_b$ is reduced. The rest follows from the fact that $q_{\tilde{b}} : C_{\tilde{b}} \to X_b$ is an unramified $\Gamma$-cover, as we have shown in Proposition~\ref{pr cartesian diagram with X_wtb} \eqref{it q is unramified}.


To see \eqref{it description of X_b}, we start by observing that \eqref{eq relation between the nu_a i} follows from \eqref{eq cart diagram Xolb}. Then, the maps $\{ q_{\gamma(a)} = \nu_{\gamma(a)} \}_{\gamma \in \Gamma}$, have all the same image. From this, in view of \eqref{it irred comp X_tildeb} and the fact that $q_{\tilde{b}}:{C}_{\tilde{b}}\to X_b$ is surjective, we conclude that $q_{\gamma(a)} = \nu_{\gamma(a)}$ is surjective for each $\gamma \in \Gamma$.  
Since ${C}_{a}$ is smooth, in order to prove that the maps $\nu_{\gamma(a)}$ are normalization morphisms, it now suffices to show that one of them (say $\nu_a$) is generically injective. If $x\in X_b$ does not belong to the ramification locus $\sing(X_b)$ then, by \eqref{it irred comp X_tildeb} and \eqref{Xb singular}, each point inside the fiber $q_{\tilde{b}}^{-1}(x)$ lies in a different irreducible component of ${C}_{\tilde{b}}$. So $\nu_{a}$ is injective over the smooth locus of $X_b$. The commutativity of \eqref{eq cartesian diagram spectral} concludes the proof of \eqref{it description of X_b}.

Consider now \eqref{it Cartesian in X_gamma}. The map $\tilde\nu_{\tilde{b}}:\wt{C}_{ \tilde{b}} \to {C}_{ \tilde{b}}$ is a normalization morphism by the description of ${C}_{ \tilde{b}}$ and its singularities given in \eqref{it irred comp X_tildeb}. To see that \eqref{normalizXtildeb-2} is Cartesian, note that by the universal property of fibered products there is a morphism 
\[
\xymatrix{\wt{C}_{\tilde{b}}\ar[r]&{C}_{a}\times_{X_b}{C}_{\tilde{b}}}.
\]
Since any morphism of principal bundles is an isomorphism, the statement follows.

Finally, for \eqref{Xb irreducible}, we already know that $X_b$ is reduced. Since ${C}_{a}$ is smooth, it is irreducible, thus $X_b$ is irreducible by \eqref{it description of X_b}. 
\end{proof}

\begin{remark}
For every $b \in B^p_\fs$ the corresponding spectral curve $X_b$ is normalized by a smooth spectral curve ${C}_{ a}$ in $B^\sm_{{C},r}$.   
\end{remark}

We introduce some notation to describe the fibers of the Hitchin map restricted to the  subvariety $\M_X(n,d)^p$. Recall that $\rho$ is defined in \eqref{eq def delta_gamma} as the degree of the ramification divisor of the spectral curves ${C}_{ a} \to {C}$. Consider the pushforward under the normalization morphism $\nu_{a}:{C}_{a} \to X_b$ defined in \eqref{eq def nu_a}, 
\begin{equation} \label{eq definition of check nu_a}
\morph{\Jac^{d+\rho}({C}_{a})}{\ol{\Jac}^{\, d+\delta}(X_b)}{L}{\nu_{a,*}L,}{}{\check{\nu}_{a}}
\end{equation}
where $\delta - \rho = \deg(\sing(X_b))$ by \eqref{Xb singular} of Theorem \ref{thm:normaliz}. Any other $\gamma(a)$ also projects to $b$ under the map $\zeta : B_{{C},r} \to B^p$, defined in Proposition \ref{prop:commutative-rank-1-hitchin}, and for such $\gamma(a)$, a similar map $\check{\nu}_{\gamma(a)}$ exists as well. Furthermore, \eqref{eq relation between the nu_a i} implies that all these morphisms share the same image,
\[
\image (\check{\nu}_a) = \image (\check{\nu}_{\gamma(a)}).
\]

Recall the morphism $\check{p}$ in \eqref{eq definition of check p_gamma}. The following proposition describes the fibers of the Hitchin map over $B^p_\fs$ restricted to $\M_X(n,d)^p$.

\begin{proposition}\label{pr restriction of M^gamma to a Hitchin fiber}
Let $b \in B^p_\fs$, and pick $a \in B^\fs_{C,r}$
such that $\zeta(a)=b$. Then,
\begin{enumerate}[(i)]
\item \label{it restriction of check p_gamma to Hitchin fiber} the first line of diagram \eqref{eq hat pgamma} restricts to 
\begin{equation} \label{eq restriction of hap p_gamma to a fiber}
\bigsqcup_{\gamma \in \Gamma}\Jac^{d+\rho}({C}_{\gamma(a)})  \xrightarrow{\ \check{p}_b\ }  h_{X,n}^{-1}(b) \cap  \M_X(n,d)^p,
\end{equation}
and $\check{p}_b$ is an unramified cover, with the Galois group $\Gamma$ acting by pullback, hence permuting the connected components of the domain. 

\item\label{it h fiber M^p} the intersection of $\M_X(n,d)^p$ with the Hitchin fiber is 
\begin{equation}\label{eq description of the fiber of h^gamma}
h_{X,n}^{-1}(b) \cap  \M_X(n,d)^p \cong \bigsqcup_{\gamma \in \Gamma}\Jac^{d+\rho}({C}_{\gamma(a)})/\Gamma \cong \Jac^{d+\rho}({C}_{a}) \cong \image(\check{\nu}_{a})\subset \overline{\Jac}^{\, d+\delta}(X_b).
\end{equation}
where the inverse of the second isomorphism is defined by assigning to  $L\in\Jac^d({C_a})$ its $\Gamma$-orbit, which can be naturally identified with an element of $\bigsqcup_{\gamma \in \Gamma}\Jac^{d+\rho}({C}_{\gamma(a)})/\Gamma$.
\end{enumerate}
\end{proposition}

\begin{proof}
The first statement in \eqref{it restriction of check p_gamma to Hitchin fiber} is clear by Proposition \ref{prop:commutative-rank-1-hitchin}. The union is disjoint since $a$ is not fixed by any element of $\Gamma$, by definition of $B^\fs_{C,r}$. This also shows that $\check{p}_b$ is an unramified cover with Galois group $\Gamma$ acting by pullback.

To see that \eqref{eq description of the fiber of h^gamma} holds, by {\it (i)} we have
\[
h_{X,n}^{-1}(b) \cap \M_X(n,d)^p \cong\bigg(\bigsqcup_{\gamma \in \Gamma}\Jac^{\, \rho +d}({C}_{\gamma(a)})\bigg)/\Gamma,
\]
hence, recalling that ${C}_{ \gamma(a)} \cong {C}_{ a}$ by Proposition \ref{pr Hitchin map is Galois equivariant}, we can choose a representative and
\[
h_{X,n}^{-1}(b) \cap \M_X(n,d)^p \cong \Jac^{\, \rho +d}({C}_{a}).
\]
Since the restriction of $\check{p}_b$ to $\Jac^{\rho +d}({C}_{a})$ coincides with $\check{\nu}_a$, \eqref{eq description of the fiber of h^gamma} follows.
\end{proof}

We finish the section with an observation that will be useful in Sections~\ref{subsec:Lagrangian} and \ref{sec:non-max-ord}.

\begin{remark}\label{rk Levi and parabolic groups}
It follows from \eqref{eq M^gamma in M_rm} that the image of the map $\hat{p} \circ \check{p}$ is contained in the Levi subgroup $\L_{(r,m)} = \GL(r,\CC)^{\times m}$ associated to the parabolic subgroup $\P_{(r,m)} = N_{\GL(n,\CC)}(\L_{(r,m)})$. Denote the unipotent radical of $\P_{(r,m)}$ by $\U_{(r,m)}$ and recall that $\P_{(r,m)}=\L_{(r,m)}\ltimes \U_{(r,m)}$. Note that $\L_{(1,n)}$ is the Cartan subgroup of $\GL(n,\CC)$ while $\P_{(1,n)}$ is the Borel subgroup.
\end{remark}

Before stating the following corollary, we need to briefly introduce some notation from \cite{Borel}. Denote by $\M_{(r,m)} \subset \M_{{C}}(n,md')$ the image of the moduli space of $\L_{(r,m)}$-Higgs bundles of multidegree $(d', \ldots, d')$. When $r=1$ and $m = n$, $\L_{(1,n)}$ is the Cartan subgroup and $\M_{(1,n)}$ is called in \cite{Borel} the \emph{Cartan locus} of $\M_{{C}}(n,nd')$ (and, as proved in loc. cit., it supports a $\BBB$-brane). Let $V_{(r,m)} \subset B_{{C},n}$ be the image of $\M_{(r,m)}$ under the Hitchin map $h_{{C},n}$.

\begin{corollary}\label{cor:pullbacktoCar}
The commutative diagram of Lemma~\ref{pullbackpgamma:inj&imag} restricts to 
\[
\xymatrix{\M_{X}(n,nd')^p\ar[r]^-{\hat{p}}\ar[d]_{h_{X,n}}&\M_{(r,m)}\ar[d]^{h_{{C},n}}
\\
B^p\ar@{^{(}->}[r]^{{p}^*}&V_{(r,m)}.}
\]
\end{corollary}

\begin{proof} It is enough to prove that $\hat{p}$ maps
$\M_{X}(n,d)^p$ to $\M_{(r,m)}$ and this follows from \eqref{eq M^gamma in M_rm}.
\end{proof}

\subsection{Cyclic covers and tensorization by torsion line bundles}

We study in this section the fixed point subvariety of $\M_X(n,d)$
under tensorization by a fixed line bundle of order $n$. For the moduli of vector bundles, Narasimhan and Ramanan \cite{NR} proved that bundles on $X$ which are in the image of the pushforward map from $C$ are fixed by tensorization with all line bundles associated to characters of the Galois group $\Gamma$. When $\Gamma$ satisfies a certain technical condition (which is the case, for example, of cyclic covers, see \cite[Lemma 2.5]{NR}), the converse is also true for simple bundles. Nasser \cite{nasser:2005} proved that when $\Gamma\cong\ZZ_n$ is cyclic, the converse holds also for non simple bundles. The study of these fixed points in the moduli of vector bundles with fixed determinant, has also been carried out in \cite{garciaprada-ramanan} and, in detail in rank $n=2$, in \cite{garciaprada-ramanan-2}.

Let $\Jac(X)[n]$ be the subgroup of $n$-torsion points of the Jacobian $\Jac(X)=\Jac^0(X)$ of $X$,
\[\Jac(X)[n]:=\{L\in\Jac(X)\suchthat L^n\cong\Oo_X\}\cong\ZZ_n^{2g}.\]  
For convenience of notation, we shall use different symbols for an element $\xi\in \Jac(X)[n]$ as an abstract group and for the corresponding order $n$ line bundle $L_\xi\in\Jac(X)[n]$ on $X$.

This group acts on $\M_X(n,d)$ by tensorization,
\begin{displaymath}
\morph{\M_X(n,d)}{\M_X(n,d)}{(E,\varphi)}{(E\otimes L_\xi, \varphi).}{}{\xi}
\end{displaymath}
Denote by $\M_X(n,d)^{\xi}$ the subvariety of points fixed by $\xi\in\Jac(X)[n]$. It is a hyperholomorphic subvariety since tensorization by a flat line bundle is holomorphic in the three complex structures of
$\M_X(n,d)$ (see \cite{garciaprada-ramanan} for a proof in the case of
$\SL(n,\CC)$-Higgs bundles, which also applies to the case of
$\GL(n,\CC)$).
\begin{notation}
Let $m$ be the order of $L_\xi$ in $\Jac(X)[n]$, and set $r = n/m$. 
\end{notation}

Then $L_\xi \in \Jac(X)[m]$ is a primitive element.  Consider the projection ${p}_\xi : |L_\xi| \to X$ and let  $\lambda_\xi : |L_\xi| \to {p}_\xi^*L_\xi$ be the  tautological section.  
Define $C_\xi$ to be the curve in the total space $|L_\xi|$ given by the zero locus of the section $\lambda_\xi^m - {p}_\xi^*\mathbf{1}\in H^0(|L_\xi|,\Oo_{|L_\xi|})$. Denote the restriction to $C_\xi$ of the projection morphism by the same symbol,
\begin{equation}\label{eq pgamma}
{p}_\xi: {C}_\xi \to X.
\end{equation}
Then, \eqref{eq pgamma} is a connected unramified regular cover of $X$ with Galois group $\ZZ_m$.

\begin{remark}
Reciprocally, a connected unramified regular $\ZZ_m$-cover $p: C \to X$ defines a line bundle $L_p \to X$ of order $m$ by setting the holonomy of $L_p$ to be given by parallel transport of the lifts from $X$ to $C$. 
\end{remark}

We next describe all the points fixed by any $\xi \in \Jac(X)[n]$.


\begin{proposition}\label{prop:fixed-points-any-order}
Let $\xi\in\Jac(X)[n]$ be of order $m$ and let $r=n/m$. A semistable Higgs bundle $(E,\varphi)\in\M_X(n,d)$ is fixed under $\xi$ if and only if it is the pushforward of an element in $\M_{{C}_\xi}(r,d)$.
\end{proposition}

\begin{proof}
Since semistability and degree are preserved under pushforward by $p_\xi$ (cf. \cite[Lemma~3.1]{NR}) and since points in the image of the pushforward are fixed (this is a direct consequence of the projection formula), all we need to prove is that all fixed points are in such image. 

Let then $(E, \varphi)\in \M_X(n,d)^\xi$. Then there exists an isomorphism 
\[
f:E\stackrel{\cong}{\longrightarrow}E\otimes L_\xi
\]
such that 
\[
\varphi \otimes \Id_{L_\xi} = f \otimes \Id_{L_\xi} \circ \varphi \circ f^{-1}.
\]
We can then treat the pair $(E, f)$ as an $L_\xi$-twisted Higgs bundle of rank $n$ (i.e. the Higgs field $f$ is twisted by $L_\xi$ instead of $K_X$). The spectral correspondence described in section \ref{sec:the Hitchin system} also holds for  $L_\xi$-twisted Higgs bundles \cite{BNR,Schaub}, hence establishing a one-to-one correspondence between isomorphism classes of pairs $(E,f)$ and their spectral datum, $\Ll \to C'$, where $\Ll$ is a  rank one torsion free sheaf on the corresponding spectral curve $C' \subset |L_\xi|$, given by the vanishing of 
the section $\sum_{i=1}^np_\xi^*s_i\lambda_\xi^{n-i}$ of ${p}_\xi^*L_\xi^n\cong\Oo_{|L_\xi|} \to |L_\xi|$, where $s_i\in H^0(X,L_\xi^i)$. Since $L_\xi$ has no global sections unless it is trivial, in which case all the global sections are constant, we find that $C'$ is given by the vanishing of
\begin{equation} \label{eq spectral curve for L_gamma}
\lambda^{m \cdot r}_\xi + {p}_\xi^*s_m \lambda^{m \cdot (r-1)}_\xi + \dots + {p}_\xi^*s_{mr},
\end{equation}
where $s_i \in H^0(X, L_\xi^{m(r-i)})=\CC$.
We can recover $f : E \stackrel{\cong}{\to} E \otimes L_\xi$ as the pushforward of the tensorization morphism $\mu_{\lambda_\xi}:\Ll \stackrel{\otimes \lambda_\xi}{\to} \Ll \otimes {p}^*L_\xi$.

Consider now the curve $C_0'\subset |L_\xi^m|\cong|\Oo_X|$ given by 
\begin{equation} \label{eq spectral curve for Oo}
\lambda^{r}_0 + p_0^*s_1 \lambda^{r-1}_0 + \dots + p_0^*s_{r},
\end{equation}
where $p_0 : |\Oo_X| \to X$ and $\lambda_0\in H^0(|\Oo_X|,\Oo_{|\Oo_X|})$ is the tautological section. Note that $C_0'$ is naturally identified with the characteristic polynomial of $f^m:E\to E$.

Since $L^m_\xi\cong\Oo_X$, there is a morphism $|L_\xi|\longrightarrow |L_\xi^m|$ given by tensorization with itself $m$-times. Under this morphism, $C'$ is sent to $C'_0$. Fixing a trivialization $\Oo_X \cong X \times \CC$, since the sections $s_i$ are constant, one has $C'_0 \cong X \times D$, where $D = \sum_i \ell_i d_i$ with $\deg D = r$ is the divisor of the points in $\CC$ defined by the zeros of the polynomial associated to \eqref{eq spectral curve for Oo}. It follows that \eqref{eq spectral curve for L_gamma} can be rewritten as
\[
\prod_{j=1}^t (\lambda_\xi^m - {p}^*d_j)^{\ell_j}. 
\]
Since $f$ is an isomorphism it follows that $\det(f) = \prod_{j=1}^t d_j^{\ell_j} \neq 0$, so all the $d_j$ are non-zero. Then, the vanishing locus of $\lambda_\xi^m - {p}^*d_j$ is naturally isomorphic to ${C}_\xi$ after scaling. It then follows that $C'=\bigsqcup_{j=1}^t{C}_\xi^{l_j}$, where ${C}_\xi^{l_j}$ is the possibly non reduced   curve given by the vanishing of $(\lambda_\xi^m - {p}^*d_j)^{l_j}$. As a consequence $C'$ projects onto ${C}_\xi$ and the following diagram commutes
\[
\xymatrix{C' \ar[rd]_{p'} \ar[r]^{u} & {C}_\xi \ar[d]^{{p}_\xi}
\\
& X.
}
\]
Then, setting $E' := u_*\Ll$ one has that
\[
{p}_{\xi,*} E' \cong E.
\]
Taking $f' : E' \stackrel{\cong}{\to} E'\otimes {p}_\xi^*L_\xi$ to be the pushforward under $u$ of $\Ll \stackrel{\otimes \lambda_\xi}{\to} \Ll \otimes ({p}_\xi\circ t)^*L_\xi$, one also has
\[
p_{\xi,*}f' \cong f.
\]

Since $p_{\xi}$ is unramified, one has that ${p}_\xi^*E$ is given by $m$ copies of $E'$ although we denote them by 
\begin{equation} \label{eq decomposition of p*E}
{p}_\xi^*E \cong \bigoplus_{i=0}^{m-1} E' \otimes {p}_\xi^*L_\xi^i,
\end{equation}
where we recall that ${p}_\xi^*L_\xi$ is isomorphic to $\Oo_{{C}_\xi}$. Note also that ${p}_\xi^*f$ is given by copies of $f'$, permuting cyclically the factors of ${p}_\xi^*E$,
\[
f' : E' \otimes {p}_\xi^*L_\xi^i \stackrel{\cong}{\longrightarrow} E' \otimes {p}_\xi^*L_\xi^{i+1}.
\]

After \eqref{eq decomposition of p*E}, the pullback ${p}_\xi^*\varphi$ can be described in terms of a $m \times m$ matrix ${p}_\xi^*\varphi = \left ( \varphi'_{ij} \right)$, where $\varphi'_{ij} : E' \otimes {p}_\xi^*L_\xi^i \longrightarrow E' \otimes {p}_\xi^*L_\xi^i \otimes K_C$. Taking a Jordan-H\"older filtration of $({p}_\xi^*E, {p}_\xi^*\varphi)$ one can always write ${p}_\xi^*\varphi$ in a upper diagonal form, {\it i.e.} with $\varphi'_{ij} = 0$ if $i < j$. Observe that ${p}_\xi^* f$ corresponds with the permutation of the factors of ${p}_\xi^*E$ by $i \mapsto i + 1$. Since $\varphi$ and $f$ commute, so do ${p}_\xi^* \varphi$ and ${p}_\xi^* f$, and therefore $\varphi'_{ij} = 0$ for any $i \neq j$ and $\varphi'_{ii} = \varphi'_{jj}$. Picking $\phi = \varphi'_{ii}$, one has that $\varphi = {p}_{\xi,*} \phi$.
\end{proof}

The following is the fundamental result describing the fixed point
subvariety $\M_X(n,d)^{\xi}$ for any $\xi \in \Jac(X)[n]$. This generalizes the description given in \cite{HT} for $r$ and $d$ coprime. 

\begin{theorem}
\label{thm:narasimhan-ramanan:1975}
Let $\xi\in\Jac(X)[m]$ be of order $m$ with $n=mr$. Then pushforward under
$p_\xi:C_\xi\to X$ induces a hyperholomorphic isomorphism
\[
\M_X(n,d)^{\xi} \cong \M_X(n,d)^{p_\xi} \cong \M_{C_\xi}(r,d)/\ZZ_m,
\]
with the Galois group $\ZZ_m$ acting by pullback.
\end{theorem}
\proof
Straightforward from Propositions \ref{prop:push-fixed} and \ref{prop:fixed-points-any-order}.
\endproof


\section{Narasimhan--Ramanan $\BBB$-branes for covers of maximal degree}
\label{sc BBB branes}

By definition (cf.~\cite{kapustin&witten}), a \emph{$\BBB$-brane} on a
hyperk\"ahler manifold $\M$ is a pair \[(\mathrm{N},(\Fff, \nabla_{\Fff})),\] where:
\begin{itemize}
\item $\mathrm{N}\subset\M$ is a \emph{hyperholomorphic subvariety}, {\it i.e.}\ a subvariety which is holomorphic with respect to the three complex structures $I_1$, $I_2$ and $I_3$.
\item $(\Fff, \nabla_{\Fff})$ is a \emph{hyperholomorphic sheaf}
  supported on $\mathrm{N}$, {\it i.e.}\ a locally free sheaf $\Fff$
  of finite rank over the ring of $C^\infty$-functions on $\mathrm{N}$
  equipped with a
  connection $\nabla_{\Fff}$ whose curvature is of type $(1,1)$ in the
  complex structures $I_1$, $I_2$ and $I_3$.
\end{itemize}


In this section we construct natural $\BBB$-branes on the moduli space of Higgs bundles supported on the image under $\check{p}$.  We shall construct two different
hyperholomorphic sheaves on this subvariety. Our
constructions depend on the choice of a flat line bundle either on $X$
(yielding a \emph{rank 1} $\BBB$-brane, in Gukov's \cite{gukov} terminology) or on ${C}$ (yielding a \emph{rank $n$} $\BBB$-brane). 

\begin{assumption} \label{ass gamma of maximal order}
From now on, until the end of this section, we will be assuming that the rank $n$ coincides with the order of the unramified cover $p : C \to X$. So, in the notation of Proposition \ref{prop:push-fixed}, $m = n$ and $r = 1$. In this case we say that the cover $p$ is of {\it maximal degree} (note that for fixed rank $n$, the degree of $p:C\to X$ yielding Narasimhan-Ramanan branes in $\M_{X}(n,d)$ is bounded by $n$).
\end{assumption}

Under this assumption ${p}\colon {C} \to X$ is a connected unramified
$n$-cover of $X$, with genus
\begin{equation}\label{eq:genusXgamma}
g({C})=n(g-1)+1
\end{equation}
and 
\[
\check{p}:\M_{{C}}(1,d)\to \M_{X}(n,d)^p
\]
is the finite morphism \eqref{eq definition of check p_gamma}.

Since the moduli space of rank $1$ Higgs bundles is the cotangent bundle of the Jacobian of $C$, we have the natural projection 
\begin{equation}\label{eq:pr1}
\beta : \M_{{C}}(1,d)\cong T^*\mathrm{Jac}^d({C}) \longrightarrow \Jac^d({C}),
\end{equation}
which is $\Gamma$-equivariant.

\begin{remark}
  The support of our $\BBB$-branes will be
  $\mathrm{N}=\M_{X}(n,d)^p\cong \M_{{C}}(1,d)/\Gamma$. Over the generic locus, it is smooth, so it makes sense to speak of a hyperholomorphic sheaf on a dense open subset of  $\mathrm{N}$. Moreover, it will be clear from our
  construction that the branes extend to coherent sheaves (in complex
  structure $I_1$) on $\M_{X}(n,d)$, supported on $\mathrm{N}$.
\end{remark}

Let 
\[(\Ff,\nabla_{\Ff}) \longrightarrow {C}\] 
be a flat line bundle on ${C}$. Since $\pi_1(\Jac^d({C}))$ is the abelianization of $\pi_1({C})$, there is a unique flat line bundle 
\begin{equation}\label{eq:checkFf}
(\widecheck{\Ff}, \check{\nabla}_{\Ff})\longrightarrow\Jac^d({C})
\end{equation}
which restricts to
$(\Ff, \nabla_{\Ff})$ on ${C}\subset \Jac^d({C})$, viewed as subspace  under the Abel--Jacobi map. Define
\begin{equation}\label{eq:Fff}
(\Fff, \nabla_{\Fff})
:=\check{p}_{*}\beta^*(\widecheck{\Ff}, \check{\nabla}_{\Ff}).
\end{equation}
Then $(\Fff, \nabla_{\Fff})$ is a rank $n$ coherent sheaf over
$\M_X(n,d)^p$. Since $\check{p}$ is hyperholomorphic and $\beta^*(\widecheck{\Ff}, \check{\nabla}_{\Ff})$ is flat (thus with curvature trivially of type $(1,1)$ in any complex structure),
it is also hyperholomorphic, and then so is $(\Fff, \nabla_{\Fff})$. Hence the pair
$(\M_X(n,d)^p, (\Fff, \nabla_{\Fff}))$ is a $\BBB$-brane.

\begin{definition}\label{def:nonbasicNR-BBBbrane}
Let $p: C \to X$ be a connected unramified $n$-cover and let $(\Ff,\nabla_{\Ff})$ be a flat
line bundle on ${C}$. The \emph{rank $n$
Narasimhan--Ramanan $\BBB$-brane} is
$$
\BBBB_{\Ff}^p := \left ( \M_X(n,d)^p, (\Fff, \nabla_{\Fff}) \right).
$$ 
We will omit the cover from the name when it is clear from the context.
\end{definition}
The brane just constructed is a rank $n$ brane in Gukov's language
\cite{gukov}.  Next we construct a rank $1$ brane arising from a flat line
bundle on the base curve $X$. Since  $\beta$ in
\eqref{eq:pr1} is $\Gamma$-equivariant and the norm map $\Nm:\Jac^d({C})\to\Jac^d(X)$ of $p$ is $\Gamma$-invariant, their composition is also $\Gamma$-invariant, and we have the following quotient maps (denoted by a bar).
\begin{equation}
\label{eq:def-Delta}
\alpha\colon\M_X(n,d)^p\cong \M_{{C}}(1,d)/\Gamma\xrightarrow{\ol{\beta}}
\Jac^d({C})/\Gamma\xrightarrow{\ol{\Nm}}\Jac^d(X),
\end{equation}
where the isomorphism is given by Theorem~\ref{thm:narasimhan-ramanan:1975}.

\begin{remark}
The map  $\alpha\colon\M_X(n,d)^p\to \Jac^d(X)$ can be interpreted as a twisted determinant map.
Indeed,  if $(E,\varphi)=p_{*}(F,\phi)$ with
$F\in\mathrm{Jac}^d({C})$, then
$\det(E)=\det(p_{*}F)=\Nm(F)\det(p_{*}\Oo_{{C}})$, so
\[
\alpha(E,\varphi)=\det(E)\det(p_{*}\Oo_{{C}})^{-1}.
\]
\end{remark}

Let now 
\((\Ll,\nabla_\Ll) \longrightarrow X\) 
be any flat line bundle on the base curve $X$ and let  
\begin{equation}\label{eq:checkLl}
(\widecheck{\Ll}, \check{\nabla}_\Ll)\longrightarrow\Jac^d(X)
\end{equation}
be the associated flat line bundle as above. Define the flat line bundle on $\M_{X}(n, d)^p$ by
\begin{equation}\label{eq:Lll}
(\Lll, \nabla_\Lll):=\alpha^*(\widecheck{\Ll}, \check{\nabla}_\Ll).
\end{equation} 

\begin{definition}\label{def:basicNR-BBBbrane}
Let $p: C \to X$ be a connected unramified $n$-cover and let $(\Ll,\nabla_\Ll)$ be a flat
line bundle on $X$. The associated \emph{rank $1$
Narasimhan--Ramanan $\BBB$-brane} is
$$
\BBBB_\Ll^p := \left ( \M_{X}(n,d)^p, (\Lll, \nabla_\Lll) \right).
$$ 
\end{definition}

In the remaining part of this section we shall study the restriction of the branes $\BBBB_\Ll^p$ and $\BBBB_{\Ff}^p$ to a certain Hitchin fiber. Recall that $m = n$ so $r = 1$ in our case. Then, $B^{\red}_{{C},1} = B_{C,1}^{\ni} = B^{\sm}_{{C},1} = B_{{C},1} = H^0({C},K_C)$. Therefore $B^p_{\fs} = B^p_{\ni}$ by Proposition \ref{prop:equationspectralcurve}. Also, note that $B_{{C},1}^\fs$ is just the subset of elements with free $\Gamma$-orbits,
\begin{equation} \label{eq df H^gamma_red}
H^0({C}, K_C)^\free := \{ \phi \in H^0({C}, K_C) \textnormal{ such that } \gamma(\phi) \neq \phi \textnormal{ for any } \gamma \in \Gamma \}. 
\end{equation}
Proposition \ref{prop:commutative-rank-1-hitchin} and Theorem \ref{thm:normaliz} imply that, in this particular case, $B^p_{\ni}=H^0(C,K_C)^\free/\Gamma$. Moreover, for $b \in B^p_{\ni}$ such that $b=\zeta(\phi) \in B^p_{\ni}$, for $\phi\in H^0({C}, K_C)^\free$, the spectral curve $X_b$ is singular and integral, with $\deg(\sing(X_b))=\delta$. Furthermore, using the notation of \eqref{eq cartesian diagram spectral},
\begin{equation} \label{eq def nu_phi}
\nu_{\phi} := q_\phi \colon {C} \to X_b
\end{equation}
is a normalization morphism and the following diagram commutes:
\begin{displaymath} 
\xymatrix{
&{C}\ar[dr]^{\nu_{{\phi}}}\ar[dd]_{{p}}&&
\\
&&X_b\ar[dl]^{\pi_b}&
\\
& X.&&
}
\end{displaymath} 
In addition, the same statement holds for $\nu_{\gamma(\phi)}=q_{\gamma(\phi)}=\gamma\circ\nu_\phi:C\to X_b$, for every $\gamma\in\Gamma$.

\begin{remark}
Notice that in this case, namely when the cover $p:C\to X$ has degree $n$, the normalization of the spectral curve $X_b$ with $b$ in $B^p_{\ni}$ is $C_\phi$, the spectral curve for the moduli of rank $1$ Higgs bundles over $C$, associated to $\phi$. But $C_\phi$ is isomorphic to $C$ via the section $\phi:C\xrightarrow{\cong} C_\phi\subset |K_C|$, so all spectral curves share (up to isomorphism) the same normalization $C$. This justifies the slightly different notation for the normalization morphism in \eqref{eq def nu_phi}, when compared with \eqref{eq def nu_a}. We have implicitly used the identification $C_\phi\cong C$ in \eqref{eq def nu_phi}, so that, strictly speaking, $\nu_\phi=q_\phi\circ\phi$.
\end{remark}

In this case, we denote the  pushforward morphism \eqref{eq definition of check nu_a} by 
\begin{equation}\label{eq:pushforwardnormalizr=1}
\check{\nu}_\phi:\Jac^d(C)\hookrightarrow\overline{\Jac}^{\, d+\delta}(X_b),
\end{equation} and Proposition \ref{pr restriction of M^gamma to a Hitchin fiber} \eqref{it h fiber M^p} reads as follows
\begin{equation}\label{eq fiber as quotient of fibers}
h_{X,n}^{-1}(b) \cap \M_X(n,d)^p \cong\bigg(\bigsqcup_{\gamma \in \Gamma}(\Jac^d({C}) \times \{\gamma(\phi)\})\bigg)/\,\Gamma \cong \Jac^d({C})\cong \image(\check{\nu}_\phi)\subset \overline{\Jac}^{\, d+\delta}(X_b),
\end{equation}
with the inverse of the second isomorphism defined by choosing one representative $\phi$ in the  $\Gamma$-orbit given by $b$, and taking
\begin{equation}\label{eq:isomJac^d}
f_\phi(L)=\Gamma\cdot (L,\phi)
\end{equation}
for $L\in\Jac^d({C})$.

For each $\gamma \in \Gamma$, let $\hat{\gamma}:\Jac^d({C})\to\Jac^d({C})$ denote the pullback map associated to the covering automorphism of ${C}$ determined by $\gamma$.
In the following proposition we study how $\Fff$ and $\Lll$, defining the Narasimhan-Ramanan $\BBB$-branes, restrict to a Hitchin fiber in $B^p_{\ni}$.

\begin{proposition} \label{prop spectral data BBB'}
Let $b \in B^p_{\ni}$, and let $\Lll$ and $\Fff$ be the hyperholomorphic bundles defined in \eqref{eq:checkLl} and \eqref{eq:Fff} respectively. The restrictions of  $\Lll$ and of  $\Fff$ to $h_{X,n}^{-1}(b) \cap \M_X(n,d)^p$ are identified, under the isomorphism \eqref{eq fiber as quotient of fibers}, with the bundles $\Nm^*\widecheck{\Ll}$ and  $\bigoplus_{\gamma \in \Gamma} \hat{\gamma}^{*}\widecheck{\Ff}$ respectively.
\end{proposition}

\begin{proof}
Pick some $\phi\in H^0({C}, K_C)^\free$ such that $\zeta(\phi)=b$ and consider the following commutative diagram
\[
\xymatrix{
\Jac^d({C})\cong\Jac^d({C})\times\{\phi\}\ar[d]_{f_\phi}\ar@{^(->}[r]^-{\tilde i}& T^*\Jac^d({C})\cong\M_{{C}}(1,d)\ar[r]^-{\beta}\ar[d]^{\check{p}}&\Jac^d({C})\ar[d]^{\Nm}
\\
h_{X,n}^{-1}(b) \cap \M_X(n,d)^p \ar@{^(->}[r]_-{i}&T^*\Jac^d({C})/\Gamma\cong\M_{X}(n,d)^p\ar[r]_-{\, \alpha}&\Jac^d(X),
}
\]
where left vertical isomorphism is the one given in \eqref{eq fiber as quotient of fibers}, and $i,\tilde i$ are the obvious inclusions. 

By definition, $\Lll=\alpha^*\widecheck{\Ll}$ hence, by commutativity of the diagram and the fact that $\beta\circ\tilde i=\id_{\Jac}$, it follows that $i^*\Lll$ is identified, via the isomorphism \eqref{eq fiber as quotient of fibers}, with $\tilde i^*\beta^*\Nm^*\widecheck{\Ll}=\Nm^*\widecheck{\Ll}$, as claimed.

Recall from \eqref{eq:Fff} that $\Fff=\check{p}_{*}\beta^*\widecheck{\Ff}$. Now, by Proposition~\ref{pr restriction of M^gamma to a Hitchin fiber} we have a commutative diagram
$$
\xymatrix{
\Jac^d({C})\cong\Jac^d({C})\times\{\phi\}\ar[dr]_\cong\ar@{^(->}[r]_-j\ar@/^1.5pc/[rr]^-{\tilde{i}}&\bigsqcup_{\gamma \in \Gamma} \Jac^d({C}) \times \{\gamma(\phi)\}\ar[d]_{\check{p}_b}\ar@{^(->}[r]&\M_{{C}}(1,d)\ar[d]^{\check{p}}\\
&h_{X,n}^{-1}(b) \cap \M_X(n,d)^p\ar@{^(->}[r]_-i&\M_X(n,d)^p.
}
$$
where the diagonal isomorphism is the one given by \eqref{eq fiber as quotient of fibers}, the rightmost square is Cartesian and the central downward arrow is the unramified $\Gamma$-cover $\check{p}_b:\bigsqcup_{\gamma \in \Gamma} \Jac^d({C}) \times \{\gamma(\phi)\}\to h^{-1}_{X,n}(b) \cap \M_X(n,d)$, whose action of the Galois group in the bundle factor is given by the pullback  $\hat{\gamma}$  associated to each  $\gamma\in\Gamma$. 
Under the isomorphism \eqref{eq fiber as quotient of fibers}, $i^*\Fff$ is identified with $\tilde{i}^*\check{p}^*\Fff$ and, since $\beta\circ\tilde i=\id_{\Jac}$,
\[
\tilde{i}^*\check{p}^*\Fff=\tilde{i}^*\check{p}^*\check{p}_{*}\beta^*\widecheck{\Ff}=\tilde{i}^*\bigoplus_{\gamma \in \Gamma} \hat{\gamma}^{*}\beta^*\widecheck{\Ff}=\tilde{i}^*\beta^*\bigoplus_{\gamma \in \Gamma} \hat{\gamma}^{*}\widecheck{\Ff}=\bigoplus_{\gamma \in \Gamma} \hat{\gamma}^{*}\widecheck{\Ff},
\]
completing the proof.
\end{proof}

\begin{remark} \label{rk BBB^p pulls back to Car}
In \cite{Borel}, the first and fourth authors constructed a brane $\mathbf{Car}(\Ll)$ supported on the Cartan locus $\M_{(1,n)} \subset\M_{{C}}(n,nd)$. Recall from Remark \ref{rk Levi and parabolic groups} that $\hat{p} (\M_X(n,d)^p)$ is contained in $\M_{1,n}$. One can check that the hyperholomorphic sheaf in $\mathbf{Car}(\Ll)$ pulls-back under $\hat{p}$ to that of $\BBBB^p_\Ll$. Note that moreover $\mathbf{Car}(\Ll)=:\M_X(n,d)^p$ for the trivial $\ZZ_n$-Galois cover $p:\bigsqcup X\longrightarrow X$. In that case however Theorem \ref{thm:narasimhan-ramanan:1975} fails.
\end{remark}

\section{Narasimhan--Ramanan dual $\BAA$-branes}\label{section:Heckebranes}

We have seen in Section \ref{sc Mm hyperkahler} that $\M_X(n,d)$ is a hyperk\"ahler variety with K\"ahler structures $((I_1, \omega_1), (I_2, \omega_2), (I_3,\omega_3))$. Following \cite{kapustin&witten}, a \emph{$\BAA$-brane} on $\M_X(n,d)$ is, by definition, a pair $(\Sigma, (W, \nabla_W))$, where:
\begin{itemize}
\item $\Sigma$ is a subvariety of $\M_X(n,d)$, which is  a complex Lagrangian for the holomorphic symplectic form $\Omega_1 = \omega_2 + \i \omega_3$.
\item $(W,\nabla_W)$ is a flat bundle supported on $\Sigma$.
\end{itemize}

The purpose of the present section is to construct a natural collection of complex Lagrangian subvarieties, whose image by the Hitchin fibration $h_{X,n}\to B_{X,n}$ is $B^p_\ni$ (cf.\ Definition~\ref{def: reducedlocus}). Each of these subvarieties depends on the unramified cover $p:C \to X$ and on a holomorphic line bundle $\Jj$ on ${C}$, and will henceforth be denoted by $\Hec^{p,\Jj}_\ni$. We will also see that the Higgs bundles lying in $\Hec^{p,\Jj}_\ni$ can be constructed from Hecke modifications of a Hitchin section of $\M_X(n,d+\delta)^p \to B^p$ constructed out of $\Jj$. Up to the choice of a flat bundle on it, $\Hec^{p,\Jj}_\ni$ is thus the support of a $\BAA$-brane, the \emph{Narasimhan--Ramanan dual $\BAA$-brane}. In Section \ref{sec:FM} will be shown that these branes are the (fiberwise) mirror transform of the Narasimhan--Ramanan $\BBB$-branes. This justifies the notation for these subvarieties, as well as their name. 


\subsection{Construction of the support}\label{subsection:constructioninfamilies}

In this section we construct certain subvarieties that later will be shown to be complex Lagrangian, hence the support of a $\BAA$-brane after the choice of a flat bundle on them.

Donagi and Pantev defined in \cite{DP} the abelianized Hecke correspondence. For the general linear group $\GL(n,\CC)$, this correspondence is given by tensorizing the spectral data by the dual of the ideal of a point on the corresponding spectral curve. In the following lines we generalize this correspondence to $0$-dimensional subschemes of length $\ell$ as it constitutes the first ingredient of our construction.

Recall the family of spectral curves $\Xxx \to B_{X,n}$ and consider its restriction to the subset $B^p_\ni \subset B_{X,n}$ that parametrizes nodal curves. The existence of the Jacobian $\Jac_{B^p_\ni}(\Xxx)$ follows by the seminal work of Grothendieck \cite[Thm. 3.1]{FGA}. The existence of the relative compactified Jacobian $\overline{\Jac}_{B^p_\ni}(\Xxx)$ follows by Altman and Kleiman's work \cite[Theorem 3.1]{altman&kleiman}, being a compactification of $\Jac_{B^p_\ni}(\Xxx)$, inducing the fiberwise compactification of the Jacobian by rank one torsion-free sheaves. We define the tensorization morphism 
\[
\morph{\overline{\Jac}_{B^p_\ni}^{d + \delta}(\Xxx) \times_{B^p_\ni} \Hilb^\ell_{B^p_\ni}(\Xxx)}{\overline{\Jac}_{B^p_\ni}^{d + \delta + \ell}(\Xxx)}{(\Ff,\Ii_Z)}{\Ff \otimes \Ii_Z^\vee.}{}{\tT^\ell_{p}}
\]
Since we are considering nodal curves, hence Cohen--Macaulay, the sheaves $\Ii_Z$ are Cohen--Macaulay too. In that case the dual sheaves $\Ii_Z^\vee$ are Cohen--Macaulay as well, so torsion--free as we are working on curves. Furthermore, one can see that the $\Ii_Z^\vee$ are semistable (hence stable since we our spectral curves are integral). It follows that $\Ff \otimes \Ii_Z^\vee$ is semistable and torsion--free provided that $\Ff$ is, proving that $\tT^\ell_{p}$ is well defined.

\begin{remark} \label{rm preimage of tT}
Recall the Abel--Jacobi map $\alpha_\ell : \Hilb^\ell(X_b) \to \overline{\Jac}^{\, \ell}(X_b)$, $\Ii_Z \mapsto \Ii_Z^\vee$. The preimage $\alpha_\ell^{-1}(\Ff')$ is identified with the open set of $\PP(H^0(X_b, \Ff'))$ of injective sections and the identification is done by considering the short exact sequence $0 \to \Oo_{X_b} \stackrel{s}{\to} \Ff' \to \Oo_Z \to 0$. It then follows that the preimage $(\tT^\ell_p)^{-1}(\Ff'')$ consists on the pairs $(\Ff,\Ii_Z)$ fitting in the short exact sequence
\[
0 \to \Ff \stackrel{s \otimes \id_\Ff}{\to} \Ff'' \to \Oo_Z \to 0.
\]
\end{remark}

Recall the isomorphism \eqref{eq spectral isomorphism} provided by the spectral correspondence. One can define the {\it $p$-adapted length $\ell$ abelianized Hecke morphism} by setting the morphism
\[
\hH^\ell_p := S_{X,n} \circ \tT^\ell_p \circ (S_{X,n}^{-1} \times \id_{\Hilb}),
\]
\[
\hH^\ell_p : \M_X(n,d) \times_{B_{X,n}} \Hilb^\ell_{B^p_\ni}(\Xxx) \longrightarrow \M_X(n,d+\ell) \times_{B_{X,n}} B^p_\ni.
\]

The second ingredient of the construction of the Narasimhan--Ramanan dual branes is a section of the structural morphisms of the relative Hilbert scheme $\Hilb^{\delta-\rho}_{B^p_{\ni}}(\Xxx) \to B^p_\ni$ given by the singularities of the spectral curves. Recall from \eqref{Xb singular} Theorem \ref{thm:normaliz} that for every $b \in B^p_{\ni}$, associated to the spectral curve $X_b = \Xxx|_b$, one has that the singularity divisor $\sing(X_b)$ is reduced and has length $\delta - \rho$. Hence, one can construct the section
\[
\morph{B^p_\ni}{\Hilb^{\delta-\rho}_{B^p_\ni}(\Xxx)}{b}{\sing(X_b).}{}{\sing^{\, p}_\ni}  
\]
We denote the image of this section by
\[
\Sing^{\, p}_\ni := \image \left ( \sing^{\, p}_\ni \right ).
\]
One obviously have a natural isomorphism provided by the structural projection
\begin{equation} \label{eq isomorphism provided by Sing}
\M_X(n,d) \times_{B_{X,n}} \Sing^{\, p}_\ni \cong \M_X(n,d) \times_{B_{X,n}} B^p_\ni.
\end{equation}

The third ingredient of the construction of the Narasimhan--Ramanan dual brane is a Hitchin section on $\M_{C}(r,d + \delta - \rho)$. Such a section will only exists under certain conditions.

\begin{assumption} \label{ass d multiple of r}
From now on, until the end of Section~\ref{sec:FM}, we will be assuming that $d$ is a multiple of $r$, hence $d = r d'$ for some integer $d'$.
\end{assumption}

\begin{remark}
Note that Assumption \ref{ass gamma of maximal order} forces $r = 1$, so it implies Assumption \ref{ass d multiple of r}. 
\end{remark}

As we have anticipated, the reason behind requiring Assumption \ref{ass d multiple of r} is because in this case one can construct a Hitchin section associated to any $\Jj \in \Jac^{\delta/r + d'}({C})$,
\[
\sigma_{C,\Jj} : B_{C,r} \to \M_{C}(r,rd' + \delta - \rho).
\]
Recall \eqref{eq s_F} and \eqref{eq spectral correspondence for varphi} from Section \ref{sec:the Hitchin system}, 
and compose the Hitchin section with the pushforward map $\check{p}$,
\[
\check{p} \circ \sigma_{C,\Jj} : B_{C,r} \to \M_{X}(n,rd' + \delta - \rho)^p.
\]
Thanks to Proposition \ref{prop:commutative-rank-1-hitchin}, $\check{p} \circ \sigma_{C,\Jj}$ factors through a Hitchin section for $\M_X(n,rd' + \delta - \rho)^p$,
\[
\sigma_{p,\Jj} : B^p \to \M_X(n, rd' + \delta -\rho)^p, 
\]
satisfying $\sigma_{p,\Jj} \circ \zeta = \check{p} \circ \sigma_{C,\Jj}$.

As before, we use capital letters to denote their image,
\[
\Sigma_{C,\Jj} := \image \left ( \sigma_{C,\Jj} \right ) \subset \M_{C}(r,rd' + \delta - \rho)
\]
and
\[
\Sigma_{p,\Jj} := \image \left ( \sigma_{p,\Jj} \right ) \subset \M_{X}(n,rd' + \delta - \rho)^p.
\]
Note that for every $\gamma \in \Gamma$ one has that $\sigma_{p,\Jj} = \sigma_{p,\gamma^*\Jj}$, hence
\[
\Sigma_{p,\Jj} = \Sigma_{p,\gamma^*\Jj} 
\]
only depends on the class of $\Jj$ under the action of $\Gamma$ by pull-back.

\begin{definition}\label{def:Heckesubvar}
For each $[\Jj]_\Gamma \in \Jac^{\delta/r + d'}({C})/\Gamma$, define the subvariety $\Hec^{p,\Jj}_\ni$ of $\M_X(n,rd')$, closed in $\M_X(n,rd') \times_{B_{X,n}} B^p_\ni$, as the image under under the isomorphism \eqref{eq isomorphism provided by Sing} of the restriction to $\M_X(n,d) \times_{B^p_\ni} \Sing^{\, p}_\ni$ of the preimage of $\Sigma_{p,\Jj} = \check{p}(\Sigma_{C,\Jj})$ under the $p$-adapted length $(\delta-\rho)$ abelianized Hecke morphism, {\it i.e.}
\[
\Hec^{p,\Jj}_\ni := \left ( \left ( \hH^{(\delta - \rho)}_p \right )^{-1} (\Sigma_{p,\Jj}) \right ) \cap \left ( \M_X(n,rd')^p_\ni \times_{B_{X,n}} \Sing^{\, p}_\ni \right ).
\]
\end{definition}

\begin{remark}
The justification for the notation used for this subvariety will be clear from Theorem \ref{tm duality} below, where we show that $\Hec^{p,\Jj}_\ni$ supports the $\BAA$-brane dual to the Narasimhan--Ramanan $\BBB$-brane.  
\end{remark}

We next turn to the description of the spectral data of $\Hec^{p,\Jj}_\ni$ in terms of Hecke modifications of suitable Higgs bundles. These have been considered in several works, such as in \cite{hitchin:2017,hwang-ramanan:2004,ramanan:2010,wilkin:2016,witten:2015}.  Let $D$ be a reduced effective divisor on $X$, $E\to X$ a vector bundle and $\alpha_y\in E_y^*$ for each point $y$ in the support of $D$. This defines
\[0\to E'\to E\xrightarrow{\alpha} \Oo_D\to 0\]
where $E'$ depends on $D$ and on the projective class of $\alpha$. The bundle $E'$ is said to be a \emph{Hecke modification of $E$} (along $D$ and associated to $\alpha$). If $E$ is equipped with a Higgs field $\psi$, then a \emph{Hecke modification of the Higgs bundle $(E,\psi)$} is a Higgs bundle $(E',\varphi)$ where $E'$ is a Hecke modification of $E$ which is compatible with $\psi$ and $\varphi$, {\it i.e.} such that the restriction of $\psi$ to $E'$ equals $\varphi$.
 
The following theorem gives a description of $\Hec^{p,\Jj}_\ni$ in terms of Hecke modifications of the Higgs bundles parametrized by the section $\sigma_{p,\Jj}$. Note that $\pi_b(\sing(X_b))$ is a reduced effective divisor of length $\delta - \rho$, whenever $b \in B^p_\ni$.

\begin{theorem}\label{pr Hecke}
Let $(E,\varphi)\in\M_X(n,d)$, with $d=rd'$, and let $b\in B^p_\ni$. The following conditions are equivalent:
\begin{enumerate}
\item $(E,\varphi)\in\Hec^{p,\Jj}_\ni\cap h_{X,n}^{-1}(b)$;
\item $(E,\varphi)$ is a Hecke modification of the Higgs bundle 
\[\left (E_{\Jj,b},\psi_{\Jj,b} \right ) := \sigma_{p,\Jj}(b)\in\M_X(n,rd'+\delta-\rho)^p\] along the divisor $\pi_b(\sing(X_b))$ on $X$.
\end{enumerate}
\end{theorem}

\begin{proof}
Choose $a \in B^\ni_{C,r}$ such that $b = \zeta(a)$. By construction, we have $E_{\Jj,b} \cong p_*\eta_{a,*}\eta_a^*\Jj$. Recalling the commuting diagram \eqref{eq normalization X b} (with $\gamma=1$), one has 
\begin{equation} \label{eq description of E_Jj}
E_{\Jj,b} \cong \pi_{b,*}\nu_{a,*}\eta_a^*\Jj.
\end{equation}
As for the Higgs field, by definition (cf. \eqref{eq spectral correspondence for varphi}),
\[\psi_{\Jj, b} \cong p_* \eta_{a,*} \mu_{\eta_a^*\Jj}:E_{\Jj,b}\to E_{\Jj,b}\otimes K_X.\] Then, $\psi_{\Jj, b} \cong \pi_{b,*} \nu_{a,*} \mu_{\eta_a^*\Jj}$ thanks to the commutativity of \eqref{eq normalization X b} (with $\gamma=1$). Recall also that $\mu_{\eta_a^*\Jj}$ is defined by the tensorization under the tautological section $\hat\lambda: C_a \to \eta_a^*K_C$, 
\[
\mu_{\eta_a^*\Jj} :  \eta_a^*\Jj \xrightarrow{\otimes\hat\lambda}  \eta_a^*\Jj \otimes \eta_a^*K_C.
\]

Let $\Ff \to X_b$ be the spectral datum of $(E, \varphi)$. Recalling Remark \ref{rm preimage of tT} the first condition is equivalent to $\Ff$ fitting in the exact sequence 
\begin{equation} \label{eq sequence spectral data}
0\to\Ff\to \nu_{a,\ast} \eta_a^*\Jj \to \Oo_{\sing(X_b)}\to 0,
\end{equation}
for any choice of $a \in B^\ni_{C,r}$ such that $b = \zeta(a)$.

By definition, the second condition above is equivalent to saying that $E$ is given by the short exact sequence
\begin{equation} \label{eq hecke transform}
0 \longrightarrow E \longrightarrow E_{\Jj, b} \longrightarrow \Oo_{\pi_b(\sing(X_b))} \longrightarrow 0,
\end{equation}
and the Higgs field is obtained by the restriction, {\it i.e.}\
\begin{equation} \label{eq hecke higgs field}
\varphi = \psi_{\Jj, b}|_{E}
\end{equation}

We can easily see that the first condition implies the second. Since $\pi_b$ is a finite morphism, pushing forward the exact sequence \eqref{eq sequence spectral data} to $X$ yields the exact sequence \eqref{eq hecke transform} after the identification \eqref{eq description of E_Jj}, hence a Hecke modification of vector bundles.

Recall that $\hat\lambda = q_a^*\lambda$, where $\lambda$ is the tautological section of $\pi_b^*K_X$, and that $K_C\cong p^*K_X$. Since $p\circ\eta_a=\pi_b\circ q_a$ and since $q_a = \nu_a$ by definition, then
\[
\mu_{\eta_a^*\Jj} :  \eta_a^*\Jj \xrightarrow{\otimes\hat\lambda}  \eta_a^*\Jj \otimes \nu_a^*\pi_b^*K_X.
\]
Then, by the projection formula, one has that 
\[
\nu_{a,*} \mu_{\eta_a^*\Jj} = \mu_{(\nu_{a,*}\eta_a^*\Jj)}:\nu_{a,*}\eta_a^*\Jj\xrightarrow{\otimes\lambda}\nu_{a,*}\eta_a^*\Jj\otimes\pi_b^*K_X 
\]
on $X_b$, so
\begin{equation}\label{eq identification of psi}
\psi_{\Jj, b} = \pi_{b,*} \mu_{(\nu_{a,*}\eta_a^*\Jj)}.
\end{equation}
Since tensorization by $\lambda$ restricts to subsheaves and $\Ff$ is a subsheaf of $\nu_{a,*}\eta_a^*\Jj$, it is clear that 
\[
\mu_{\Ff} = \mu_{(\nu_{a,*}\eta_a^*\Jj)}|_{\Ff}.
\]
Taking the pushforward under $\pi_b$ yields $\varphi=\psi_{\Jj,b}|_E$ by \eqref{eq identification of psi}, proving \eqref{eq hecke higgs field}.

For the converse statement, suppose $(E,\varphi)$ is a Hecke modification of $(E_{\Jj,b}, \psi_{\Jj,b})$, so that we have \eqref{eq hecke transform} and \eqref{eq hecke higgs field}. If $\Ff \to X_b$ is again the spectral datum, then \eqref{eq hecke transform} is the same as  
\[
0 \longrightarrow \pi_{b,*}\Ff \longrightarrow \pi_{b,*}\nu_{a,*} \eta_a^*\Jj \longrightarrow \Oo_{\pi_b(\sing(X_b))} \longrightarrow 0.
\]
Recall from \cite{BNR, simpson2} that the $\pi_{b,\ast}\Oo_{X_b}$-module structure on $E$ and on $E_{\Jj,b}$ is precisely given by the corresponding Higgs fields, then \eqref{eq hecke higgs field} implies that this an exact sequence of $\pi_{b,*}\Oo_{X_b}$-modules.

Since $\pi_b$ is a finite morphism, $\pi_{b,*}$ is an exact functor (as the higher direct image sheaves vanish because its fibers are zero dimensional) from the category of $\Oo_{X_b}$-modules to the category of $\pi_{b,*}\Oo_{X_b}$-modules. Moreover, since $\pi_b$ is affine, this is an equivalence of categories (cf.  \cite[ex. 5.17 p.128]{hartshorne:1977}, so the previous sequence holds if and only if we have that \eqref{eq sequence spectral data} holds and hence $(E,\varphi)\in\Hec^{p,\Jj}_\ni\cap h_{X,n}^{-1}(b)$.
\end{proof}

We finish the section showing that $\Hec^{p,\Jj}_\ni$ lies in the smooth locus.

\begin{proposition} \label{prop points Lag are stable}
The subvariety $\Hec^{p,\Jj}_\ni$ is contained in the stable locus of $\M_X(n,rd')$.
\end{proposition}
\begin{proof}
$\Hec^{p,\Jj}_\ni$ maps to $B^p_\ni$ under the Hitchin map, hence every Higgs bundle in $\Hec^{p,\Jj}_\ni$ corresponds to an irreducible (and reduced) spectral curve, by \eqref{Xb irreducible} of Theorem~\ref{thm:normaliz}. But a destabilizing Higgs subbundle of a strictly polystable Higgs bundle gives rise to a proper component of the corresponding spectral curve, which hence is not irreducible.
\end{proof}




\subsection{Spectral data and parabolic modules} \label{sc spectral dara for Heche}

We shall now provide the spectral description of the subvarieties $\Hec^{p,\Jj}_\ni$. The use of \emph{parabolic modules} (cf. \cite{rego:1980,cook:1993,cook:1998, bhosle:1992}) will be crucial along this section, as they provide a convenient way to relate rank one torsion-free sheaves on a singular curve with line bundles over its normalization. 

Given a nodal curve $X_b$, where $b = \zeta(a)\in B^p_\ni$ for some $a \in B^\ni_{C,r}$, its normalization is $C_a$ with $\nu_a:C_a\to X_b$ being the normalization morphism, as described in \eqref{eq def nu_a}. One can consider the pull-back morphism, 
\begin{equation} \label{eq definition of hat nu}
\morph{\Jac^{\,d+\delta}(X_b)}{\Jac^{d+\delta}({C}_{a})}{L}{\nu_{a}^*L.}{}{\hat{\nu}_{a}}
\end{equation}
Note that $\hat{\nu}_a$ does not extend to $\overline{\Jac}^{\,d+\delta}(X_b)$. One of the motivations to introduce parabolic modules is that their moduli space is a compactification of $\Jac^{\,d+\delta}(X_b)$ (different to $\overline{\Jac}^{\,d+\delta}(X_b)$) where $\hat{\nu}_a$ extends naturally.

Consider
\[
\widetilde{D}_a:=\nu_a^{-1}(\sing(X_b)).
\]
Decompose the singular divisor as $\sing(X_b) = D_b^1 + \dots + D_b^s$, in such a way that each subdivisor $D_b^i$ is supported on the reduced point $x_i \in \sing(X_b)$, with $x_i \neq x_j$. This induces the decomposition $\widetilde{D}_a = \widetilde{D}_a^1 + \dots + \widetilde{D}_a^s$, where each subdivisor is $\widetilde{D}_a^i := \nu_a^{-1}(D_a^i)$.

\begin{definition}\label{def par mod}
A (rank $1$) \emph{parabolic module} over ${C}_a$ associated to $\widetilde{D}_a$, of degree $d+\delta$ and type $\ell = (\ell_1, \dots, \ell_s)$, is a pair $(M,V)$ where $M\in \Jac^{d+\delta}(C_a)$ and $V$ is a vector subspace of $M\otimes\Oo_{\widetilde{D}_a}$ such that: 
\begin{enumerate}
\item $V$ is $\bigoplus_{i = 1}^s V^i$ with $V^i \subset (M \otimes \Oo_{\widetilde{D}^i_a})$;
\item for every $i$, the vector space $V^i$ has dimension $\ell_i>0$;
\item$V^i$ is an $\Oo_{x_i}$-submodule of $M\otimes\Oo_{\widetilde{D}^i_a}$ via pushforward under $\nu_a$ {\it i.e.} via the inclusion $\Oo_{X_b}\hookrightarrow\nu_{a,*}\Oo_{{C}_{ a}}$.
\end{enumerate} 
See \cite{cook:1993,cook:1998} for more details.
\end{definition}

Write $\PMod^{d+\delta}_{\ell}({C}_{ a},\widetilde{D}_a)$ for the moduli space of parabolic modules over ${C}_{ a}$ associated to $\widetilde{D}_a$, of degree $d+\delta$ and type $\ell$.  It is an integral and projective variety (see \cite{cook:1993}).

Recall that we have defined $B_{C,r}^\ni$ as the open subset of $B_{C,r}^{\fs}$ such that the intersection of $C_a$ with any $C_{\gamma(a)}$ as only nodal singularities.

\begin{proposition}
Take $a \in B_{C,r}^\ni$, then $\deg(\widetilde{D}_a)= 2 \deg(\sing (X_b))$ and $\widetilde{D}_a \longrightarrow \sing(X_b)$ is a $2:1$ cover.
\end{proposition}

\begin{proof}
Since $a$ is chosen in $B_{C,r}^\ni$, it follows that $C_{\tilde{b}}$ has only nodal singularities given by the intersection of two irreducible components. Since $q_{\tilde{b}}$ is a connected unramified cover by Proposition \ref{pr cartesian diagram with X_wtb} \eqref{it q is unramified} and the normalization $\nu_a$ is the restriction of $q_{\tilde{b}}$ to the irreducible component $C_a$ of $C_{\tilde{b}}$, one has that $\widetilde{D}_a = \nu_a^{-1}(\sing(X_b))$ can be described as
\[
\widetilde{D}_a = C_a \cap \sing(C_{\tilde{b}})
\]
and so we have that $\widetilde{D}_a$ is the union of the nodal intersections ${C}_{ a} \cap {C}_{ \gamma(a)}$, for all $\gamma \in \Gamma$. Then, it follows from \eqref{eq intersection of components X_gamma,a} that 
\[
\deg(\widetilde{D}_a)= 2nr(m-1)(g-1),
\]
which is twice $\deg(\sing (X_b))$. Consider now a singular point $x$ of $X_b$. If $y \in C_a \cap C_{\gamma(a)}$ maps under $q_{\tilde{b}}$ to $x$ so does $\gamma^{-1}(y) \in C_{\gamma^{-1}(a)} \cap C_a$. Since the action of $\Gamma$ is free in $C$, we conclude that $y \neq \gamma^{-1}(y)$. If there is another $\gamma' \neq \gamma$ in $\Gamma$ such that $(\gamma')^{-1}(y) \in C_a$ then $y \in C_{\gamma'(a)}$ as well, so $y \in C_a \cap C_{\gamma(a)} \cap C_{\gamma'(a)}$ is a triple intersection which it is excluded as $a \in B_{C,r}^\ni$. We conclude that among the nodal singularities of ${C}_{ \tilde{b}}$ mapping to $x$ under $q_{\tilde{b}}$, there are exactly two which lie in ${C}_{a}$, so the projection $\widetilde{D}_a \longrightarrow \sing(X_b)$ is a $2$-cover.
\end{proof}

    For the rest of the section we take $a \in B_{C,r}^\ni$ and $b = \zeta(a) \in B^p_\ni$, so $X_b$ is irreducible with nodal singularities by Theorem~\ref{thm:normaliz}. Observe that in this case $\deg(D^i_a) = 1$, $\deg(\widetilde{D}_a^i)= 2$ and $s = \deg(\sing(X_b))$. Set also $\ell = (1, \dots, 1)$. We are under the assumptions considered in \cite{cook:1998}, so, by Theorem 1 in loc. cit., there is a finite morphism
\begin{equation} \label{eq definition of tau}
\morph{\PMod^{d+\delta}_{\ell}({C}_{ a},\widetilde{D}_a)}{\overline{\Jac}^{\, d+\delta}(X_b),}{(M,V)}{\Ff}{}{\tau}
\end{equation}
where $\Ff$ is defined by the short exact sequence
\begin{equation} \label{eq preimage of tau}
0\to\Ff\to \nu_{a,\ast}M \to \nu_{a,\ast}\big(M\otimes\Oo_{\widetilde{D}_a}/V\big)\to 0,
\end{equation}
and the second map is the composition $\nu_{a,\ast}M\to\nu_{a,\ast}\big(M\otimes\Oo_{\widetilde{D}_a}\big)\to\nu_{a,\ast}\big(M\otimes\Oo_{\widetilde{D}_a}/V\big)$. Note that $\deg(\nu_{a,*}(M \otimes \Oo_{\widetilde{D}^i_a})/V^i) = 1$, so 
\begin{equation} \label{eq sing of X_b from par mods}
\nu_{a,\ast}\big(M\otimes\Oo_{\widetilde{D}_a}/V\big) \cong \Oo_{\sing(X_b)}.
\end{equation}
By definition of $V$, the quotient $M\otimes\Oo_{\widetilde{D}_a}/V$ is an $\Oo_{X_b}$-module, so $\Ff$ inherits an $\Oo_{X_b}$-module structure as well. In addition, $\deg(\nu_{a,\ast}M)=d+\delta + \deg(\sing(X_b))$ and $\deg\big(\nu_{a,\ast}\big(M\otimes\Oo_{\widetilde{D}_a}/V\big)\big)= \deg(\sing(X_b))$, thus indeed $\deg(\Ff)=d+\delta$. 

Let $\tau_0$ denote the restriction of $\tau$ to $\tau^{-1}(\Jac^{d+\delta}(X_b))$. From \cite[Theorem 1]{cook:1998} we know that it is an isomorphism 
\begin{equation}\label{eq:JasobianinParMod}
\tau_0:\tau^{-1}(\Jac^{d+\delta}(X_b))\xrightarrow{\ \cong\ }\Jac^{d+\delta}(X_b),
\end{equation}
so $\Jac^{d+\delta}(X_b)$ can be seen as a dense open subspace of $\PMod^{d+\delta}_{\ell}({C}_{ a},\widetilde{D}_a)$ via $\tau_0^{-1}$. In other words, $ \PMod^{d+\delta}_{\ell}({C}_{ a},\widetilde{D}_a)$ is a compactification of $\Jac^{d+\delta}(X_b)$, which is different from $\overline{\Jac}^{\, d+\delta}(X_b)$.

\begin{lemma} \label{lm tau surjective}
The morphism $\tau:\PMod^{d+\delta}_{\ell}({C}_{ a},\widetilde{D}_a)\to\overline{\Jac}^{\, d+\delta}(X_b)$ is surjective.
\end{lemma}
\begin{proof}
It follows from \cite{AIK} that $\overline{\Jac}^{\, d+\delta}(X_b)$ is irreducible. The restriction $\tau_0$ is surjective, hence the lemma is an immediate consequence of the compactness of $\PMod^{d+\delta}_{\ell}({C}_{ a},\widetilde{D}_a)$ and irreducibility of $\overline{\Jac}^{\, d+\delta}(X_b)$.
\end{proof}

Consider the projection onto the first factor,
\begin{equation} \label{eq definition of dotnu}
\morph{\PMod^{d+\delta}_{\ell}({C}_{ a},\widetilde{D}_a)}{\Jac^{d+\delta}({C}_{ a})}{(M,V)}{M.}{}{\dot{\nu}_{a}}
\end{equation}
This is a fiber bundle, with projective fibers given by products of  closed subschemes of Grassmanians \cite{cook:1993}. The restriction of this morphism to $\Jac^{d+\delta}(X_b)\subset \PMod^{d+\delta}_{\ell}({C}_{ a},\widetilde{D}_a)$  coincides with $\hat\nu_a$ from \eqref{eq definition of hat nu}, {\it i.e}. the diagram
\begin{equation} \label{eq hatnu and dotnu commute}
\xymatrix{
\Jac^{d+\delta}(X_b) \ar@{^(->}[d]_-{\tau_0^{-1}} \ar[rd]^-{\hat\nu_a}\\
\PMod^{d+\delta}_{\ell}({C}_{ a},\widetilde{D}_a) \ar[r]^-{\dot\nu_a} & \Jac^{d+\delta}({C}_{ a}) 
}
\end{equation}
commutes.  
So $\PMod^{d+\delta}_{\ell}({C}_{ a},\widetilde{D}_a)$ is a compactification of $\Jac^{d+\delta}(X_b)$ to which the pullback map $\hat\nu_a$ extends, in contrast to what happens with $\overline{\Jac}^{\, d+\delta}(X_b)$. The next lemma relates the closure of the fiber of $\hat\nu_a$ with the fiber of $\dot\nu_a$.

\begin{lemma}\label{lemma:fiberHecke}
Let $M\in\Jac^{d+\delta}({C}_{ a})$. Then 
\begin{equation} \label{eq tau nudot -1 is the closure}
\overline{\hat\nu_a^{-1}(M)}=\tau(\dot{\nu}_{a}^{-1}(M)),
\end{equation}
and it classifies those $\Ff \in \overline{\Jac}^{\, d+\delta}(X_b)$ fitting in the short exact sequence
\begin{equation} \label{eq description of tau of the preimage of nu dot}
0\to\Ff\to \nu_{a,\ast}M \to \Oo_{\sing(X_b)}\to 0.
\end{equation}
Furthermore, the restriction of $\tau$ to $\dot{\nu}_{a}^{-1}(M)$ is a closed embedding.
\end{lemma}

\begin{proof}
First of all, the image under $\tau$ of the fiber of $\dot{\nu}_a$ over $M$ is given by $\Ff$ fitting in \eqref{eq preimage of tau}. Then, it follows from \eqref{eq sing of X_b from par mods} that $\tau(\dot{\nu}_{a}^{-1}(M))$ is given by those $\Ff$ fitting in \eqref{eq description of tau of the preimage of nu dot}.

We now address \eqref{eq tau nudot -1 is the closure}. The commutative diagram \eqref{eq hatnu and dotnu commute} can be completed as 
\[\xymatrix{ \hat\nu_a^{-1}(M)\ar@{^(->}[r]\ar@{^(->}[d]_{\tau_0^{-1}}&\Jac^{d+\delta}(X_b).\ar@{^(->}[d]_{\tau_0^{-1}}\ar[dr]^-{\hat\nu_a}& \\
\dot{\nu}_{a}^{-1}(M)\ar@{^(->}[r]& \PMod^{d+\delta}_{\ell}({C}_{ a},\widetilde{D}_a)\ar[d]_-{\tau}\ar[r]^-{\dot{\nu}_{a}}&\Jac^{d+\delta}({C}_{ a})\\
\overline{\hat\nu_a^{-1}(M)}\ar@{^(->}[r]&\overline{\Jac}^{\, d+\delta}(X_b)&
}\]
Of course, $\tau\tau_0^{-1}$ is the identity on $\Jac^{d+\delta}(X_b)$. 
Now, since $\tau$ is a closed morphism, $\tau(\dot{\nu}_{a}^{-1}(M))$ is closed in $\overline{\Jac}^{\, d+\delta}(X_b)$; clearly, it contains $\tau(\tau_0^{-1}(\hat\nu_a^{-1}(M)))=\hat\nu_a^{-1}(M)$, so $\ol{\hat\nu_a^{-1}(M)}\subset\tau(\dot{\nu}_{a}^{-1}(M))$.

 Conversely, given any parabolic module $(M,V)\in\dot{\nu}_{a}^{-1}(M)$, we know that there is a sequence of elements $(p_j)_j$ in $\Jac^{d+\delta}(X_b)$ such that its image under $\tau_0^{-1}$ converges to $(M,V)$. We can find such a sequence in $\tau_0^{-1}(\hat\nu_a^{-1}(M))$ as follows (so that $\PMod^{d+\delta}_{\ell}({C}_{ a},\widetilde{D}_a)$ is actually fiberwise compactification of $\Jac^{d+\delta}(X_b)$). The pullback map $\hat\nu_b:\Jac^{d+\delta}(X_b)\to\Jac^{d+\delta}(C_a)$ is a locally trivial fibration, with the fiber $F$ being isomorphic to products of powers of $\CC^*$ and of $\CC$. Trivialize this fibration on an open set $U$ around $M\in \Jac^{d+\delta}(C_a)$, hence becoming a product $U\times F$. On this product, project the sequence $(p_j)_j$ onto a sequence $(q_j)_j$ on $\{M\}\times F$ ({\it i.e.}  via the map $U\times F\to \{M\}\times F$, $(L,z)\mapsto (M,z)$). So $(q_j)_j$ is a sequence on $\hat\nu_b^{-1}(M)$ and, since \eqref{eq hatnu and dotnu commute} commutes, $(\tau_0^{-1}(q_j))_j$ is a sequence in $\dot{\nu}_{a}^{-1}(M)$, which  converges to the same point as $(\tau_0^{-1}(p_j))_j$, namely $(M,V)$. Now, taking the image under the closed morphism $\tau$, we see that that $\tau(M,V)=\lim q_j$, thus $\tau(M,V)\in \overline{\hat\nu_a^{-1}(M)}$, and therefore $\ol{\hat\nu_a^{-1}(M)}\supset\tau(\dot{\nu}_{a}^{-1}(M))$.
 
It remains to prove the last claim {\it i.e.} to prove that $\tau$ restricted to $\dot{\nu}_{a}^{-1}(M)$ is injective. Take $(M,V)$ and $(M,V')$ in $\dot{\nu}_{a}^{-1}(M)$, both mapping to $\Ff$ under $\tau$. It follows that both $V$ and $V'$ must be isomorphic to $\Ff\otimes\Oo_{\sing(X_b)}$, hence $(M,V)\cong (M,V')$.
\end{proof}

\begin{remark}\label{rmk closures intersect}
The previous lemma shows that, for a fixed $M$, we can identify $\overline{\hat\nu_a^{-1}(M)}$ and $\dot{\nu}_{a}^{-1}(M)$ via $\tau$. However, the map $\tau$ is not generally injective, as there are $M\ncong M'$ such that $\tau(\dot\nu_a^{-1}(M))\cap \tau(\dot\nu_a^{-1}(M'))$ is non-empty, that is, the \emph{closures} of $\hat\nu_a^{-1}(M)$ and of $\hat\nu_a^{-1}(M')$ will intersect; cf. \cite[Example 5.4]{gothen-oliveira:2013} for an example of this phenomenon for $A$-type singularities.
\end{remark}

We finish this section describing the restriction of the Hitchin map to $\Hec^{p,\Jj}_\ni$.

\begin{theorem}\label{tm:spectdataHec}
Let $b \in B^p_\ni$ and $a \in B^\ni_{C,r}$ such that $b = \zeta(a)$, then 
\begin{equation} \label{eq description of h_b restricted to Hek}
h_{X,n}^{-1}(b) \cap \Hec^{p,\Jj}_\ni = \bigcup_{\gamma^*\Jj \in \Gamma(\Jj)} S_{X,n} \left( \overline{\hat{\nu}_a^{-1}(\eta_a^* (\gamma^*\Jj))  }  \right ),
\end{equation}
where $\Gamma(\Jj)$ denotes the $\Gamma$-orbit of $\Jj$. Furthermore, 
\begin{equation} \label{eq dim Hitchin fiber cap Hecke}
\dim \left ( h_{X,n}^{-1}(b) \cap \Hec^{p,\Jj}_\ni \right ) = \delta - \rho = n(n-r)(g-1).
\end{equation}
\end{theorem}

\begin{proof}
As we have seen in the proof of Theorem \ref{pr Hecke}, the points $h_{X,n}^{-1}(b) \cap \Hec^{p,\Jj}_\ni$ are those Higgs bundles whose spectral data $\Ff$ fit in the exact sequence \eqref{eq sequence spectral data} for any choice of $a \in B^\ni_{C,r}$ such that $b = \zeta(a)$. Observing that $\eta_{\gamma(a)}^*\Jj = \eta_{a}^* (\gamma^* \Jj)$, if we fix $a \in B^\ni_{C,r}$, we can understand $h_{X,n}^{-1}(b) \cap \Hec^{p,\Jj}_\ni$ as the set of $\Ff$ fitting in
\[
0\to\Ff\to \nu_{a,\ast} \eta_a^*\gamma^* \Jj \to \Oo_{\sing(X_b)}\to 0,
\]
for every $\gamma \in \Gamma(\Jj)$. Then, the first statement follows from Lemma \ref{lemma:fiberHecke}. It then follows that
\[
\dim \left ( h_{X,n}^{-1}(b) \cap \Hec^{p,\Jj}_\fs \right ) = \dim \left ( \overline{\hat{\nu}_a^{-1}(\eta_{a}^*\Jj)} \right ) = \dim \left ( \hat{\nu}_{a}^{-1}(\eta_{a}^*\Jj) \right ).
\]

We can find in \cite{EGA} a description of the fibers of $\hat\nu_a$, being a torsor for $H^0(X_b, \Oo^*_{\sing(X_b)})$. Since,
\[
\dim \left ( H^0(X_b, \Oo_{\sing(X_b)}) \right ) = \deg(\sing(X_b)),
\]
the second statement follows from \eqref{Xb singular} of Theorem \ref{thm:normaliz}.
\end{proof}

\begin{remark}\label{rmk:Jdescends}
In particular, if $\Jj$ is pulled back from a line bundle over $X$ then $\Gamma(\Jj)=\{ \Jj \}$, thus \eqref{eq description of h_b restricted to Hek} just becomes $S_{X,n} \left ( \overline{\hat{\nu}_{a}^{-1}(\Jj)}\right )$. Notice that this condition on $\Jj\in\Jac^{\delta/r + d'}({C})$ is only possible if $d'$ is a multiple of $m$. 
\end{remark}

\subsection{A complex Lagrangian subvariety}\label{subsec:Lagrangian}

Keeping $\Jj\in\Jac^{\delta/r+d'}({C})$ fixed, we now study some properties of $\Hec^{p,\Jj}_\ni$. Particularly relevant is the proof that $\Hec^{p,\Jj}_\ni$ is a complex Lagrangian subvariety of $\M_X(n,rd')$.

Recall that $\B$ denotes the Borel subgroup of $\GL(n,\CC)$ and $\U$ its unipotent subgroup. Fix a square root $K_C^{1/2}$ of $K_C$. We now study the relation of our subvariety $\Hec^{p,\Jj}_\ni$ with the \emph{unipotent locus} in $\M_{{C}}(n,nd')$, as defined in \cite[Section 4]{Borel} out of a line bundle $\Ll \to C$ of degree $d'$, \begin{equation}\label{eq Uni}
\Uni_{{C}}\left (\Ll \right)=\left\{
(E,\varphi) \in \M_C(n,nd') \ 
\left|\ 
\begin{array}{l}
\exists \, \sigma\in H^0(X,E/\B):\\ 
\varphi\in H^0(X,E_\sigma(\mathfrak{b})\otimes K_X);\\
E_\sigma / \U \cong \bigoplus_{i = 1}^{n} \Ll K_C^{(n+1-2i)/2}.
\end{array}
\right.
\right\},
\end{equation}
        The relevant line bundles for us are those of the orbit $\Gamma(\Jj) = \{ \gamma^*\Jj \textnormal{ for } \gamma \in \Gamma \}$, and for each of them the corresponding unipotent locus will be denoted by $\Uni_C(\gamma^{\ast}\Jj)\subset\M_{C}(n,nd')$. Recall the Cartan locus $\M_{(1,n)}$ of $\M_{{C}}(n,nd')$, introduced before Corollary~\ref{cor:pullbacktoCar}, and let $V$ be $h_{{C},n}(\M_{(1,n)})$. Then $h_{{C},n}(\Uni(\gamma^{i,\ast}\Jj))=V$, for every $i$ (cf.~\cite{Borel}). 
        
        Recall the pullback map $\hat{p}:\M_X(n,rd')\to\M_{{C}}(n,nd')$ from \eqref{eqpullback}.

\begin{proposition}\label{pr hat p}
Consider the open subvariety $\Hec^{p,\Jj}_{\Jac}$ of $\Hec^{p,\Jj}_\ni$ defined as the intersection of $\Hec^{p,\Jj}_\ni$ with the open subset $\Jac^{d + \delta}(X_b)$ of every Hitchin fiber $h_{X,n}^{-1}(b)$. Then $\Hec^{p,\Jj}_{\Jac}$ is mapped under the pullback map \eqref{eqpullback} to the union of $\Uni_C(\gamma^{\ast}\Jj)$, for $\gamma \in \Gamma$, {\it i.e.}\
$$
\hat{p}\left(\Hec^{p,\Jj}_{\Jac} \right )\subset \bigsqcup_{\gamma^{*}\Jj \in \Gamma (\Jj)} \Uni_C(\, \gamma^{*}\Jj).
$$ 
\end{proposition}
\begin{proof}
By the spectral correspondence, the proposition can be proved by showing that the pullback of the spectral data of any Higgs bundle in $\Hec^{p,\Jj}_{\Jac}$ gives the spectral data of a Higgs bundle in $\bigsqcup_{\gamma^{*}\Jj \in \Gamma (\Jj)} \Uni_C(\, \gamma^{i,*}\Jj)$. 

Take $b\in B^p_\ni\subset B^p$ and let $\tilde{b}={p}^*b$. Then $C_{\tilde{b}}$ is integral and has nodal singularities. Furthermore, by Corollary~\ref{cor:pullbacktoCar}, $\tilde{b}\in V$. 
So, if $(E,\varphi)$ represents a point in $\Hec^{p,\Jj}_{\Jac}$ mapping to $b$, then $\hat{p}(E,\varphi)$ maps to $\tilde{b}\in V$. By Theorem \ref{tm:spectdataHec}, the spectral datum of $(E,\varphi)$ is given by a line bundle $L\in\Jac^{d+\delta}(X_b)$ such that $\nu^*_a L\cong \gamma^{*}\Jj$ for some $\gamma \in \Gamma$.
Using the commutative diagram of 
Theorem~\ref{thm:normaliz}~(\ref{it Cartesian in X_gamma}), we see that $\tilde\nu_{\tilde{b}}^*q_{\tilde{b}}^*L\cong\tilde q^*_a\gamma^{*}\Jj$ is the exterior product of all the elements in the orbit $\Gamma(\Jj)$. Recall that $C_{\tilde{b}}$ is a disconnected curve, whose connected components are all isomorphic to ${C}$. We see that the spectral datum of $\hat{p}((E,\varphi))$ is a line bundle on $C_{\tilde{b}}$ (namely $q_{\tilde{b}}^*L$) which pulls back under the normalization map $\tilde\nu_{\tilde{b}}:\widetilde{C}_{\tilde{b}} \to C_{\tilde{b}}$ to the exterior product of all the elements in the orbit $\Gamma(\Jj)$. By Proposition 4.5 (4.14) of \cite{Borel}, this is the spectral datum of an element of $\Uni_C(\, \gamma^{*}\Jj)$, proving that $\hat{p}(E,\varphi)\in \Uni_C(\, \gamma^{*}\Jj)$.
\end{proof}

Given a Higgs bundle $(E,\varphi)$, the associated \emph{deformation complex} is defined by
\begin{equation}\label{eq:defcomplex}
C^\bullet_{(E,\varphi)} : \End(E) \xrightarrow{[-,\varphi]} \End(E) \otimes K.
\end{equation}
Its hypercohomology $\HH^*(C^\bullet_{(E,\varphi)})$ fits in the long exact sequence
\begin{equation}\label{eq:leshypercohom}
\begin{split}
0 \longrightarrow \HH^0(C^\bullet_{(E,\varphi)})  &\longrightarrow H^0(X,\End(E))  \longrightarrow H^0(X,\End(E) \otimes K_X)  \longrightarrow \HH^1(C^\bullet_{(E,\varphi)}) \longrightarrow
\\
& \longrightarrow H^1(X,\End(E)) \longrightarrow H^1(X,\End(E) \otimes K_X)  \longrightarrow \HH^2(C^\bullet_{(E,\varphi)})  \longrightarrow 0.
\end{split}
\end{equation}
If $(E,\varphi)$ is a stable Higgs bundle, then it represents a smooth
point of the moduli space $\M_X(n,d)$, with tangent space isomorphic
to $\HH^1(C^\bullet_{(E,\varphi)})$.


Recall from Section \ref{sc Mm hyperkahler} the holomorphic symplectic form $\Omega_{X,1} = \omega_2 + \i \omega_3$ on $\M_X(n,rd')$ associated to the complex structure $I_1$. We can now prove the main result of this section.

\begin{theorem} \label{tm lagrangian}
$\Hec^{p,\Jj}_\ni$ is a Lagrangian subvariety of $\M_X(n,rd')$ with respect to $\Omega_{X,1}$.
\end{theorem}
\begin{proof}
First we prove that $\Hec^{p,\Jj}_\ni$ is isotropic. Recall the open subset $\Hec^{p,\Jj}_{\Jac} \subset \Hec^{p,\Jj}_\ni$ defined in the previous proposition. It is enough to show that the symplectic form $\Omega_{X,1}$ vanishes on $\Hec^{p,\Jj}_{\Jac}$. 

Let $(E,\varphi)\in\Hec^{p,\Jj}_{\Jac}$. It is a smooth point of $\M_{X}(n, d)$, by Proposition~\ref{prop points Lag are stable}.
Consider the polystable Higgs bundle $(\tilde{E}, \tilde\varphi):=\hat{p}(E, \varphi)=({p}^*E, {p}^*\varphi)$ over ${C}$. By Proposition~\ref{pr hat p}, $(\tilde{E}, \tilde\varphi)\in\Uni_C(\gamma^{\ast}\Jj)\subset\M_{{C}}(n,nd')$ for some $\gamma^*\Jj\in \Gamma(\Jj)$. In addition, by \cite[Proposition 4.5]{Borel}, $(\tilde{E}, \tilde\varphi)$ is stable, thus also represents a smooth point of $\M_{{C}}(n,nd')$.

As $(E,\varphi)$ represents a smooth point of the moduli space, the corresponding tangent space is isomorphic to $\HH^1(C^\bullet_{(E,\varphi)})$, where $C^\bullet_{(E,\varphi)}$ is the complex given by \eqref{eq:defcomplex}. Since $\HH^1(C^\bullet_{(E,\varphi)})$ fits in \eqref{eq:leshypercohom}, any tangent vector in $T_{(E,\varphi)}\M_X(n,rd')\cong \HH^1(C^\bullet_{(E,\varphi)})$ is determined by an element in $H^0(X,\End(E)\otimes K_X)\cong H^{1,0}(X,\End(E))$ (providing the deformation of the Higgs field) and by an element in $H^1(X,\End(E))\cong H^{0,1}(X,\End(E))$ (providing the deformation of the holomorphic structure of $E$). Thus, every  $v,w\in T_{(E,\varphi)}\M_X(n,rd')$, may be represented as $v=(\alpha_1,\alpha_2)$ and $w=(\beta_1,\beta_2)$, with $\alpha_1,\beta_1\in \Omega^{1,0}(X,\End(E))$ and $\alpha_2,\beta_2\in\Omega^{0,1}(X,\End(E))$.
Then \cite{hitchin-self}, 
\begin{equation}\label{eq:omegax}
\Omega_{X,1}(v,w)=\int_X\tr(\alpha_1\wedge\beta_2-\alpha_2\wedge\beta_1)\in \CC.
\end{equation}
Pick the holomorphic symplectic form $\Omega_{C,1}$ on $\M_{{C}}(n,nd')$. Analogous statements hold for any pair of tangent vectors $\tilde{v},\tilde{w}\in T_{(\tilde{E},\tilde\varphi)}\M_{{C}}(n,nd')\cong \HH^1(C^\bullet_{(\wt{E}, \wt{\varphi})})$, where $C^\bullet_{(\wt{E},\wt{\varphi})}$ is the deformation complex of $(\tilde{E},\tilde\varphi)$, defined in \eqref{eq:defcomplex}. Thus, \begin{equation}\label{eq:omegaxgamma}
\Omega_{C,1}(\tilde{v},\tilde{w})=\int_{{C}}(\tilde{\alpha}_1\wedge\tilde{\beta}_2-\tilde{\alpha}_2\wedge\tilde{\beta}_1)\in \CC.
\end{equation}

If $\tilde{v}=d\hat{p}(v)$ and $\tilde{w}=d\hat{p}(w)$, then 
$\tilde{\alpha}_i={p}^*\alpha_i$ and $\tilde{\beta}_i={p}^*\beta_i$, for $i=1,2$, hence \eqref{eq:omegax} and \eqref{eq:omegaxgamma} imply that 
\begin{equation}\label{eq:nomegaX=omegXgamma}
{\hat{p}}^*\,\Omega_{C,1}=d\hat{p}^{t}\,\Omega_{C,1} = m\Omega_{X,1}
\end{equation}
because ${p}:{C}\to X$ is a degree $m$ map.

Assume now that $v,w$ are tangent vectors to $\Hec^{p,\Jj}_{\Jac}$.
By Proposition~\ref{pr hat p}, we have that $d\hat{p}(v)$ and $d\hat{p}(w)$ are tangent vectors to the unipotent locus of $\M_{{C}}(n,nd')$, which is isotropic, by Proposition 4.2 of \cite{Borel}. Hence $\Omega_{C,1}(d\hat{p}(v),d\hat{p}(w))=0$, so $\Omega_{X,1}(v,w)=0$ by \eqref{eq:nomegaX=omegXgamma}, proving that $\Hec^{p,\Jj}_{\Jac}$, thus $\Hec^{p,\Jj}_\ni$, is isotropic.

It remains to prove that $\Hec^{p,\Jj}_\ni$ is a mid-dimensional subvariety of $\M_X(n,rd')$, that is, $\dim(\Hec^{p,\Jj}_\ni)=n^2(g-1)+1$, after \eqref{eq:dim}. Since $\Hec^{p,\Jj}_\ni$ lies in the smooth locus of $\M_X(n,rd')$ after Proposition \ref{prop points Lag are stable}, its dimension can be computed through the Hitchin map restricted to $\Hec^{p,\Jj}_\ni$, namely by adding the dimension of $B^p_\ni$ to the dimension of any fiber $h_{X,n}^{-1}(b) \cap \Hec^{p,\Jj}_\ni$.

From \eqref{eq dim B^gamma} we have that $\dim B^p_{\ni} = \dim B^p = rn(g-1) + 1$.
On the other hand, by \eqref{eq dim Hitchin fiber cap Hecke}, one has that $\dim (h_{X,n}^{-1}(b) \cap \Hec^{p,\Jj}_\ni) = n (n-r)(g-1)$. Then, their sum equals 
\[
\dim B^p + \dim (h_{X,n}^{-1}(b) \cap \Hec^{p,\Jj}_\ni) = rn(g-1) + 1 + n (n-r)(g-1).
\]
Adding up yields $\dim(\Hec^{p,\Jj}_\ni)=n^2(g-1)+1$ as claimed.
\end{proof}

\section{Mirror symmetry and branes}\label{sec:FM}

In two preceding sections we constructed the Narasimhan--Ramanan $\BBB$-branes $\BBBB^p_{\Ff}$ and $\BBBB^p_\Ll$ (cf.\ Definitions~\ref{def:nonbasicNR-BBBbrane} and \ref{def:basicNR-BBBbrane}) and the complex Lagrangian subvarieties $\Hec^{p,\Jj}_\ni$ (cf.\ Definition~\ref{def:Heckesubvar}). Recall that the Definitions~\ref{def:nonbasicNR-BBBbrane} and \ref{def:basicNR-BBBbrane} required Assumption \ref{ass gamma of maximal order}, thus $r = 1$. Note as well that, by construction, both branes $\BBBB^p_\Ll$ and $\BBBB^p_{\Ff}$ determine coherent sheaves on $\M_X(n,d)$ (with respect to the complex structure $I_1$), and their support fiber (via the Hitchin map) over the locus $B^p$. On the other hand $\Hec^{p,\Jj}_\ni$ is only constructed over the open dense subspace $B^p_\ni \subset B^p$ parametrizing integral and nodal spectral curves there. 

Mirror symmetry predicts that a $\BBB$-brane ought to be dual to a $\BAA$-brane. It is expected that such duality is realized via a Fourier--Mukai transform relative to a Lagrangian fibration associated to the complex structure $I_1$. In this section we prove that (the restriction to $B^p_\ni$ of) the coherent sheaf on $\M_X(n,d)$ given by the Narasimhan--Ramanan $\BBB$-brane $\BBBB^p_\Ll$ is Fourier--Mukai transformed into a sheaf supported over $\Hec^{p,\Jj}_\ni$, for an explicit choice of $\Jj$ depending on $\Ll$. We obtain an explicit relation on this transformed sheaf which is enough to find its support, but we do not have a full description of it. The only missing piece to produce the complete (fiberwise) mirror symmetry is a global description of the corresponding flat bundle over $\Hec^{p,\Jj}_\ni$; we have partial information on it, but not a global one. Starting with $\BBBB^p_{\Ff}$ instead, we obtain a similar duality statement.

We address only the case of degree $d=0$ (thus the case of $d$ multiple of $n$ also follows), because in this case we have the Hitchin sections \cite{hitchin_lie} as global Lagrangian sections of the Hitchin fibration $h_{X,n} : \M_X(n,0) \to B_{X,n}$. This allow us to perform the fiberwise Fourier--Mukai transform without using a gerbe (or using a trivial one).
For $d$ non-multiple of $n$, then $h_{X,n}$ has no such global Lagrangian section, hence a gerbe is required to properly perform the relative Fourier--Mukai (cf.~\cite{HT}). On the other hand, we expect that all the analysis in the preceding sections generalizes to the setting of parabolic Higgs bundles, and there, under mild assumptions, the Fourier--Mukai duality can be performed for any degree $d$ without the need for a gerbe. This is because, for an appropriate choice of parabolic weights, the Hitchin fibration always admits a Hitchin section (cf.~\cite{gothen-oliveira:2017}).



\subsection{Review of autoduality of compactified Jacobians of integral curves}
\label{sc FM}
In this section we review autoduality of compactified Jacobians of integral curves with planar singularities and the associated Fourier--Mukai transform given by Arinkin in \cite{arinkin}. Since spectral curves are contained in the surface $|K_X|$, his construction applies to any integral spectral curve $X_b$, in particular to all curves for $b\in B^p_\ni$ by Theorem~\ref{thm:normaliz}  \eqref{Xb irreducible}. In this context, Arinkin's autoduality statement becomes becomes the autoduality of the corresponding Hitchin fibers $h_{X,n}^{-1}(b)$. 

Take an integral curve with planar singularities $X_b$ and consider an integer $\delta$. Then every semistable rank $1$ torsion-free sheaf on $X_b$ is indeed stable, and $\ol{\Jac}^{\, \delta}(X_b)$ is therefore a fine moduli space with universal family $\Uu_b \to X_b \times \ol{\Jac}^{\, \delta}(X_b)$. Denote by $\Uu^0_b$ its restriction to $X_b \times \Jac^{\, \delta}(X_b)$. Before constructing the Poincar\'e sheaf, we first construct the Poincar\'e bundle using $\Uu_b$ and $\Uu^0_b$. Choose a point $\sigma_X(b)$ in $\ol{\Jac}^{\, \delta}(X_b)$ and let $\Uu_b$ and $\Uu^0_b$ be normalized with respect to $\sigma_X(b)$.

Given a flat morphism $f : Y \to S$ whose geometric fibers are curves, we can define the determinant of cohomology (see \cite{knudsen&mumford} and \cite[Section 6.1]{esteves}) as follows. If $\Ee$ is an $S$-flat sheaf on $Y$, the determinant of cohomology $\Dd_f(\Ee)$ is an invertible sheaf on $S$, constructed locally as the determinant of complexes of free sheaves, which is locally quasi-isomorphic to $Rf_*\Ee$. Consider the triple product $X_b \times \ol{\Jac}^{\, \delta}(X_b) \times \Jac^{\, \delta}(X_b)$ and the projection $f_{23} : X_b \times \ol{\Jac}^{\, \delta}(X_b) \times \Jac^{\, \delta}(X_b) \to \ol{\Jac}^{\, \delta}(X_b) \times \Jac^{\, \delta}(X_b)$, which is flat and whose fibers are curves. Consider as well the corresponding projections $f_{12}$ and $f_{13}$. The Poincar\'e line bundle $\Pp_b \to \ol{\Jac}^{\, \delta}(X_b) \times \Jac^{\, \delta}(X_b)$ is the invertible sheaf
\begin{equation}\label{eq:Poincarebundle}
\Pp_b:= \Dd_{f_{23}} \left ( f_{12}^*\Uu_b \otimes f_{13}^* \Uu^0_b  \right )^{-1} \otimes \Dd_{f_{23}} \left ( f_{13}^* \Uu^0_b  \right ) \otimes \Dd_{f_{23}} \left ( f_{12}^* \Uu_b  \right ).
\end{equation}

The restriction of the Poincar\'e bundle $\Pp_b$ to the point associated to $M \in \Jac^{\, \delta}(X_b)$, that is,  $\Pp_{b,M} := \Pp_b|_{\ol{\Jac}^{\, \delta}(X_b) \times \{ M \}}$, is the line bundle over $\ol{\Jac}^{\, \delta}(X_b)$  given by
\begin{equation} \label{eq description of Pp_M}
\Pp_{b,M} = \Dd_{f_2} (\Uu_b \otimes f_1^*M)^{-1} \otimes \Dd_{f_2}(f_1^*M) \otimes \Dd_{f_2}(\Uu_b),
\end{equation}
where we have considered the projections $f_1 : X_b \times \ol{\Jac}^{\, \delta}(X_b) \to X_b$ and $f_2 : X_b \times \ol{\Jac}^{\, \delta}(X_b) \to \ol{\Jac}^{\, \delta}(X_b)$.

Our Poincar\'e bundle is constructed over $\ol{\Jac}^{\, \delta}(X_b) \times \Jac^{\, \delta}(X_b)$. A similar construction can be performed over $\Jac^{\, \delta}(X_b) \times \ol{\Jac}^{\, \delta}(X_b)$, which coincides with $\Pp_b$ after restricting both to $\Jac^{\, \delta}(X_b) \times \Jac^{\, \delta}(X_b)$. Gluing both line bundles over $\Jac^{\, \delta}(X_b) \times \Jac^{\, \delta}(X_b)$, one can define the line bundle 
\begin{equation}\label{eq:P_bsharp}
\Pp_b^\sharp \to \left ( \ol{\Jac}^{\, \delta}(X_b) \times \ol{\Jac}^{\, \delta}(X_b) \right )^\sharp,
\end{equation} where
\begin{equation}\label{eq:JactimesJaccard}
\left (\ol{\Jac}^{\, \delta}(X_b) \times \ol{\Jac}^{\, \delta}(X_b) \right)^\sharp := \left( \Jac^{\, \delta}(X_b) \times \ol{\Jac}^{\, \delta}(X_b) \right) \cup \left( \ol{\Jac}^{\, \delta}(X_b) \times \Jac^{\, \delta}(X_b) \right).
\end{equation}
Arinkin \cite{arinkin} extended this construction to the compactified Jacobian, obtaining the {\it Poincar\'e sheaf}
\[
\ol{\Pp}_b \to \ol{\Jac}^{\, \delta}(X_b) \times \ol{\Jac}^{\, \delta}(X_b),
\]
showing as well that it is a Cohen-Macaulay sheaf. Therefore, considering the injection
\begin{equation}\label{eq:inj-j}
j : \left (\ol{\Jac}^{\, \delta}(X_b) \times \ol{\Jac}^{\, \delta}(X_b) \right)^\sharp \hookrightarrow \ol{\Jac}^{\, \delta}(X_b) \times \ol{\Jac}^{\, \delta}(X_b), 
\end{equation}
one has that the Poincar\'e sheaf satisfies \cite[Lemma 6.1 (2)]{arinkin}
\begin{equation}\label{eq:Poincarsheaf}
\ol{\Pp}_b \cong j_* \Pp_b^\sharp.
\end{equation}

Taking the projections $\pi_1$, $\pi_2$ onto the first and second factors
\begin{equation}\label{eq:projections-compact-Jac}
\xymatrix{
&  \ol{\Jac}^{\, \delta}(X_b) \times \ol{\Jac}^{\, \delta}(X_b)  \ar[ld]_{\pi_1} \ar[rd]^{\pi_2} &
\\
\ol{\Jac}^{\, \delta}(X_b) & & \ol{\Jac}^{\, \delta}(X_b),
}
\end{equation}
and using $\ol{\Pp}_b$ as a kernel, we consider the Fourier--Mukai functor on the bounded derived category of coherent sheaves on $\ol{\Jac}^{\, \delta}(X_b)$,
\begin{equation} \label{eq FM}
\morph{D^b \left ( \ol{\Jac}^{\, \delta}(X_b) \right )}{D^b \left ( \ol{\Jac}^{\, \delta}(X_b) \right )}{\Ff^\bullet}{R \pi_{2,*}(\pi_1^*\Ff^\bullet \otimes \ol{\Pp}_b).}{}{\Theta_b}
\end{equation}

The following is due to Arinkin in the case of $X_b$ integral.

\begin{theorem}[\cite{arinkin}]\label{thm:deriveequiv}
Let $X_b$ be an integral curve with planar singularities and $\delta$ an integer. The moduli space of rank $1$ torsion-free sheaves over $\ol{\Jac}^{\, \delta}(X_b)$ is $\ol{\Jac}^{\, \delta}(X_b)$ itself. Furthermore the Fourier--Mukai functor $\Theta_b$ is a derived equivalence.
\end{theorem}


\subsection{Fourier--Mukai and normalization}\label{subsection:FM}

Let us consider the normalization $\nu_a : C_a \to X_b$ of an integral curve $X_b$ with nodal singularities, and denote by $\delta-\rho$ the degree of the singular divisor $\sing(X_b)$, where $\delta$ and $\rho$ are integers. In this section we study the interplay of $\nu_a$ with the Fourier--Mukai transform constructed by Arinkin\footnote{The results contained in this section were previously known to D. Arinkin, to whom we are indebted for conversations.}.


Since ${C}_a$ is smooth, the Jacobian $\Jac^{\, 0}({C}_a)$ is a smooth abelian variety known to be autodual. Choosing a point in $\Jac^{\, \delta}(C_a)$ and $\Jac^{\, \rho}(C_a)$ provides an isomorphism from the abelian variety $\Jac^{\, 0}({C}_a)$ to the torsors $\Jac^{\, \delta}({C}_a)$ and $\Jac^{\, \rho}({C}_a)$. Using this isomorphism, we naturally obtain a Poincar\'e line bundle
\[\Pp_a\longrightarrow\Jac^{\, \rho}({C}_a) \times \Jac^{\, \delta}({C}_a)\]
from the Poincaré line bundle over $\Jac^{\, 0}({C}_a) \times \Jac^{\, 0}({C}_a)$. Then $\Pp_a$ is a universal family of topologically trivial line bundles over $\Jac^{\, \rho}({C}_a)$ parametrized by $\Jac^{\, \delta}({C}_a)$.

Consider the pushforward morphism 
\[\check{\nu}_a:\Jac^{\, \rho}({C}_a)\longrightarrow\ol{\Jac}^{\, \delta}(X_b)\] induced from the normalization map $\nu_a : {C}_a \longrightarrow X_b$, and the pullback map 
\[
\hat{\nu}_a : \Jac^{\, \delta}(X_b) \longrightarrow \Jac^{\, \delta}({C}_a).
\]

Let $\id_{\wt{\Jac}}$ and $\id_{\Jac}$ be the identity morphisms in $\Jac^{\, \rho}({C}_a)$ and $\Jac^{\, \delta}(X_b)$ respectively. Recall from \eqref{eq:Poincarsheaf} that the Poincar\'e sheaf $\ol{\Pp}_b$ on $\ol{\Jac}^{\, \delta}(X_b) \times\ol{\Jac}^{\, \delta}(X_b)$ is constructed from the Poincar\'e line bundle 
\[\Pp_b \longrightarrow \ol{\Jac}^{\, \delta}(X_b) \times \Jac^{\, \delta}(X_b)\] given in \eqref{eq:Poincarebundle}. Then both $\left ( \check{\nu}_a \times \id_{\Jac}  \right )^*\Pp_{b}$ and $(\id_{\widetilde{\Jac}} \times \hat{\nu}_a)^*\Pp_a$ are bundles over $\Jac^{\, \rho}({C}_a) \times \Jac^{\, \delta}(X_b)$. The next result shows that they differ from a line bundle which is a pullback from a line bundle over $\Jac^{\, \delta}(X_b)$.

Consider the projections
\[
\xymatrix{
& \Jac^{\, \rho}({C}_a) \times \Jac^{\, \delta}(X_b) \ar[ld]_{\pi'_1} \ar[rd]^{\pi'_2}   &
\\
\Jac^{\, \rho}({C}_a) &   & \Jac^{\, \delta}(X_b).
}
\]

\begin{proposition} \label{pr relation of Poincares}
Let $\nu_a : C_a \to X_b$ be normalization of an integral curve $X_b$ with planar singularities. Then we have that
\[
\left ( \check{\nu}_a \times \id_{\Jac}  \right )^*\Pp_{b} \cong (\id_{\widetilde{\Jac}} \times \hat{\nu}_a)^*\Pp_a \otimes (\pi'_2)^* \Ww, 
\]
for some line bundle $\Ww \to \Jac^{\, \delta}(X_b)$.
\end{proposition}

\begin{proof}
After a certain adaptation, the proof is analogous to that of \cite[Lemma 5.2]{Borel}. We include it here for the sake of clarity.
First note that $\left ( \check{\nu}_a \times \id_{\Jac}  \right )^*\Pp_{b}$ is a family of topologically trivial line bundles over $\Jac^{\, \rho}({C}_a)$ parametrized by $\Jac^{\, \delta}(X_b)$. Since $\Pp_a \longrightarrow \Jac^{\, \rho}({C}_a) \times \Jac^{\, \delta}({C}_a)$ is a universal family for such objects, there exists a morphism
\[
g : \Jac^{\, \delta}(X_b) \longrightarrow \Jac^{\, \delta}({C}_a),
\]
such that 
\begin{equation}\label{eq:existsgsuchthat}
\left ( \check{\nu}_a \times \id_{\Jac}  \right )^*\Pp_{b} \cong (\id_{\widetilde{\Jac}} \times g)^*\Pp_a \otimes (\pi'_2)^* \Ww_b, 
\end{equation}
for some line bundle $\Ww_b \to \Jac^{\, \delta}(X_b)$. We claim that $g = \hat{\nu}_a$. In order to prove it, we shall need several preliminary statements.

Recall the description of $\Pp_{b,M}$ given in \eqref{eq description of Pp_M},
\[\Pp_{b,M} = \Dd_{f_2} (\Uu_b \otimes f_1^*M)^{-1} \otimes \Dd_{f_2}(f_1^*M) \otimes \Dd_{f_2}(\Uu_b),\] for each $M \in \Jac^{\, \delta}(X_b)$, where 
\[f_1 : X_b \times \ol{\Jac}^{\, \delta}(X_b) \longrightarrow X_b\ \ \ \text{ and }\ \ \ f_2 : X_b \times \ol{\Jac}^{\, \delta}(X_b) \longrightarrow \ol{\Jac}^{\, \delta}(X_b).\] 
Take also the projections
\[\tilde{f}_1 : {C}_a \times \Jac^{\, \rho}({C}_a) \longrightarrow {C}_a,\ \ \ \ \tilde{f}_2 : {C}_a \times \Jac^{\, \rho}({C}_a) \longrightarrow \Jac^{\, \rho}({C}_a)\] 
and
\[f'_1 : X_b \times \Jac^{\, \rho}({C}_a) \longrightarrow X_b\ \ \ \ \ f'_2 : X_b \times \Jac^{\, \rho}({C}_a) \longrightarrow \Jac^{\, \rho}({C}_a).\] 

Obviously
\begin{equation} \label{eq comm diagram for f}
\xymatrix{
X_b \times \Jac^{\, \rho}({C}_a) \ar[rr]^{\id_{X_b} \times \check{\nu}_a} \ar[rrd]_{f'_1} & & X_b \times \overline{\Jac}^{\, \delta}(X_b) \ar[d]^{f_1}
\\
& & X_b}
\end{equation}
commutes.

The following diagram is also obviously Cartesian,
\[
\xymatrix{
X_b \times \Jac^{\, \rho}({C}_a) \ar[rr]^{\id_{X_b} \times \check{\nu}_a} \ar[d]_{f'_2} & & X_b \times \ol{\Jac}^{\, \delta}(X_b) \ar[d]^{f_2}
\\
\Jac^{\, \rho}({C}_a) \ar[rr]^{\check{\nu}_a} & & \ol{\Jac}^{\, \delta}(X_b),
}
\]
thus, since we know from \cite[Proposition 44 (1)]{esteves} that the determinant of cohomology commutes with base change, 
\begin{equation}\label{eq:basechange}
\check{\nu}_a^* \Dd_{f_2} = \Dd_{f'_2} (\id_{X_b} \times  \check{\nu}_a)^*. 
\end{equation}

Let $\Uu_a \longrightarrow {C}_a \times \Jac^{\, \rho}({C}_a)$ be the universal bundle of topologically trivial line bundles over ${C}_a$. Since ${C}_a$ is smooth, we may apply \eqref{eq description of Pp_M}, so that the Poincar\'e bundle $\Pp_a$ satisfies 
\begin{equation}\label{eq:fiberPgamma}
\Pp_{a, N} = \Dd_{\tilde{f}_2} (\Uu_a \otimes \tilde{f}_1^*N)^{-1} \otimes \Dd_{\tilde{f}_2}(\tilde{f}_1^*N) \otimes \Dd_{\tilde{f}_2}(\Uu_a),
\end{equation} for any $N\in\Jac^{\, \delta}({C}_a)$.
Recall that $\Uu_b \to X_b \times \ol{\Jac}^{\, \delta}(X_b)$ is the universal sheaf of degree $\delta$ torsion-free sheaves on $X_b$ and consider the pullback $(\id_{X_b} \times \check{\nu}_a)^*\Uu_b$ which is a sheaf over $X_b\times \Jac^{\, \rho}({C}_a)$. Observe that $(\nu_a \times \id_{\widetilde{\Jac}})_* \Uu_a$ is a family of rank $1$ torsion-free sheaves of degree $\delta$ over $X_b$. Then, by the universality property, there exists a morphism $h : \Jac^{\, \rho}({C}_a) \to \overline{\Jac}^{\, \delta}(X_b)$ and a line bundle $\Ww_a \to \Jac^{\, \rho}({C}_a)$ such that 
\[
(\id_{X_b} \times h)^* \Uu_b \cong (\nu_a \times \id_{\widetilde{\Jac}})_* \Uu_a \otimes (f'_2)^*\Ww_a.
\]
Note that $h$ coincides pointwise with $\check{\nu}_a$, since $\Jac^{\, \rho}({C}_a)$ is smooth it follows that $h = \check{\nu}_a$. Then,
\begin{equation} \label{eq relation between Uu_b and Uu_gamma}
(\id_{X_b} \times \check{\nu}_a)^* \Uu_b \cong (\nu_a \times \id_{\widetilde{\Jac}})_* \Uu_a \otimes (f'_2)^*\Ww_a.
\end{equation}
Since the diagram
\[
\xymatrix{
{C}_a \times \Jac^{\, \rho}({C}_a) \ar[rr]^{\nu_a \times \id_{\widetilde{\Jac}}} \ar[rrd]_{\tilde{f}_2} & & X_b \times \Jac^{\, \rho}({C}_a) \ar[d]^{f'_2}
\\
& & \Jac^{\, \rho}({C}_a),
}
\]
commutes, the definition of the determinant of cohomology ensures that
\begin{equation} \label{eq determinant of coh and pushforward}
\Dd_{f'_2} (\nu_a \times \id_{\widetilde{\Jac}})_* \cong \Dd_{\tilde{f}_2}.
\end{equation}

Finally,
\[
\xymatrix{
& & & {C}_a \times \Jac^{\, \rho}({C}_a) \ar[llld]_{\nu_a \circ \tilde{f}_1} \ar[rrrd]^{\tilde{f}_2} \ar[d]^{(\nu_a \times \id_{\widetilde{\Jac}})} & & &
\\
X_b & & & X_b \times \Jac^{\, \rho}({C}_a) \ar[lll]^{f'_1} \ar[rrr]_{f'_2} & & & \Jac^{\, \rho}({C}_a).
}\]
also commutes, so $f'_1((f'_2)^{-1}(U)) = \nu_a \tilde{f}_1 (\tilde{f}_2^{-1}(U))$ for every open subset $U \subset \Jac^{\, \rho}({C}_a)$. It then follows that for each $M \in \Jac^{\, \delta}(X_b)$ and each open set $U\subset\Jac^{\, \rho}({C}_a)$,
\begin{align*}
((f'_2)_* (f'_1)^* M)(U) = & \lim_{W \supseteq f'_1((f'_2)^{-1}(U))} M(W)
\\
= & \lim_{W \supseteq \nu_a \tilde{f}_1 (\tilde{f}_2^{-1}(U))} M(W) 
\\
= & ((\tilde{f}_2)_*(\nu_a \circ \tilde{f}_1)^* M)(U), 
\end{align*}
so $(f'_2)_* (f'_1)^* = (\tilde{f}_2)_*(\nu_a \circ \tilde{f}_1)^*$. As a consequence of this identification, we have 
\begin{equation} \label{eq relation of det coh for M}
\Dd_{f'_2} (f'_1)^* \cong \Dd_{\tilde{f}_2} \tilde{f}_1^* \nu_a^*.
\end{equation}

Using the projection formula and \eqref{eq:existsgsuchthat}--\eqref{eq relation of det coh for M}, we have that, for any $M \in \Jac^{\, \delta}(X_b)$, 
\begin{align*}
\Pp_{a, g(M)} \cong \, & \check{\nu}_a^*\Pp_{b,M}
\\
\cong \, &  \check{\nu}_a^* \left ( \Dd_{f_2} \left (\Uu_b \otimes f_1^*M \right )^{-1} \otimes \Dd_{f_2}(f_1^*M) \otimes \Dd_{f_2}\left (\Uu_b \right ) \right )
\\
\cong \, &  \check{\nu}_a^* \Dd_{f_2} \left (\Uu_b \otimes f_1^*M \right )^{-1} \otimes \check{\nu}_a^* \Dd_{f_2}(f_1^*M) \otimes \check{\nu}_a^* \Dd_{f_2}\left (\Uu_b \right )
\\
\cong \, &  \Dd_{f'_2} \left ( (\id_{X_b} \times  \check{\nu}_a)^* \left (\Uu_b \otimes f_1^*M \right) \right )^{-1} \otimes \Dd_{f'_2} \left ( (\id_{X_b} \times  \check{\nu}_a)^*(f_1^*M) \right ) \otimes \Dd_{f'_2} \left ( (\id_{X_b} \times  \check{\nu}_a)^* \Uu_b \right )
\\
\cong \, &  \Dd_{f'_2} \left ( (\id_{X_b} \times  \check{\nu}_a)^* \Uu_b \otimes (f'_1)^*M \right )^{-1} \otimes \Dd_{f'_2} \left ( (f'_1)^*M) \right ) \otimes \Dd_{f'_2} \left ( (\id_{X_b} \times  \check{\nu}_a)^* \Uu_b \right )
\\
\cong \, &  \Dd_{f'_2} \left ( (\nu_a \times  \id_{\widetilde{\Jac}})_* \Uu_a \otimes (f'_1)^*M  \otimes (f'_2)^*\Ww_a \right )^{-1} \otimes 
\\
& \qquad \qquad  \otimes \Dd_{f'_2} \left ( (f'_1)^*M) \right ) \otimes \Dd_{f'_2} \left ( (\nu_a \times  \id_{\widetilde{\Jac}})_* \Uu_a  \otimes (f'_2)^*\Ww_a \right )
\\
\cong \, &  \Dd_{f'_2} \left ( (\nu_a \times  \id_{\widetilde{\Jac}})_* \Uu_a \otimes (f'_1)^*M  \right )^{-1} \otimes \Ww_a^{-1} \otimes 
\\
& \qquad \qquad  \otimes \Dd_{f'_2} \left ( (f'_1)^*M) \right ) \otimes \Dd_{f'_2} \left ( (\nu_a \times  \id_{\widetilde{\Jac}})_* \Uu_a  \right ) \otimes \Ww_a
\\
\cong \, &  \Dd_{f'_2} \left ( (\nu_a \times  \id_{\widetilde{\Jac}})_*  \left (\Uu_a \otimes \tilde{f}_1^* \nu_a^* M \right ) \right )^{-1} \otimes \Dd_{f'_2} \left ( (f'_1)^*M) \right ) \otimes \Dd_{f'_2} \left ( (\nu_a \times  \id_{\widetilde{\Jac}})_* \Uu_a \right )
\\ 
\cong \, &  \Dd_{\tilde{f}_2} (\Uu_a \otimes \tilde{f}_1^*\nu_a^*M)^{-1}  \otimes \Dd_{\tilde{f}_2} \left (\tilde{f}_1^*\nu_a^*M \right ) \otimes \Dd_{\tilde{f}_2} (\Uu_a)
\\ 
\cong \, & \Pp_{\nu_a^*M}
\\ 
\cong \, & \Pp_{a, \hat{\nu}_a(M)}.
\end{align*}
This implies that $g(M) = \hat{\nu}_a(M)$ for any $M \in \Jac^{\, \delta}(X_b)$. Since $\Jac^{\, \delta}(X_b)$ is smooth, this suffices to state that $g = \hat{\nu}_a$, thus completing the proof.
\end{proof}

Consider now the projections
\begin{equation}\label{eq:tildepi12}
\xymatrix{
&  \Jac^{\, \rho}({C}_a) \times \Jac^{\, \delta}({C}_a) \ar[ld]_{\eta_1} \ar[rd]^{\eta_2} &
\\
\Jac^{\, \rho}({C}_a) & & \Jac^{\, \delta}({C}_a).
}
\end{equation}
Using the Poincaré line bundle $\Pp_a \to \Jac^{\, \rho}({C}_a) \times \Jac^{\, \delta}({C}_a)$ as kernel, define the Fourier--Mukai functor
\begin{equation}\label{eqThetagamma}
\morph{D^b(\Jac^{\, \rho}({C}_a))}{D^b(\Jac^{\, \delta}({C}_a))}{\Ee^\bullet}{R\eta_{2,*} (\eta_1^*\Ee^\bullet \otimes \Pp_a).}{}{\Theta_a}
\end{equation}
Mukai \cite[Theorem 2.2]{mukai} proved in this classical setting that the above functor is a derived equivalence. Note also that it is the particular case of Arinkin's Theorem \ref{thm:deriveequiv} for smooth curves.

The next theorem establishes a relation between the Fourier--Mukai functors $\Theta_b$ in \eqref{eq FM} and $\Theta_a$ in \eqref{eqThetagamma}, for complexes arising as pushforward via $\check\nu_a$. Recall the moduli space of parabolic modules $\PMod^{\delta}_{\ell}({C}_a,\widetilde{D}_a)$ of degree $\delta$ and type $\ell = (1, \dots, 1)$, and let $\Oo_{\PMod}$ be the corresponding structure sheaf. Recall also the morphisms $\tau$ in \eqref{eq definition of tau} and $\dot\nu_a$ in \eqref{eq definition of dotnu}.

\begin{theorem} \label{tm relation of FM}
Let $\nu_a : C_a \to X_b$ be normalization of an integral curve $X_b$ with planar singularities and let $\Ee^\bullet$ be a complex in $D^b(\Jac^{\, \rho}({C}_a))$. Then there is an isomorphism
\[
\Theta_b(R\check{\nu}_{a,*}\Ee^\bullet)\otimes\tau_* \Oo_{\PMod} \cong R\tau_* \dot{\nu}_a^* \Theta_a(\Ee^\bullet) \otimes \ol{\Pp}_b|_{ \{ \nu_{a,*}\Oo_{{C}_a} \} \times \ol{\Jac} }.
\]
\end{theorem}
\begin{proof}
Consider the pullback of the Poincar\'e sheaf to $\ol{\Jac}^{\, \delta}(X_b) \times \PMod^{\delta}_{\ell}({C}_a,\widetilde{D}_a)$,
\begin{equation} \label{eq definition dotPp}
\dot\Pp_b := (\id_{\ol{\Jac}} \times \tau)^* \ol{\Pp}_b.
\end{equation}
Applying the projection formula yields
\begin{equation} \label{eq relation between dotP and olP}
(\id_{\ol{\Jac}} \times \tau)_* \dot\Pp_b \cong \ol{\Pp}_b\otimes (\id_{\ol{\Jac}} \times \tau)_* \Oo_{\ol{\Jac} \times \PMod}.
\end{equation}

Consider the projections,
\begin{equation} \label{eq projections dotpi}
\xymatrix{
& \ol{\Jac}^{\, \delta}(X_b) \times \PMod^{\delta}_{\ell}({C}_a,\widetilde{D}_a) \ar[ld]_{\dot\pi_1} \ar[rd]^{\dot\pi_2} & 
\\
\ol{\Jac}^{\, \delta}(X_b) & & \PMod^{\delta}_{\ell}({C}_a,\widetilde{D}_a)
}
\end{equation}
and notice that, since $\ol{\Jac}^{\, \delta}(X_b)$ is projective, connected and reduced, 
\[\dot\pi_{2,\ast}\Oo_{\ol{\Jac}\times\PMod}\cong\Oo_{\PMod}.\]
Note also that 
\[\pi_2 \circ (\id_{\ol{\Jac}} \times \tau) = \tau \circ \dot\pi_2\ \ \ \text{ and }\ \ \ \dot\pi_1 = \pi_1 \circ (\id_{\ol{\Jac}} \times \tau),\]
where $\pi_1$ and $\pi_2$ are the projections defined in \eqref{eq:projections-compact-Jac}. Now, recalling \eqref{eq FM} and using these relations, the identity \eqref{eq relation between dotP and olP}, and the projection formula and the fact that the derived direct image is functorial with respect to compositions, we have that, for $\Ff^\bullet\in D^b\big(\ol{\Jac}^{\, \delta}(X_b)\big)$,
\begin{equation} \label{eq:Theta_b-it1}
\begin{split}
R\tau_* R\dot\pi_{2,*} \big (\dot\pi_1^*\Ff^\bullet\otimes \dot\Pp_b \big )&\cong R\pi_{2,*}R(\id_{\ol{\Jac}} \times \tau)_*\big ((\id_{\ol{\Jac}} \times \tau)^*\pi_1^*\Ff^\bullet\otimes \dot\Pp_b \big )\\
&\cong R\pi_{2,*}\big (\pi_1^*\Ff^\bullet\otimes (\id_{\ol{\Jac}} \times \tau)_*\dot\Pp_b \big )\\
&\cong R \pi_{2,*}(\pi_1^*\Ff^\bullet \otimes \ol{\Pp}_b)\otimes\pi_{2,*}(\id_{\ol{\Jac}} \times \tau)_* \Oo_{\ol{\Jac} \times \PMod}\\
&\cong R \pi_{2,*}(\pi_1^*\Ff^\bullet \otimes \ol{\Pp}_b)\otimes\tau_*\dot\pi_{2,*}\Oo_{\ol{\Jac} \times \PMod}\\
&=\Theta_b(\Ff^\bullet)  \otimes\tau_* \Oo_{\PMod}.
\end{split}
\end{equation}

The next step consists of establishing a relation between $\dot\Pp_b$ and $\Pp_a$. As $\dot\Pp_b$ parametrizes rank 1 torsion-free sheaves over $\ol\Jac^{\, \delta}(X_b)$, then 
\[(\check\nu_a \times \id_{\PMod})^*\dot\Pp_b \to \Jac^{\, \rho}({C}_a) \times \PMod^{\delta}_{\ell}({C}_a,\widetilde{D}_a)\] is a family of rank 1 torsion-free sheaves over $\Jac^{\, \rho}({C}_a)$ ({\it i.e.} line bundles on $\Jac^{\, \rho}({C}_a)$) parametrized by $\PMod^{\delta}_{\ell}({C}_a,\widetilde{D}_a)$. Since $\Pp_a \to \Jac^{\, \rho}({C}_a) \times \Jac^{\, \delta}({C}_a)$ is a universal bundle, there is a morphism 
\[
g : \PMod^{\delta}_{\ell}({C}_a,\widetilde{D}_a) \to \Jac^{\, \delta}({C}_a),
\]
such that
\[
(\check{\nu}_a \times \id_{\PMod})^*\dot\Pp_b \cong (\id_{\wt{\Jac}} \times g)^*\Pp_a \otimes (\pi''_2)^*\Ww'.
\]
where $\pi''_2$ is the projection given by
\begin{equation} \label{eq projections pi''}
\xymatrix{
& \Jac^{\, \rho}({C}_a) \times \PMod^{\delta}_{\ell}({C}_a,\widetilde{D}_a) \ar[ld]_{\pi''_1} \ar[rd]^{\pi''_2}   &
\\
\Jac^{\, \rho}({C}_a) &   & \PMod^{\delta}_{\ell}({C}_a,\widetilde{D}_a),
}
\end{equation}
and $\Ww'$ is some line bundle on $\PMod^{\delta}_{\ell}({C}_a,\widetilde{D}_a)$.

Thanks to \eqref{eq:Poincarsheaf} and to the definition of $\dot\Pp_b$ given in \eqref{eq definition dotPp}, we see that that restricting it to $\Jac^{\, \delta}(X_b) \hookrightarrow{\tau_0^{-1}} \PMod^{\delta}_{\ell}({C}_a,\widetilde{D}_a)$  gives $\dot\Pp_b |_{\ol{\Jac} \times \Jac} \cong \Pp_b$. Then, with $|_{\wt{\Jac}\times\Jac}$ denoting the restriction to  $\Jac^{\, \rho}({C}_a)\times\Jac^{\, \delta}(X_b)$ via $\tau_0^{-1}$, 
\[
(\id_{\wt{\Jac}} \times g)^*\Pp_a |_{\wt{\Jac} \times \Jac} \otimes (\pi''_2)^*\Ww'|_{\wt{\Jac} \times \Jac} \cong (\check{\nu}_a \times \id_{\PMod})^*\dot\Pp_b|_{\wt{\Jac} \times \Jac} \cong (\check{\nu}_a \times \id_{\Jac})^*\Pp_b.
\]
So Proposition \ref{pr relation of Poincares} shows that
\[
(\id_{\wt{\Jac}} \times g)^*\Pp_a |_{\wt{\Jac} \times \Jac} \otimes (\pi''_2)^*\Ww'|_{\wt{\Jac} \times \Jac} \cong (\id_{\wt{\Jac}} \times \hat\nu_a)^*\Pp_a \otimes (\pi'_2)^*\Ww_b.
\]
Then the diagram
\[
\xymatrix{
\Jac^{\, \delta}(X_b) \ar@{^(->}[d]_-{\tau_0^{-1}} \ar[rd]^-{\hat\nu_a}\\
\PMod^{\delta}_{\ell}({C}_a,\widetilde{D}_a) \ar[r]^-{g} & \Jac^{\, \delta}({C}_a) 
}
\]
commutes, hence we conclude from \eqref{eq hatnu and dotnu commute} that $g = \dot\nu_a$, as both coincide in the dense open subset $\Jac^{\, \delta}(X_b)$. As a consequence, we obtain
\begin{equation} \label{eq preliminary relation between dotPoincare and wtPoincare}
(\check{\nu}_a \times \id_{\PMod})^*\dot\Pp_b \cong (\id_{\wt{\Jac}} \times \dot\nu_a)^*\Pp_a \otimes (\pi''_2)^*\Ww'.
\end{equation} 

Restricting \eqref{eq preliminary relation between dotPoincare and wtPoincare} to $\{ \Oo_{{C}_a} \} \times\PMod^{\delta}_{\ell}({C}_a,\widetilde{D}_a)$ yields
\[
\Ww' \cong ((\check{\nu}_a \times \id_{\PMod})^*\dot\Pp_b)|_{ \{ \Oo_{{C}_a} \} \times \PMod} \cong \dot{\Pp}_b |_{ \{ \nu_{a, *}\Oo_{{C}_a} \} \times \PMod},
\]
because $\Pp_a$ is normalized as in \eqref{eq:Poinc-normalized}. But $\dot{\Pp}_b$ is the pullback of $\ol{\Pp}_b$ under $(\id_{\ol{\Jac}} \times \tau)$, so
\begin{equation} \label{eq description of Ww'}
\Ww' \cong \tau^* \big(\ol{\Pp}_b |_{ \{ \nu_{a, *}\Oo_{{C}_a} \} \times \ol{\Jac}}\big).
\end{equation} 
Combining this description with \eqref{eq preliminary relation between dotPoincare and wtPoincare}, we conclude that
\begin{equation} \label{eq relation between dotPoincare and wtPoincare}
(\check{\nu}_a \times \id_{\PMod})^*\dot\Pp_b \cong (\id_{\wt{\Jac}} \times \dot\nu_a)^*\Pp_a \otimes (\pi''_2)^*\tau^*\big(\ol{\Pp}_b |_{ \{ \nu_{a, *}\Oo_{{C}_a} \} \times \ol{\Jac}}\big).
\end{equation} 

We now address the last part of the proof. Recall the projections $\pi''_1$ and $\pi''_2$ from \eqref{eq projections pi''}, $\dot{\pi}_1$ and $\dot{\pi}_2$ from \eqref{eq projections dotpi} and finally $\eta_1$ and $\eta_2$ from \eqref{eq:tildepi12}, and observe that 
\begin{equation}\label{eq:functo}
\begin{split}
\pi''_2 &= \dot\pi_2 \circ (\check{\nu}_a \times \id_{\PMod}),
\\ 
\pi''_1 &= \eta_1 \circ (\id_{\widetilde{\Jac}} \times \dot{\nu}_a).
\end{split}
\end{equation}
Moreover,
\[
\xymatrix{
\Jac^{\, \rho}({C}_a) \times \PMod^{\delta}_{\ell}({C}_a,\widetilde{D}_a) \ar[rr]^-{\pi''_1} \ar@{^{(}->}[d]_-{\check{\nu}_a\times\id_{\PMod}} & & \Jac^{\, \rho}({C}_a) \ar@{^{(}->}[d]^-{\check{\nu}_a}
\\
\ol{\Jac}^{\, \delta}(X_b)\times \PMod^{\delta}_{\ell}({C}_a,\widetilde{D}_a) \ar[rr]^-{\dot\pi_1} & & \ol{\Jac}^{\, \delta}(X_b),
}
\]
and
\[
\xymatrix{
\Jac^{\, \rho}({C}_a) \times \PMod^{\delta}_{\ell}({C}_a,\widetilde{D}_a) \ar[rr]^-{\id_{\widetilde{\Jac}} \times \dot{\nu}_a} \ar[d]_-{\pi''_2} & & \Jac^{\, \rho}({C}_a) \times \Jac^{\, \delta}({C}_a) \ar[d]^-{\eta_2}
\\
\PMod^{\delta}_{\ell}({C}_a,\widetilde{D}_a) \ar[rr]^{\dot{\nu}_a} & & \Jac^{\, \delta}({C}_a)
}
\]
are Cartesian diagrams.

The derived direct image and pullback are functorial with respect to compositions. Furthermore, the base-change formula applies to the two previous Cartesian diagrams. So, starting from \eqref{eq:Theta_b-it1} and using these facts, together with \eqref{eq relation between dotPoincare and wtPoincare} and with the projection formula, finally yields
\begin{align*}
\Theta_b(R\check{\nu}_{a,*}\Ee^\bullet)\otimes\tau_* \Oo_{\PMod} \cong & R\tau_* R\dot\pi_{2,*} \big ( \dot\pi_1^*R\check{\nu}_{a,*}\Ee^\bullet \otimes \dot\Pp_b \big )
\\
\cong & R\tau_* R\dot\pi_{2,*} \big ( R(\check{\nu}_a \times \id_{\PMod})_* (\pi''_1)^*\Ee^\bullet \otimes \dot\Pp_b \big )
\\
\cong & R\tau_* R\dot\pi_{2,*}R (\check{\nu}_a \times \id_{\PMod})_*  \big ( (\pi''_1)^*\Ee^\bullet \otimes (\check{\nu}_a \times \id_{\PMod})^*\dot\Pp_b \big )
\\
\cong & R\tau_* R\pi''_{2,*} \big ( (\pi''_1)^*\Ee^\bullet \otimes (\check{\nu}_a \times \id_{\PMod})^*\dot\Pp_b \big )
\\
\cong & R\tau_* R\pi''_{2,*} \big ( (\pi''_1)^*\Ee^\bullet \otimes (\id_{\wt{\Jac}} \times \dot\nu_a)^*\Pp_a \otimes (\pi''_2)^*\tau^* \big(\ol{\Pp}_b |_{ \{ \nu_{a, *}\Oo_{{C}_a} \} \times \ol{\Jac}}\big ) \big )
\\
\cong & R\tau_* R\pi''_{2,*} \big ( (\pi''_1)^*\Ee^\bullet \otimes (\id_{\wt{\Jac}} \times \dot\nu_a)^*\Pp_a \big ) \otimes \ol{\Pp}_b |_{ \{ \nu_{a, *}\Oo_{{C}_a} \} \times \ol{\Jac}}
\\
\cong & R\tau_* R\pi''_{2,*} (\id_{\widetilde{\Jac}} \times \dot{\nu}_a)^* (\eta_1^*\Ee^\bullet \otimes \Pp_a) \otimes \ol{\Pp}_b |_{ \{ \nu_{a, *}\Oo_{{C}_a} \} \times \ol{\Jac}}
\\
\cong & R\tau_* \dot{\nu}_a^* R\eta_{2,*}  (\eta_1^*\Ee^\bullet \otimes \Pp_a) \otimes \ol{\Pp}_b |_{ \{ \nu_{a, *}\Oo_{{C}_a} \} \times \ol{\Jac}}
\\
\cong & R\tau_* \dot{\nu}_a^* \Theta_a(\Ee^\bullet) \otimes \ol{\Pp}_b |_{ \{ \nu_{a, *}\Oo_{{C}_a} \} \times \ol{\Jac}},
\end{align*}
as claimed.
\end{proof}

\subsection{Branes and Fourier--Mukai transform}

We are now at the last step towards the goal of proving the duality statement between the branes we constructed.

Along this section we fix the degree to be trivial, $d = 0$. We require also Assumption \ref{ass gamma of maximal order}, so $p : C \to X$ has order $m = n$ (hence $r = 1$ and $\rho = 0$) and the spectral data of $\M_X(n,d)^p$ is as described in \eqref{eq fiber as quotient of fibers}. In particular, the normalization of the spectral curves is always $C$ and $B^p_\ni$ coincides with $B^p_{\ni}$, the subset parametrizing integral and nodal curves. 

Let us use the Hitchin section $\sigma_C$ constructed from a spin structure $K_C^{1/2} = p^*K_X^{1/2}$ to choose a point in $\Jac^{\, \delta}({C})$. Consider associated the Poincar\'e bundle
\[
\Pp \to \Jac^{\, 0}({C}) \times \Jac^{\, \delta}({C}).
\]
If $N\in\Jac^{\, \delta}({C})$, then $\Pp_{N} =\Pp|_{\Jac^{0}(C)\times\{N\}}$ is the line bundle over $\Jac^{\,0}(C)$ corresponding to the point $N\otimes K_C^{-(n-1)/2}$ of $\Jac^{\, 0}({C})$ under autoduality of $\Jac^{\, 0}({C})$.
We can assume that $\Pp_a$ is normalized so that
\begin{equation}\label{eq:Poinc-normalized}
\Pp |_{ \{ \Oo_{{C}} \} \times\Jac^{\, \delta}({C}) } \cong \Oo_{\Jac^{\, \delta}({C})}.
\end{equation}

 Let $b \in B^p_{\ni}$ and $\phi\in H^0({C},K_C)^\free$ be a representative of $b$, and recall from \eqref{eq:pushforwardnormalizr=1} the pushforward morphism $\check{\nu}_\phi:\Jac^0(C)\hookrightarrow\overline{\Jac}^{\,\delta}(X_b)$. We wish to understand the Fourier--Mukai transform of the sheaf $\check{\nu}_{\phi,*}\Nm^* \widecheck{\Ll}$ over $\ol{\Jac}^{\, \delta}(X_b)\cong h_{X,n}^{-1}(b)$ under the derived equivalence $\Theta_b:D^b \big ( \ol{\Jac}^{\, \delta}(X_b) \big )\to D^b \big ( \ol{\Jac}^{\, \delta}(X_b) \big )$.
Indeed, by Proposition~\ref{prop spectral data BBB'}, $\check{\nu}_{\phi,*}\Nm^* \widecheck{\Ll}$ is supported on $\image (\check{\nu}_\phi)\cong \Jac^0(C)$ and is the restriction to the Hitchin fiber over $b$ (intersected with $\M_X(n,0)^p$) of the hyperholomorphic line bundle $\Lll$ defining the rank $1$ Narasimhan-Ramanan $\BBB$-brane $\BBBB_\Ll^p$.  
It is a classical fact that that the Fourier--Mukai of a line bundle over $\Jac^{\, 0}({C})$ is a complex supported only in one degree (namely the genus of ${C}$), so it can be considered as a sheaf as opposed to a complex. Hence, by Theorem \ref{tm relation of FM}, $\Theta_b(\check{\nu}_{\phi,*}\Nm^* \widecheck{\Ll})$ is a sheaf over $\ol{\Jac}^{\, \delta}(X_b)$, whose support is the intersection of the support of the dual $\BAA$-brane with the Hitchin fiber over $b$.

By considering the rank $n$ Narasimhan-Ramanan $\BBB$-brane $\BBBB_{\Ff}^p$, we conclude by the same token that the support of the sheaf $\Theta_b(\bigoplus_{\gamma \in \Gamma}\check{\nu}_{\phi,*} \hat{\gamma}^{*}\widecheck{\Ff})$ determines the support of the dual $\BAA$-brane.




\begin{theorem} \label{tm support of the dual brane in the Hitchin fiber}
Consider the moduli space $\M_X(n,0)$ and the Narasimhan-Ramanan $\BBB$-branes $\BBBB_\Ll^p$ and $\BBBB_{\Ff}^p$ on it associated to a connected unramified cover $p :C \to X$. Let $b=\zeta(\phi) \in B^p_{\ni}$ for some $\phi\in H^0({C},K_C)^\free$. Let $\hat\nu_\phi: \Jac^\delta(X_b)\to\Jac^\delta(C)$ be the pullback morphism associated to the normalization $\nu_\phi$ (thus corresponding to \eqref{eq definition of hat nu}).
\begin{enumerate}[(i)]
\item\label{item1} Let \[\hat{\Ll}:={p}^*\Ll\otimes K_C^{(n-1)/2}\in\Jac^{\, \delta}({C}).\] 
The Fourier--Mukai transform of the hyperholomorphic sheaf $\Lll|_{h_{X,n}^{-1}(b)}\cong \check{\nu}_{\phi,*}\Nm^*\widecheck{\Ll}$ satisfies the relation
\begin{equation}\label{eq:relation FM transf 1}
\Theta_b(\check{\nu}_{\phi,*}\Nm^*\widecheck{\Ll})\otimes\tau_* \Oo_{\PMod}\cong \ol{\Pp}_b |_{ \{ \nu_{\phi, *}\Oo_{{C}} \} \times \overline{\hat\nu_\phi^{-1}(\hat{\Ll})}}
\end{equation} and its support is
\[
\supp(\Theta_b(\check{\nu}_{\phi,*}\Nm^*\widecheck{\Ll}))=\overline{\hat\nu_\phi^{-1}(\hat{\Ll})}=\Hec^{p, \hat{\Ll}}_\ni \cap h_{X,n}^{-1}(b).
\]

\item\label{item2} Analogously, let 
\begin{equation} \label{eq def hatFf}
\hat{\Ff}:=\Ff\otimes K_C^{(n-1)/2}
\end{equation}
and let $\Gamma(\Ff)$ be the orbit of $\Ff$ by the Galois group $\Gamma$ of ${p}$. The Fourier--Mukai transform of the hyperholomorphic sheaf $\Fff|_{h_{X,n}^{-1}(b)}\cong \bigoplus_{\gamma \in \Gamma}\check{\nu}_{\phi,*} \hat{\gamma}^{*}\widecheck{\Ff}$ satisfies the relation 
\begin{equation}\label{eq:relation FM transf n}
\Theta_b \bigg(\bigoplus_{\gamma \in \Gamma}\check{\nu}_{\phi,*} \hat{\gamma}^{*}\widecheck{\Ff}\bigg)\otimes\tau_* \Oo_{\PMod}\cong  \bigoplus_{\gamma \in \Gamma} \tau_*\Oo_{\hat{\nu}_\phi^{-1}(\gamma^{*}\hat{\Ff})} \otimes \ol{\Pp}_b |_{ \{ \nu_{\phi, *}\Oo_{{C}} \} \times \ol{\Jac}}
\end{equation}
and its support is 
\[
\supp\bigg(\Theta_b \bigg  (\bigoplus_{\gamma \in \Gamma}\check{\nu}_{\phi,*} \hat{\gamma}^{*}\widecheck{\Ff}\bigg)\bigg)=\bigcup_{\gamma^{*}\hat{\Ff} \in \Gamma(\Ff)}\overline{\hat{\nu}_\phi^{-1}(\gamma^{*}\hat{\Ff})}=\Hec^{p, \hat{\Ff}}_\ni \cap h_{X,n}^{-1}(b).\]

\end{enumerate}
\end{theorem}
\begin{proof}
By Theorem~\ref{tm relation of FM} we have 
\begin{equation}\label{eq:ThetabThetagammahyperhol}
\Theta_b(\check{\nu}_{\phi,*}\Nm^*\widecheck{\Ll})\otimes\tau_* \Oo_{\PMod}\cong\tau_*\dot\nu_\phi^*\Theta(\Nm^*\widecheck{\Ll}) \otimes \ol{\Pp}_b |_{ \{ \nu_{\phi, *}\Oo_{{C}} \} \times \ol{\Jac}}
\end{equation}
so we need to compute $\Theta(\Nm^*\widecheck{\Ll})$; cf. \eqref{eqThetagamma}, we have removed the index $a$ since now all the $C_a$ are isomorphic to $C$. This is the classical Fourier--Mukai transform on an abelian variety, the only difference being that $\Theta:D^b(\Jac^{\, 0}({C}))\to D^b(\Jac^{\, \delta}({C}))$ takes values in the derived category of complexes over the torsor $\Jac^{\, \delta}({C})$ and not over the actual abelian variety $\Jac^{\, 0}({C})$. So we must use the identification we settled $\Jac^{\, 0}({C})\xrightarrow{\cong}\Jac^{\, \delta}({C})$, by tensorization by $K_C^{(n-1)/2}$.
As is well-known, $\Theta(\Nm^*\widecheck{\Ll})$ is the skyscraper sheaf over the point of $\Jac^{\, \delta}({C})$ whose corresponding point over $\Jac^{\, 0}({C})$ corresponds to the line bundle $\Nm^*\widecheck{\Ll}$ under the autoduality of $\Jac^{\, 0}({C})$.

Consider the commutative diagram 
\begin{equation}\label{eq:diagramrestrictiontocurve}
\xymatrix{{C}\ar@{^{(}->}[r]^-{A_{C}}\ar[d]_{{p}}&\Jac^{\, 0}({C})\ar[d]^{\Nm}\\
X\ar@{^{(}->}[r]_-{A_{X}}&\Jac^0(X),}
\end{equation}
where the horizontal maps are the Abel--Jacobi maps determined by chosen base points on ${C}_a$ and $X$, which correspond under ${p}\colon {C}\to X$. Since the isomorphisms yielding the autoduality of $\Jac^{\, 0}({C})$ and $\Jac^0(X)$ are given by pullback of the Abel--Jacobi maps, this yields the commutative diagram
\[
\xymatrix{\Jac^{\, 0}({C})^{\vee}\ar[rr]^-{A_{{C}}^*}_{\cong}&&\Jac^{\, 0}({C})\ar[rr]^{\cong}_{-\otimes K_C^{(n-1)/2}}&&\Jac^{\, \delta}({C})\\\Jac^0(X)^{\vee}\ar[rr]_-{A_{X}^*}^{\cong}\ar[u]_{\Nm^*}&&\Jac^0(X)\ar[u]^{{p}^*}&&}
\]
(where the maps no longer depend on the choice of base points). Recall that by definition $\widecheck{\Ll}$ is the line bundle over $\Jac^0(X)$ which corresponds to the flat line bundle $\Ll$ over $X$ under autoduality, {\it i.e.} $\Ll\cong A_X^*\widecheck{\Ll}$. So we conclude that $\Theta(\Nm^*\widecheck{\Ll})\cong \Oo_{{p}^*\Ll\otimes K_C^{(n-1)/2}}=\Oo_{\hat\Ll}$.  

Hence, $\tau_*\dot\nu_\phi^*\Theta(\Nm^*\widecheck{\Ll})\cong\tau_*\Oo_{\dot\nu_\phi^{-1}(\hat{\Ll})}$.
Since $\tau$ is finite morphism, $\tau_*\Oo_{\dot\nu_\phi^{-1}(\hat{\Ll})}$ is a coherent sheaf, hence (cf. \cite[ex. 5.5, 5.6, p.124]{hartshorne:1977}) its support in $\overline{\Jac}^{\, \delta}(X_b)$ is closed and is the closure of the image by $\tau$ of the support of $\Oo_{\dot\nu_\phi^{-1}(\hat{\Ll})}$ which, by Lemma \ref{lemma:fiberHecke}, is $\overline{\hat\nu_\phi^{-1}(\hat{\Ll})}$. In addition, by the same lemma, the restriction of $\tau$ to $\dot\nu_\phi^{-1}(\hat{\Ll})$ is a closed embedding, thus we actually have  $\tau_*\Oo_{\dot\nu_\phi^{-1}(\hat{\Ll})}\cong\Oo_{\overline{\hat\nu_\phi^{-1}(\hat{\Ll})}}$.

Considering the transform just as a sheaf (thus ignoring the only degree where the complex is non-zero), we then have, by \eqref{eq:ThetabThetagammahyperhol}, 
\[\Theta_b(\check{\nu}_{\phi,*}\Nm^*\widecheck{\Ll})\otimes\tau_* \Oo_{\PMod}\cong \ol{\Pp}_b |_{ \{ \nu_{\phi, *}\Oo_{{C}} \} \times \overline{\hat\nu_\phi^{-1}(\hat{\Ll})}}.\]  
The sheaf $\tau_* \Oo_{\PMod}$ is supported in $\overline{\Jac}^{\, \delta}(X_b)$, thus 
\[
\supp(\Theta_b(\check{\nu}_{\phi,*}\Nm^*\widecheck{\Ll}))=\overline{\hat\nu_\phi^{-1}(\hat{\Ll})}
\]
as claimed. Note finally that $\overline{\hat\nu_\phi^{-1}(\hat{\Ll})}$ is indeed the spectral data of the intersection $\Hec^{p,\hat{\Ll}}_{\ni} \cap h_{X,n}^{-1}(b)$, by Theorem \ref{tm:spectdataHec} (since $\hat\Ll={p}^*(\Ll^{-1}\otimes K_X^{(n-1)/2})$, then $\Gamma(\hat\Ll)= \{\hat{\Ll} \}$ in \eqref{eq description of h_b restricted to Hek}; cf.\ Remark~\ref{rmk:Jdescends}), completing the proof of \eqref{item1}.
%
%

For the proof of \eqref{item2}, we have, again by Theorem~\ref{tm relation of FM},

\begin{align*}
\Theta_b \bigg(\bigoplus_{\gamma \in \Gamma}\check{\nu}_{\phi,*} \hat{\gamma}^{*}\widecheck{\Ff}\bigg)\otimes\tau_* \Oo_{\PMod} & \cong \bigoplus_{\gamma \in \Gamma}(\Theta_b (\check{\nu}_{\phi,*} \hat{\gamma}^{*}\widecheck{\Ff})\otimes\tau_* \Oo_{\PMod})
\\
& \cong \bigoplus_{\gamma \in \Gamma} \tau_* \hat\nu_\phi^*\Theta(\hat{\gamma}^{*}\widecheck{\Ff})\otimes \ol{\Pp}_b |_{ \{ \nu_{\phi, *}\Oo_{{C}} \} \times \ol{\Jac}} 
\end{align*}
and a similar argument to the one given above, proves
$\Theta(\hat{\gamma}^{*}\widecheck{\Ff})\cong\Oo_{\gamma^*\hat\Ff}$.
Thus
\[\Theta_b \bigg(\bigoplus_{\gamma \in \Gamma}\check{\nu}_{\phi,*} \hat{\gamma}^{*}\widecheck{\Ff}\bigg)\otimes\tau_* \Oo_{\PMod}\cong \bigoplus_{\gamma \in \Gamma}\tau_*\Oo_{\hat{\nu}_\phi^{-1}(\gamma^{*}\hat{\Ff})} \otimes \ol{\Pp}_b |_{ \{ \nu_{\phi, *}\Oo_{{C}_a} \} \times \ol{\Jac}}.\]
As in the preceding case, $\bigoplus_{\gamma \in \Gamma}\tau_*\Oo_{\hat{\nu}_\phi^{-1}(\gamma^{*}\hat{\Ff})}$ is supported in the (not necessarily disjoint) union $\bigcup_{\gamma^{*}\hat{\Ff} \in \Gamma(\Ff)}\overline{\hat{\nu}_\phi^{-1}(\gamma^{*}\hat{\Ff})}$, hence so is the right-hand side of the above isomorphism as $\tau_* \Oo_{\PMod}$ is supported on $\ol{\Jac}^{\, \delta}(X_b)$. We conclude that 
\[
\supp\bigg(\Theta_b \bigg(\bigoplus_{\gamma \in \Gamma}\check{\nu}_{\phi,*} \hat{\gamma}^{*}\widecheck{\Ff}\bigg)\bigg)=\bigcup_{\gamma^{*}\hat{\Ff} \in \Gamma(\Ff)}\overline{\hat{\nu}_\phi^{-1}(\gamma^{*}\hat{\Ff})}\]
which coincides with $\Hec^{p,\hat{\Ff}}_{\ni}\cap h_{X,n}^{-1}(b)$ by Theorem \ref{tm:spectdataHec}.
\end{proof}

Since for each $b\in B^p_{\ni}$, $\Theta_b$ is a derived equivalence by Theorem~\ref{thm:deriveequiv}, the previous theorem provides the following fiberwise duality statement as an immediate consequence. 

\begin{theorem} \label{tm duality}
Let $p : C \to X$ be a connected unramified $n$-cover. Consider the moduli space $\M_X(n,0)$ and let $\hat\Ll$ and $\hat\Ff$ as in Theorem~\ref{tm support of the dual brane in the Hitchin fiber}.
\begin{enumerate}[(i)]
\item The (fiberwise) dual of the rank $1$ Narasimhan-Ramanan $\BBB$-brane $\BBBB^p_\Ll$ (over $B^p_{\ni}$) is the $\BAA$-brane supported on $\Hec^{p,\hat{\Ll}}_{\ni}$, and whose flat bundle satisfies \eqref{eq:relation FM transf 1}.
\item The (fiberwise) dual of the rank $n$ Narasimhan-Ramanan $\BBB$-brane $\BBBB^p_{\Ff}$ (over $B^p_{\ni}$) is the $\BAA$-brane supported on $\Hec^{p,\hat{\Ff}}_{\ni}$, and whose flat bundle satisfies \eqref{eq:relation FM transf n} .
\end{enumerate}
\end{theorem}

\begin{remark}
In \cite{Borel}, duality is conjectured between the $\BBB$-brane $\mathbf{Car}(\Ll)$ (supported on the Cartan locus $\M_{(1,n)} \subset\M_{{C}}(n,nd)$) and the $\BAA$-brane $\mathbf{Uni}(\Ll)$ (supported on the unipotent locus $\Uni_C(\Ll)\subset\M_{{C}}(n,nd)$). Recall from Remark \ref{rk BBB^p pulls back to Car} that $\BBBB^p_\Ll$ pulls-back to $\mathbf{Car}(\Ll)$ under $\hat{p}$ while $\Hec^{p,\Ll}_{\ni}$ is sent to $\Uni_C(\Ll)$ under $\hat{p}$ as we have seen in Proposition~\ref{pr hat p}. Thus, Theorem \ref{tm duality} give us some indications that the general principles of this conjecture seem to hold true. 
\end{remark}

\section{Branes in the absence of a Hitchin section}\label{sec:non-max-ord}

In Section \ref{sc BBB branes} we worked under Assumption \ref{ass gamma of maximal order} to construct a family of $\BBB$-branes supported on $\M_X(n,d)^p$. In Section \ref{section:Heckebranes} we required Assumption \ref{ass d multiple of r}, which is weaker than Assumption \ref{ass gamma of maximal order}, to define the Lagrangian subvariety $\Hec^{p,\Jj}_\ni$ over the locus of Hitchin base $B^p_\ni$ of those spectral curves whose normalization lives in $B^\sm_{C,r}$. 

A straight-forward observation is that, when Assumption \ref{ass gamma of maximal order} fails, one can always define a $\BBB$-brane on $\M_X(n,d)^p$ by considering the trivial bundle on it.

Without Assumption \ref{ass d multiple of r} we face a problem for the construction of our Narasimhan--Ramanan dual $\BAA$-branes, namely the lack of a section for the Hitchin fibration $h_{C,r} : \M_C(r,d) \to B_{C,r}$. Instead, we pick the Lagrangian multisection of the Hitchin fibration given by a very stable bundle. With this multisection we define a Lagrangian subvariety which we study using the branes associated with parabolic subgroups from \cite[Section 6]{Borel}. 

Given a stable bundle $\Vv \to C$, one has that $(\Vv,\phi)$ is a stable Higgs bundle for every $\phi \in H^0(C,\End(\Vv)\otimes K_C)$. Then, we have a natural morphism
\[
\map{H^0(C,\End(\Vv) \otimes K_C)}{\M_C(r,d)}{\phi}{(\Vv,\phi),}{}
\]
which is an embedding as $\Vv$ is simple. We denote by $\Sigma_\Vv$ the image of this map. It is well-known that this provides a Lagrangian subvariety.

\begin{proposition}
For every stable bundle $\Vv$, $\Sigma_\Vv$ is Lagrangian.
\end{proposition}

\begin{proof}
Since $\Vv$ is stable it is simple, so one has that 
\[
\dim \Sigma_\Vv = \dim H^0(C,\End(\Vv)\otimes K_C) = \frac{1}{2} \dim \M_C(r,d).
\]
Since the vector bundle is fixed along $\Sigma_\Vv$, note also that the projection 
\[
T_{(\Vv, \phi)} \Sigma_\Vv \to H^1(C, \End(\Vv)) 
\]
is constantly zero, so $\Omega_{C,1}$ vanishes there.
\end{proof}

After Laumon \cite{laumon}, a vector bundle $\Vv \to C$ is {\it very stable} if it has no non-zero nilpotent Higgs field. It can be shown \cite{laumon} that a very stable bundle is stable (provided $g \geq 2$) and that the locus of very stable bundles is a dense open subset of the moduli space of vector bundles. The fourth author and C. Pauly  proved the following (see \cite[Theorem 1.1 and Corollary 1.2]{verystable}).

\begin{theorem}[\cite{verystable}] \label{tm very stable}
Let $\Vv$ be a stable bundle. Then, $\Vv$ is very stable if and only if the restriction $h_{C,r}|_{\Sigma_\Vv}$ of the Hitchin fibration $h_{C,r}$ to $\Sigma_\Vv$ is finite and surjective.  
\end{theorem}
One thus easily deduces the following.
\begin{corollary}\label{cor multisection from very stable}
Under the hypotheses of Theorem \ref{tm very stable}, $\Sigma_\Vv \to B_{C,r}$ is a Lagrangian multisection of $h_{C,r}$.
\end{corollary}

In view of Corollary \ref{cor multisection from very stable} we provide the following definition analogous to Definition \ref{def:Heckesubvar}. 

\begin{definition}\label{def:Heckesubvar for Vv}
For any very stable bundle $\Vv \to {C}$ of rank $r$ and degree $d - \rho + \delta$, define the subvariety $\Hec^{p,\Vv}_\ni$ of $\M_X(n,d)$ closed in $\M_X(n,d) \times_{B_{X,n}} B^p_\ni$ as
\[
\Hec^{p,\Vv}_\ni := \left( \left ( \hH^{(\delta - \rho)}_p \right )^{-1} (\Sigma_{\Vv})\right) \cap \left ( \M_X(n,rd')^p_\ni \times_{B^p_\ni} \Sing^{\, p}_\ni \right ).
\]
\end{definition}

As we did in Theorem \ref{tm:spectdataHec} for $\Hec^{p,\Jj}_\ni$, let us study the spectral datum for $\Hec^{p,\Vv}_\ni$. Let us denote by $\Sigma_{\Vv,a}$ the fiber $\Sigma_{\Vv} \to B_{C,r}$ over $a\in B_{C,r}$. Set also 
\[
\Sigma'_{\Vv,a} := S_{C,r}^{-1}(\Sigma_{\Vv,a}),
\]
note that, by definition of $S_{C,r}$ and $\Sigma_{\Vv,a}$, for every $\Jj \in \Sigma'_{\Vv,a} \subset \Jac^{\, d+\delta}(C_a)$ one has
\[
\eta_{a,*} \Jj \cong \Vv.
\]

\begin{proposition}\label{prop:spectdataHec very stable}
Let $b \in B^p_\ni$ and $a \in B^\ni_{C,r}$ such that $b = \zeta(a)$, then 
\begin{equation} \label{eq description of h_b restricted to Hek for very stable}
h_{X,n}^{-1}(b) \cap \Hec^{p,\Vv}_\ni = \bigcup_{\Jj \in \Sigma'_{\Vv,a}, \gamma^*\Jj \in \Gamma(\Jj)} S_{X,n} \left( \overline{\hat{\nu}_a^{-1}(\gamma^*\Jj)  }  \right ),
\end{equation}
where $\Gamma(\Jj)$ denotes the $\Gamma$-orbit of $\Jj$. Furthermore, 
\begin{equation} \label{eq dim Hitchin fiber cap Hecke very stable}
\dim \left ( h_{X,n}^{-1}(b) \cap \Hec^{p,\Vv}_\ni \right ) = \delta - \rho = n(n-r)(g-1)
\end{equation}
and 
\begin{equation} \label{eq Hec^Vv mid dimensional}
\dim \Hec^{p,\Vv}_\ni = n^2(g-1) + 1 = \frac{1}{2} \dim \M_X(n,d).
\end{equation} 
\end{proposition}

\begin{proof}  
After a trivial adaptation of the proof of Theorem \ref{tm:spectdataHec} one gets the description given in \eqref{eq description of h_b restricted to Hek for very stable}.
The rest of the proposition follows from this fact, \eqref{eq dim Hitchin fiber cap Hecke very stable} is clear after the description of the fibers of $\hat\nu_a$ given in \cite{EGA} and \eqref{Xb singular} of Theorem \ref{thm:normaliz}. The proof of \eqref{eq Hec^Vv mid dimensional} follows from \eqref{eq dim Hitchin fiber cap Hecke very stable} and \eqref{eq dim B^gamma}.
\end{proof}

In order to prove that the above manifold is isotropic we compare it with some complex Lagrangian submanifolds inside $\M_{{C}}(n,dm)$ defined in \cite[Section 6]{Borel}.

Let $\Vv \to C$ of rank $r$ and degree $d - \rho + \delta$, and choose an ordering $o \in \Ord(\Gamma)$ of the elements of $\Gamma$, $o = (\gamma_{o,1}, \ldots, \gamma_{o,m})$. For each $o$ and each $i = 1, \ldots, m$, define the vector bundle $\Vv_{o,i} = \gamma_{o,i}^*\Vv \otimes K_C^{-r(m-i-1)}$ be vector bundles, and consider the variety
\begin{equation}\label{eq Uni par}
\Uni^{(r,m)}_{{C}}\left (\Vv, o \right)=\left\{
(E,\varphi) \in \M_C(n,dm) \ 
\left|\ 
\begin{array}{l}
\exists \, \sigma\in H^0(X,E/\P_{(r,m)}):\\ 
\varphi\in H^0(X,E_\sigma(\mathfrak{p}_{(r,m)})\otimes K_X);\\
E_\sigma / \U_{(r,m)} \cong \Vv_{o,1} \oplus \cdots \oplus \Vv_{o,m}.
\end{array}
\right.
\right\},
\end{equation}
where we recall the notation introduced in Remark \ref{rk Levi and parabolic groups}. These subvarieties are studied in \cite[Section 6]{Borel} and it follows from \cite[Theorem 6.7]{Borel} that $\Uni^{(r,m)}_{{C}}(\Vv, o)$ is Lagrangian (in particular, it is isotropic).

We prove a relation between $\Hec^{p,\Vv}_\ni$ and $\Uni^{(r,m)}_{{C}}(\Vv, o)$ analogous to that described in Proposition \ref{pr hat p}.

\begin{proposition} \label{pr hat p for very stable}
Let $\hat{p}:\M_X(n,d)\to\M_{C}(n,md)$ be the pullback morphism. Consider the open subvariety $\Hec^{p,\Vv}_{\Jac}$ of $\Hec^{p,\Vv}_\ni$ defined as its intersection with the open subset $\Jac^{d + \delta}(X_b)$ of every Hitchin fiber $h_{X,n}^{-1}(b)$. Then $\Hec^{p,\Vv}_{\Jac}$ is mapped under the pullback map \eqref{eqpullback} to the union of $\Uni^{(r,m)}_{{C}}(\Vv, o)$ for all different $o \in \Ord(\Gamma)$, {\it i.e.}
\[
\hat{p}\left(\Hec^{p,\Vv}_{\Jac} \right )\subset \bigsqcup_{o \in \Ord(\Gamma)} \Uni^{(r,m)}_{{C}} (\Vv, o).
\]
\end{proposition}
\begin{proof}

We will check that the spectral datum of 
$p^\ast (E,\varphi)$ satisfies the conditions of the spectral datum of $\Uni_C^{(r,m)}(\Vv, o)$, which proves the statement.

Let $L\in\Jac^{d+\delta}(X_b)$ be the spectral datum of $(E,\varphi)\in\Hec^{p,\Vv}_\ni$. By Cartesianity of the square in \eqref{eq cartesian diagram spectral}, we know that $\tilde L=q_{\tilde{b}}^*L\in\Jac(X_{\tilde{b}})$ is the spectral datum for $(\tilde E,\tilde{\varphi}):={p}^*(E,\varphi)$. Also, Cartesianity of \eqref{normalizXtildeb-2} implies that
$$
\tilde{\nu}_{\tilde{b}}^*\tilde L=\tilde q_{a}^*\nu_a^*L.
$$
By \eqref{eq description of h_b restricted to Hek for very stable}, $\nu_a^*L$ equals $\gamma^*\Jj_i$, with $\Jj_i \in \Sigma'_{\Vv,a}$. So, choosing an ordering $o \in \Ord(\Gamma)$, we write
\begin{equation} \label{eq spectral data for p*Hec}
\tilde{\nu}_{\tilde{b}}^*\tilde L=\tilde q_{a}^*\gamma^*\Jj_i = ((\gamma_{o,1}\gamma)^*\Jj_i, \ldots, (\gamma_{o,m}\gamma)^*\Jj_i).
\end{equation}

Since the spectral data $\wt{L}$ satisfies \eqref{eq spectral data for p*Hec}, it is then a line over $C_{\tilde{b}} = \bigcup_{\gamma' \in \Gamma} C_{\gamma'(a)}$ whose restriction to any $C_{\gamma'(a)}$ is $(\gamma'\gamma)^*\Jj_i$. Then, using the order $o$ we can set $Z_j = \bigcup_{i=j}^m$ and $\widetilde{L}_j = \ker ( \widetilde{L} \to \widetilde{L}|_{Z_j} )$. This defines a filtration of $\Oo_{C_{\tilde{b}}}$
\[
0 \subset \wt{L}_1 \subset \wt{L}_2 \subset \cdots \subset \wt{L}_m = \wt{L}. 
\]
Taking the pushforward of this filtration to $C$ provides a filtration 
\[
0 \subset \widetilde{E}_1 \subset \widetilde{E}_1 \subset \cdots \subset \widetilde{E}_m = \wt{E}
\]
preserved by the Higgs field, and whose graded part $\gr(\wt{E})$ is precise $\bigoplus_i \Vv_i$ appearing in the definition of $\Uni^{(r,m)}_C(\Vv, o)$. Then, $(\wt{E}, \wt{\varphi})$ lies in some $\Uni^{(r,m)}_C(\Vv, o)$ and the proof is completed.

Alternatively, checking that the line bundles in \eqref{eq spectral data for p*Hec} satisfy \cite[Assumption 1]{Borel}, we may apply  \cite[Proposition 6.6]{Borel} to conclude.
\end{proof}
We can finally prove the main theorem of this section, whose proof mimics that of Theorem~\ref{tm lagrangian} and is thus omitted. 
\begin{theorem}
The manifold $\Hec^{\xi,\Vv}_\ni$ is Lagrangian.
\end{theorem}

\begin{proof}
Isotropicity is proved as in Theorem~\ref{tm lagrangian} making use in the proof of Proposition \ref{pr hat p for very stable} instead of Proposition \ref{pr hat p}, and recalling that the subvarieties $\Uni_C^{(r,m)}(\Vv,o)$ are isotropic. Then, the proof follows from \eqref{eq Hec^Vv mid dimensional}.
\end{proof}

The question of how to produce hyperholomorphic bundles $\M_C(r,d)$ dual to the $\BAA$-brane supported on a Lagrangian multisection $\Sigma_\Vv$ associated to a very stable bundle is being studied by Hausel and Hitchin \cite{HH}.

\end{document}